\documentclass[12pt, reqno]{amsart} 

\edef\restoreparindent{\parindent=\the\parindent\relax}
\usepackage{parskip}
\usepackage{tabularray}
\restoreparindent

\usepackage{floatrow}
\usepackage{amssymb, amsfonts, amsthm, amsmath, amscd}
\usepackage{graphicx}   
\usepackage[latin2]{inputenc}
\usepackage{t1enc}
\usepackage[mathscr]{eucal}
\usepackage{indentfirst}
\usepackage{pict2e}
\usepackage{epic}
\numberwithin{equation}{section}
\usepackage{stmaryrd}
\usepackage{scalerel,stackengine}
\usepackage[margin=2.0cm]{geometry}
\usepackage{mathtools}
\usepackage[colorlinks = true,
            linkcolor = blue,
            urlcolor  = blue,
            citecolor = blue,
            anchorcolor = blue]{hyperref}
\usepackage[titletoc]{appendix}
\usepackage[T1]{fontenc}
\usepackage[numbers,square]{natbib}
\usepackage{adjustbox}
\usepackage{dsfont}
\usepackage{times}
\usepackage{enumerate}
\usepackage{setspace}
\usepackage{leftidx}
\usepackage{multirow}

\usepackage{tikz-cd}
\usepackage{pgf,tikz}
\usetikzlibrary{math}
\usetikzlibrary{arrows,decorations.markings}
\usetikzlibrary{patterns,intersections}
\usetikzlibrary{calc}
\usetikzlibrary{decorations.pathreplacing}

\theoremstyle{plain}
\newtheorem{theorem}{Theorem}[section]
\newtheorem{lemma}[theorem]{Lemma}
\newtheorem{corollary}[theorem]{Corollary}
\newtheorem{proposition}[theorem]{Proposition}
\newtheorem{conjecture}[theorem]{Conjecture}

\theoremstyle{definition}
\newtheorem{definition}[theorem]{Definition}

\newtheorem{remark}[theorem]{Remark}

\DeclareMathOperator{\real}{Re}

\def\spp{z} 
\newcommand{\bulletsmall}{~\begin{picture}(-1,1)(-1,-3)\circle*{3}\end{picture}\!\!\quad} 
\def\ag{\mathcal{G}r} 
\def\CCp{\mathbb{C}_{\real>0}} 
\def\fm{f_P} 
\def\fmt{f_{P,t}} 
\def\gammam{\gamma_P} 
\def\gammamt{\gamma_{P,t}} 
\def\pim{\pi_{P}} 
\def\FFF{\mathcal{V}} 
\def\UU{\mathcal{U}} 
\def\DDD{\mathbb{D}} 
\def\co{\mathbb{X}_{\bullet}} 
\def\character{\chi} 
\def\liecharacter{\psi} 
\def\ol{\overline}
\def\mirrormap{\tau} 
\def\iii{\mathbf{i}} 
\def\ttp{\theta^+} 
\def\ttm{\theta^-} 
\def\XXp{\mathbb{X}^+} 
\def\bundlez{\mathcal{E}_0} 
\def\bundle{\mathcal{E}} 
\def\Xp{(X_P)_{>0}} 
\def\X{X_P} 
\def\Xti{X_{P,t,i}}
\def\Xt{X_{P,t}} 
\def\Xtp{(X_{P,t})_{>0}} 
\def\Xtpp{(X_{P,t'})_{>0}} 
\def\vol{\omega_{\X}} 
\def\ZL{Z(L_P)} 
\def\ZLpp{Z(L_P)_{\real >0}}
\def\ZLp{Z(L_P)_{>0}} 
\def\GPd{G^{\vee}/P^{\vee} } 
\def\GBd{G^{\vee}/B^{\vee} } 
\def\ia{\mathcal{I}_A} 
\def\sa{s_A} 
\def\ib{\mathcal{I}_B} 
\def\sb{s_B} 
\def\Amodule{ QH_{T^{\vee}\times\mathbb{G}_m}^{\bullet}(G^{\vee}/P^{\vee} )[q_i^{-1}|~i\in I\setminus I_P]} 
\def\Bmodule{ G_0(\X,\fm,\gammam,\pim)} 
\def\Mir{\Phi_{mir}} 
\def\EO{E} 
\DeclareMathOperator{\jac}{Jac}
\def\jacobi{\jac(\X,\fm,\gammam,\pim)} 
\def\UUU{\mathbb{D}_{\lambda}} 
\def\BF{B^-_{F}} 
\def\BFt{\widetilde{B}^-_{F}} 
\def\UF{U^-_{F}} 
\def\Ga{{\mathbb G}_a}
\def\SSS{\sym^{\bullet}(\mathfrak{t})}

\DeclareMathOperator{\critical}{Crit}
\def\crit{\critical(\fm/\ZL)}
\def\holosection{\mathcal{O}_{hol}}

\DeclareMathOperator{\pic}{Pic}
\DeclareMathOperator{\mcf}{mc} 
\DeclareMathOperator{\pr}{pr}
\DeclareMathOperator{\coker}{coker}
\DeclareMathOperator{\ev}{ev}
\DeclareMathOperator{\id}{id}
\DeclareMathOperator{\sym}{Sym}
\DeclareMathOperator{\spec}{Spec}
\DeclareMathOperator{\conv}{Conv}
\DeclareMathOperator{\End}{End}
\DeclareMathOperator{\lie}{Lie}
\DeclareMathOperator{\pd}{PD}
\DeclareMathOperator{\pt}{pt}
\DeclareMathOperator{\image}{Im}
\begin{document}
\title{Gamma conjecture I for flag varieties}

\author{Chi Hong Chow}
\address{\parbox{\linewidth}{
Max Planck Institute for Mathematics, 53111 Bonn, Germany\\ Department of Mathematics, Virginia Tech, Blacksburg, VA 24061, USA}}
\email{chow@vt.edu}

\begin{abstract} 
We prove Gamma conjecture I for all flag varieties by following a strategy proposed by Galkin and Iritani. The main new ingredient is showing that the totally positive part of the Rietsch mirror is mirror to the $\widehat{\Gamma}$-class and contains the critical point of the superpotential that corresponds to the Perron-Frobenius eigenvalue on the A-side.
\end{abstract}


\maketitle
\section{Introduction}\label{Section-introduction}
\subsection{Main results}\label{Subsection-mainresults}
\textit{Gamma conjecture I}, proposed by Galkin, Golyshev and Iritani \cite{GGI}, asserts that for every Fano manifold $F$, the limit of a normalization of Givental's $J$-function $J_F(s)$ is equal to a multiplicative characteristic class $\widehat{\Gamma}_F\in H^{\bullet}(F)$ defined in terms of the gamma function $\Gamma$. 

\begin{conjecture}\label{GammaconjI} (Gamma conjecture I \cite{GGI}) For every Fano manifold $F$, we have
\[  \lim_{\mathbb{R}_{>0}\ni ~s\to +\infty} \frac{J_F(s)}{\langle \pd [\pt], J_F(s)\rangle}= \widehat{\Gamma}_F\in H^{\bullet}(F),\]
where $\widehat{\Gamma}_F:=\prod_{i=1}^{\dim F}\Gamma(1+\delta_i)$ and $\delta_1,\ldots,\delta_{\dim F}$ are the Chern roots of the tangent bundle of $F$.
\end{conjecture}

Conjecture \ref{GammaconjI} is related to mirror symmetry as follows. Fano mirror symmetry, going back to Eguchi-Hori-Xiong \cite{EHX} and Givental \cite{GiventalICM}, is a duality between Fano manifolds $F$ and Landau-Ginzburg models $(F^{\vee},W)$. The \textit{$D$-module mirror conjecture} (e.g., \cite{KKP_JDG}) asserts that the quantum $D$-module of $F$ is isomorphic as $D$-modules to the Gauss-Manin system of $(F^{\vee},W)$. Conjecturally, the latter $D$-module is the de Rham realization of an exponential motive associated with $(F^{\vee},W)$, which is isomorphic to the Betti realization, induced by exponential period integrals. The latter has a natural integral structure, ideally generated by the Lefschetz thimbles. The \textit{mirror symmetric Gamma conjecture}, proposed independently by Iritani \cite{IZ} and by Katzarkov, Kontsevich and Pantev \cite{KKP}, asserts that under mirror symmetry, this integral structure corresponds to $\widehat{\Gamma}_F\cup ch(K_0(F))$ at the Frobenius limit\footnote{Here, the Chern character $ch$ is defined to be the sum of $e^{2\pi\sqrt{-1}\times\text{Chern root}}$.}. In this context, Gamma conjecture I, as explained by Galkin and Iritani \cite{GI}, reflects the following expectation:
\begin{conjecture}\label{MSGammaconj}
The element $\widehat{\Gamma}_F\cup ch([\mathcal{O}_F]) = \widehat{\Gamma}_F $ is mirror to the Lefschetz thimble associated with a particular critical point of $W$.
\end{conjecture}

More precisely, it asserts an identification between certain flat sections of the A/B-model $D$-modules, with the former defined in terms of $\widehat{\Gamma}_F$ and the fundamental solution $S(\hslash)$, and the latter defined in terms of oscillatory integrals over the given Lefschetz thimble $C$. When projected to a particular component, it takes the following form:
\[\int_F S(\hslash)\left(\hslash^{-\mu} \hslash^{c_1} \widehat{\Gamma}_{F}\right)= \hslash^{-\frac{1}{2}\dim F}\int_C e^{-W/\hslash}\omega_{F^{\vee}}\]
for some volume form $\omega_{F^{\vee}}$ on $F^{\vee}$.

In this paper, we focus on a class of Fano manifolds called \textit{flag varieties}, by which we mean complex projective varieties that are homogeneous under a simple algebraic group. We prove
\begin{theorem}\label{mainG}
Conjecture \ref{GammaconjI} holds for all flag varieties. 
\end{theorem}

\begin{theorem}\label{mainMS}
(Precise version given by Theorem \ref{outlineA=Bthm}) Conjecture \ref{MSGammaconj} holds for all flag varieties, with Lefschetz thimble replaced by the totally positive part of the Rietsch mirror. 
\end{theorem}

Conjecture \ref{GammaconjI} was proved for type $A$ Grassmannians by Galkin, Golyshev and Iritani \cite{GGI}; Fano threefolds of Picard rank one by Golyshev and Zagier \cite{GZ}; a class of toric Fano manifolds by Galkin and Iritani \cite{GI} (see the next paragraph); del Pezzo surfaces by Hu, Ke, Li and Yang \cite{HKLY}; and the twistor bundles of certain hyperbolic sixfolds\footnote{They are non-K\"ahler monotone symplectic manifolds. Nevertheless, Conjecture \ref{GammaconjI} can be formulated in the same way.} by Hugtenburg \cite{Kai}. 

For toric Fano manifolds, Conjecture \ref{MSGammaconj} was proved in full generality by Iritani \cite{IZ}, as a key step in their proof of the mirror symmetric Gamma conjecture. Furthermore, Galkin and Iritani \cite{GI} proved that Conjecture \ref{GammaconjI} is equivalent to a B-model analog of \textit{Property $\mathcal{O}$}, namely whether the critical point contained in the Lefschetz thimble from Conjecture \ref{MSGammaconj} (Iritani's theorem) is equal to the predicted one. While numerical experiments confirm a large number of examples satisfying this property, counterexamples have recently been found by Galkin, Hu, Iritani, Ke, Li and Su \cite{counterexample}. In light of this, they proposed two modifications of Conjecture \ref{GammaconjI} which become true for all toric Fano manifolds and are implied by the original conjecture, given (the A-model version of) \textit{Property $\mathcal{O}$} \cite[Definition 3.1.1]{GGI}.

\begin{remark}
Gamma conjecture I is one of the two conjectures proposed by Galkin, Golyshev and Iritani \cite{GGI} in their study of the relationship between the quantum connections of Fano manifolds and the associated Gamma classes. Gamma conjecture II, which is a refinement of Dubrovin's conjecture \cite{Dubrovin}, asserts that if the big quantum cup product of a Fano manifold $F$ is generically semi-simple\footnote{Sanda and Shamoto \cite{SS} later generalized the conjecture to include non semi-simple cases.}, then the fundamental solution basis of its quantum connection near the irregular singular point is given by $\{S(\hslash)(\hslash^{-\mu} \hslash^{c_1} \widehat{\Gamma}_F\cup ch([E_i]))\}_{i=1}^N$ for some full exceptional collection $\{E_1,\ldots, E_N\}$ of $D^b(F)$. The case of type A Grassmannians (resp. quadrics) was verified by these authors in the same paper (resp. Hu and Ke \cite{HK}). For toric Fano manifolds, a weaker version, where $D^b(F)$ is replaced by $K_0(F)$, was proved by Iritani \cite{IZ}, and the full version was recently proved by Fang and Zhou \cite{FZ}. 
\end{remark}
\begin{remark}
There are also versions of Gamma conjectures for Calabi-Yau manifolds, which assume the existence of mirrors. Hosono \cite{Hosono} conjectured an equality between the A/B-model central charges (hypergeometric series and periods) associated with mirror Calabi-Yau manifolds, where the Gamma class appears implicitly in the former expression. By taking the asymptotics of these central charges, Abouzaid, Ganatra, Iritani and Sheridan \cite{AGIS} formulated a Calabi-Yau version of the mirror symmetric Gamma conjecture in the same spirit as the Fano case. These conjectures could explain the appearance of the zeta values in the periods of Calabi-Yau manifolds--an observation made by many people in the development of mirror symmetry. See e.g., \cite{CY_ref_1, CY_ref_2, CY_ref_3, CY_ref_4, CY_ref_5, CY_ref_6, CY_ref_7, CY_ref_8, CY_ref_9} for related discussion.
\end{remark}
\begin{remark}
In a series of works \cite{TV1, TV2, TV4}, Tarasov and Varchenko constructed, for arbitrary type A flag variety, a basis of simultaneous solutions to the equivariant quantum differential equations and the limits of the $q$KZ equations, and proved that their asymptotics near the regular singular point are the products of $\widehat{\Gamma}$-class and the equivariant Chern characters of certain vector bundles. (Compare \cite[Section 3.5]{Iritani_shift}.) They also proved that these solutions can be expressed as integrals over what they called ``hypergeometric Landau-Ginzburg mirror symmetry model'', which is different from the Rietsch mirror. This may offer new insight into Galkin-Golyshev-Iritani's Gamma conjectures through analyzing the asymptotic behavior of these integrals near the irregular singular point. See \cite{CV, TV3} for a discussion on the case of projective spaces.
\end{remark}

\subsection{Outline of proof}
In what follows, we denote a flag variety by $\GPd$. Theorem \ref{mainMS} will be proved within the proof of Theorem \ref{mainG}, which we now sketch. Define
\[ \EO^{q=1}:=\max\{ |\lambda||~\lambda\text{ is an eigenvalue of }c_1(\GPd)\star_{q=1} -\}\]
and 
\[\quad\qquad
\mathcal{A}_{\GPd}:=  \left\{ s:\mathbb{R}_{>0}\rightarrow H^{\bullet}(\GPd)\left|
\begin{array}{l}
\nabla_{\partial_{\hslash}}s=0~\text{ and }\\ [.7em]
\exists~ m\in\mathbb{Z}~,~ \left|\left|e^{\frac{\EO^{q=1}}{\hslash}} s(\hslash)\right|\right|\overset{\hslash \to 0}{=\joinrel=}O(\hslash^{m})
\end{array}
\right.\right\},\]
where $\nabla_{\partial_{\hslash}}$ is the quantum connection of $\GPd$ in the $\hslash$-direction. By a result of Galkin, Golyshev and Iritani \cite{GGI}, Theorem \ref{mainG} follows from
\begin{proposition}\label{outlineO} $\EO^{q=1}$ is an eigenvalue of $c_1(\GPd)\star_{q=1}-$ with multiplicity one.
\end{proposition}
\begin{proposition}\label{outlineGGIthm}
(= Proposition \ref{propositionb}) $\mathcal{A}_{\GPd}$ contains $S(\hslash)\left(\hslash^{-\mu} \hslash^{c_1} \widehat{\Gamma}_{\GPd}\right)$, where $S(\hslash)\left(\hslash^{-\mu} \hslash^{c_1} -\right)$ is the fundamental solution of $\nabla_{\partial_{\hslash}}$ associated with the regular singular point $\hslash = \infty$.
\end{proposition}
Proposition \ref{outlineO} has been proved by Cheong and Li \cite{CheongLi}, based on a result of Rietsch \cite{RietschJAMS}. We prove Proposition \ref{outlineGGIthm} by adopting a general strategy of Galkin and Iritani \cite{GI}, which they used to handle the case of toric Fano manifolds. In fact, this is the point to which we referred in the earlier discussion of the relationship between Gamma conjecture I and mirror symmetry. It consists of the following steps:
\begin{enumerate}
\item Construct a Landau-Ginzburg (LG) model $(F^{\vee},W)$ and a middle-dimensional possibly non-compact cycle $C\subset F^{\vee}$.

\item Prove the $D$-module mirror conjecture for the pair $F$ and $(F^{\vee},W)$, and Conjecture \ref{MSGammaconj} for the cycle $C$.

\item Prove that $s(\hslash):=S(\hslash)\left(\hslash^{-\mu} \hslash^{c_1} \widehat{\Gamma}_{F}\right)$ satisfies the desired asymptotic growth by estimating the oscillatory integrals from Conjecture \ref{MSGammaconj} using the stationary phase approximation. 
\end{enumerate} 

\bigskip
\noindent\underline{\textit{Step (1)}} $~$Rietsch \cite{Rietsch} constructed a mirror of $\GPd$ consisting of the following data
\begin{itemize}
\item a smooth affine variety $\X$;

\item a regular function $\fm\in\mathcal{O}(\X)$;

\item a smooth morphism $\pim:\X\rightarrow\ZL$ onto a subtorus $\ZL$ of $T$;

\item a morphism $\gammam:\X\rightarrow T$; and

\item a fiberwise volume form $\vol\in\Omega^{top}(\X/\ZL)$ relative to $\pim$.
\end{itemize}

We think of $(\X,\fm)$ as a family of LG models parametrized by $\pim$, and our desired $(F^{\vee},W)$ is the fiber over $t=1$. The additional data $\gammam$ and $\vol$ are used in Step (2). To construct the cycle $C$, we realize, following \cite{LamTemplier}, the Rietsch mirror as the \textit{parabolic geometric crystal} introduced by Berenstein and Kazhdan \cite{BKI, BK}, and take $C$ to be the fiber of the \textit{totally positive part} $\Xp$ of $\X$\footnote{The subset $\Xp$ is closely related to the canonical positive structure on $\X$, which Berenstein and Kazhdan used to construct \textit{Kashiwara's crystal} \cite{Kashiwara} via tropicalization.}.

\bigskip
\noindent\underline{\textit{Step (2)}} $~$ The $D$-module mirror conjecture has been proved in our previous work:
\begin{theorem}\label{outlinemirrorthm} (\cite{Chow3}; recalled in more details in Theorem \ref{mirrorthm}) There exists an isomorphism $\Mir$ between the equivariant $A$-model $D$-module associated with $\GPd$ and the equivariant $B$-model $D$-module associated with the Rietsch mirror.
\end{theorem}

Roughly speaking, these $D$-modules are families of vector bundles parametrized by $\hslash$ and the $T^{\vee}$-equivariant parameters $h$, which are equipped with flat connections given by the quantum connection (in the $q$-direction) for $\GPd$ and the Gauss-Manin connection (in the $t$-direction) for the Rietsch mirror. The isomorphism $\Mir$ is accompanied by an isomorphism $\mirrormap$ that identifies $q$ with $t$.

We prove the following form of Conjecture \ref{MSGammaconj}:
\begin{theorem}\label{outlineA=Bthm} (Precise version of Theorem \ref{mainMS}) For given $\hslash$, $h$ and $t$, we have
\begin{equation}\label{outlineA=Bthmeq1}
S(\hslash,h,\mirrormap(t))\left(\hslash^{-\mu} \hslash^{c_1} \widehat{\Gamma}_{\GPd}\right)=\hslash^{-\frac{\ell}{2}}\sum_{v\in W^P} \left(\int_{\Xtp} e^{-\fmt/\hslash} \gammamt^{h/\hslash} \Mir^{-1}(\sigma^v)_{(-\hslash,h,t)}\right) \sigma_v.
\end{equation}
\end{theorem}

Our proof begins by adapting Iritani's argument in \cite{IZ} for the toric case: we apply the fact that both sides of \eqref{outlineA=Bthmeq1} are flat sections of the same $D$-module to reduce to verifying that their leading order terms are equal for generic $h$. By the localization formula, these leading order terms are in bijective correspondence with the set of $T^{\vee}$-fixed points of $\GPd$. We handle the terms corresponding to one of these fixed points by direct computation. To handle the others, observe that for the LHS, the $G^{\vee}$-action on $\GPd$ relates these terms to the one we have handled. Thus it suffices to establish similar relations for the RHS, which we achieve by making use of the rational Weyl group action on $\X$ constructed by Berenstein and Kazhdan.

\bigskip
\noindent\underline{\textit{Step (3)}} $~$By Theorem \ref{outlineA=Bthm}, the result follows immediately from the stationary phase approximation, provided we can prove that for $t=1$, the restriction $\fmt|_{\Xtp}$ has a unique critical point, which is non-degenerate and whose critical value is equal to $\EO^{q=1}$\footnote{That the critical value is equal to $\EO^{q=1}$ is exactly what we meant by the ``$B$-model analog of Property $\mathcal{O}$'' in our discussion on the case of toric Fano manifolds in Section \ref{Subsection-mainresults}.}. By the fact that after a reparametrization $\Xtp\simeq \mathbb{R}_{>0}^{\ell}$, $\fmt|_{\Xtp}$ is equal to the sum of coordinates plus a Laurent polynomial with positive coefficients, it suffices to find a critical point with critical value equal to $\EO^{q=1}$. The existence relies on a description of $\Mir$ and a result of Lam and Rietsch \cite{LamRietsch}, which says that the spectrum of the composition of \textit{Peterson-Lam-Shimozono's homomorphism} \cite{LamShimozono} and \textit{Yun-Zhu's isomorphism} \cite{YunZhu} takes the \textit{Schubert positive point} constructed in \cite{RietschJAMS} to a \textit{totally non-negative point} in the sense of Lusztig \cite{Lusztig_positivity}. To prove that the critical value is equal to $\EO^{q=1}$, we apply the \textit{first Chern class theorem}, which says that $\Mir$ maps $[\fm\vol ]$ to $c_1(\GPd)$, and a result of Cheong and Li \cite{CheongLi}, which says that the evaluation of $c_1(\GPd)$ at the above Schubert positive point is equal to $\EO^{q=1}$. 

\begin{remark} There has been substantial work on the study of oscillatory integrals over the mirrors of complete flag varieties $G^{\vee}/B^{\vee}$. Generalizing Givental's work \cite{Givental} on the type A case, Rietsch \cite{RietschIntegral} proved that for arbitrary type, these integrals, whenever convergent, are solutions to the \textit{(open) quantum Toda lattices} \cite{OpenToda}. Combined with a result of Kim \cite{Kim}, this proves the original mirror conjecture \cite{EHX, GiventalICM} (for $F=G^{\vee}/B^{\vee}$) that the \textit{quantum differential equations} of a Fano manifold $F$ are satisfied by the oscillatory integrals over its mirror. Rietsch's result was later generalized by Chhaibi \cite{Chhaibi} and Lam \cite{Lam} to the case where equivariant perturbation is present. Related results were also obtained by Gerasimov, Kharchev, Lebedev and Oblezin \cite{GKLO}, and Gerasimov, Lebedev and Oblezin \cite{GLO1, GLO2}. 
\end{remark}
\begin{remark} 
Working in a different context, Chhaibi \cite[Theorem 5.4.1]{Chhaibi} (see also \cite[Theorem 3.8]{Chhaibi_2}) proved that for the case $P=B$, the equivariant oscillatory integral $\int_{\Xtp} e^{-\fmt/\hslash} \gammamt^{h/\hslash} \vol$ is $W$-invariant and that one of its asymptotics involves a product of gamma factors. These properties are used in Step (2). Their extension to the general case, which is established in the present paper, requires proving the convergence of the integral (see Remark \ref{remark_convergence} below) and the $W$-invariance (up to sign) of the fiberwise volume form $\vol$ (Lemma \ref{Wpreserveomega}), whose existing proofs work only in this special case. 
\end{remark}
\begin{remark}\label{remark_convergence} The convergence of the equivariant oscillatory integrals from the RHS of \eqref{outlineA=Bthmeq1} is not immediate, and it follows from the existence of a critical point of $\fmt|_{\Xtp}$ established in Step (3). When $P=B$, Rietsch \cite{RietschIntegral} verified it by analyzing the superpotential directly.
\end{remark}

\subsection{Organization of paper} $~$

\begin{minipage}[t]{.15\textwidth}
Section \ref{Section-notation}.
\end{minipage}\begin{minipage}[t]{.82\textwidth} 
The A-model and B-model data are formulated in terms of a pair of simple groups that are Langlands dual to each other. Our convention is that the group for the A-model (resp. B-model) is denoted by $G^{\vee}$ (resp. $G$). This choice is made to maintain compatibility with the work \cite{BK}. In Section \ref{Subsection-Bmodelnotation}, we establish notation for $G$. In Section \ref{Subsection-Amodelnotation}, we establish notation for $G^{\vee}$. 
\end{minipage}

\vspace{.2cm}
\begin{minipage}[t]{.15\textwidth}
Section \ref{Section-amodel}.
\end{minipage}\begin{minipage}[t]{.82\textwidth} 
In Section \ref{Subsection-Amodelflagvariety} and Section \ref{Subsection-A-modelconnection}, we recall the A-model data associated with flag varieties including the quantum cohomology, quantum connection and fundamental solution. In Section \ref{Subsection-Schubert-positive-point}, we recall a result of Rietsch about the existence of Schubert positive points and a result of Cheong and Li about the Perron-Frobenius property of these points with respect to the quantum multiplication by the first Chern class.
\end{minipage}

\vspace{.2cm}
\begin{minipage}[t]{.15\textwidth}
Section \ref{Section-bmodel}.
\end{minipage}\begin{minipage}[t]{.82\textwidth}
In Section \ref{Subsection-Rietschmirror}, we recall Lam-Templier's definition of the Rietsch mirror formulated in terms of Berenstein-Kazhdan's geometric crystal. In Section \ref{Subsection-B-modelconnection}, we recall the corresponding B-model data including the Brieskorn lattice, Gauss-Manin connection and Jacobi algebra. In Section \ref{Subsection-Toruscharts}, Section \ref{Subsection-Positivepart} and Section \ref{Subsection-Weylgroupaction}, we discuss the additional data associated with a geometric crystal, namely the torus charts, totally positive part and rational Weyl group action.
\end{minipage}

\vspace{.2cm}
\begin{minipage}[t]{.15\textwidth}
Section \ref{Section-mirror-thm}.
\end{minipage}\begin{minipage}[t]{.82\textwidth} 
In Section \ref{Subsection-mirror-thm-statement}, we recall a mirror theorem recently proved by the author, which states that the $D$-modules from Section \ref{Subsection-A-modelconnection} and Section \ref{Subsection-B-modelconnection} are isomorphic. In Section \ref{Subsection-first-chern-class-thm}, we derive the first Chern class theorem from the mirror theorem. In Section \ref{Subsection-descriptionofmirrorisom}, we recall a description of the semi-classical limit of the mirror isomorphism in terms of Yun-Zhu's isomorphism and Peterson-Lam-Shimozono's homomorphism. In Section \ref{Subsection-totally-positive-critical-point}, we apply this description and a result of Lam and Rietsch to prove that the mirror isomorphism takes the Schubert positive points from Section \ref{Subsection-Schubert-positive-point} to some fiberwise critical points of the restriction of the superpotential to the totally positive part of the Rietsch mirror.
\end{minipage}

\vspace{.2cm}
\begin{minipage}[t]{.15\textwidth}
Section \ref{Section-flatsections}.
\end{minipage}\begin{minipage}[t]{.82\textwidth}
In Section \ref{Subsection-Aside}, we study a flat section of the quantum $D$-module constructed using the fundamental solution and the $\widehat{\Gamma}$-class of the flag variety (LHS of \eqref{outlineA=Bthmeq1}). In Section \ref{Subsection-Bside}, we study a flat section of the same $D$-module defined in terms of the mirror isomorphism from Section \ref{Subsection-mirror-thm-statement} and oscillatory integrals over the totally positive part of the Rietsch mirror (RHS of \eqref{outlineA=Bthmeq1}). In Section \ref{Subsection-A=B}, we prove Theorem \ref{outlineA=Bthm} (= Theorem \ref{mainMS}), i.e., the above two flat sections are equal.
\end{minipage}

\vspace{.2cm}
\begin{minipage}[t]{.15\textwidth}
Section \ref{Section-final}.
\end{minipage}\begin{minipage}[t]{.82\textwidth}
We prove Theorem \ref{mainG}.
\end{minipage}

\vspace{.2cm}
\begin{minipage}[t]{.15\textwidth}
Appendix \ref{Subsection-frobenius-method}.
\end{minipage}\begin{minipage}[t]{.82\textwidth}
We recall some results on differential equations with regular singularities.
\end{minipage} 

\vspace{.2cm}
\begin{minipage}[t]{.15\textwidth}
Appendix \ref{Subsection-proofs-from-bmodel}.
\end{minipage}\begin{minipage}[t]{.82\textwidth}
We give proofs of unproved results stated in Section \ref{Section-bmodel}.
\end{minipage}

\vspace{.2cm}
\begin{minipage}[t]{.15\textwidth}
Appendix \ref{Subsection-proofs-from-flatsections}.\end{minipage}\begin{minipage}[t]{.82\textwidth}
We give proofs of unproved results stated in Section \ref{Section-flatsections}.
\end{minipage}

\vspace{.2cm}
\begin{minipage}[t]{.15\textwidth}
Appendix \ref{Subsection-LamRietschexposition}.  
\end{minipage}\begin{minipage}[t]{.82\textwidth}
We give an exposition of a result of Lam and Rietsch used in Section \ref{Subsection-totally-positive-critical-point}.
\end{minipage}

\subsection*{Acknowledgements}
I would like to thank Thomas Lam and Leonardo Mihalcea for their interest and valuable discussions. I am grateful to the Max Planck Institute for Mathematics in Bonn for its financial support and for providing an excellent working environment.
\section{Notation}\label{Section-notation}
\subsection{Notation for $G$ (B-model)}\label{Subsection-Bmodelnotation}
Fix a simple complex algebraic group $G$ and a maximal torus $T$. We assume $G$ is of adjoint type, i.e., the center $Z(G)$ of $G$ is trivial. Denote by $R$ the set of roots of $(G,T)$, and by $\alpha$ and $\alpha^{\vee}$ the roots and coroots respectively. Fix a fundamental system $\{\alpha_1,\ldots,\alpha_r\}$ for the root system. Denote by $R^+$ the set of positive roots, and by $B$ (resp. $B^-$) the corresponding (resp. opposite) Borel subgroup of $G$. Denote by $U$ and $U^-$ the unipotent radicals of $B$ and $B^-$ respectively. The Lie algebras of the algebraic groups we have introduced are denoted by the standard notations. We also denote by $\mathfrak{g}_{\alpha}$ the one-dimensional root space of $\mathfrak{g}$ associated with a root $\alpha$.

Fix a subset $I_P$ of $I:=\{1,\ldots,r\}$. For convenience, we exclude the case $I_P=I$ for which our main results hold trivially. Denote by $R_P$ (resp. $R_P^+$) the set of $\alpha\in R$ (resp. $\alpha\in R^+$) that are generated by $\alpha_i$ for $i\in I_P$. Denote by $P$ the parabolic subgroup of $G$ satisfying $\lie(P)=\lie(B)\oplus\bigoplus_{\alpha\in -R_P^+}\mathfrak{g}_{\alpha}$, by $L_P$ its Levi subgroup, and by $\ZL$ the center of $L_P$, which is also the kernel of the group homomorphism $T\rightarrow \mathbb{G}_m^{I_P}$ defined by $t\mapsto (\alpha_i(t))_{i\in I_P}$. Since $G$ is of adjoint type, the group homomorphism 
\begin{equation}\label{Bmodelnotationeq1}
(\alpha_i|_{\ZL})_{i\in I\setminus I_P}:\ZL\rightarrow \mathbb{G}_m^{I\setminus I_P}
\end{equation}
is an isomorphism.

Let $W:=N_G(T)/T$ be the Weyl group of $(G,T)$, and $W_P$ the subgroup of $W$ generated by the simple reflections $s_i$ for $i\in I_P$. Denote by $w_0$ the longest element of $W$, and by $w_0^P$ the longest element of $W_P$. Define $w_P:=w_0^Pw_0$. Throughout this paper, we use $\ell$ to denote $\ell(w_P)$, the length of $w_P$, which is equal to the size of $R^+\setminus R^+_P$, and also the dimension of $\GPd$, which will be introduced in Section \ref{Subsection-Amodelflagvariety}.

Let $i\in I$. Fix elements $e_i\in\mathfrak{g}_{\alpha_i}$ and $f_i\in\mathfrak{g}_{-\alpha_i}$ such that $[e_i,f_i]=\alpha^{\vee}_i\in\mathfrak{t}$. There exist unique group homomorphisms $x_i:\mathbb{G}_a\rightarrow U$ and $y_i:\mathbb{G}_a\rightarrow U^-$ satisfying $\lie(x_i)(1)=e_i$ and $\lie(y_i)(1)=f_i$. Then there exist unique group homomorphisms $\character_i:U\rightarrow\mathbb{G}_a$ and $\liecharacter_i:U^-\rightarrow\mathbb{G}_a$ satisfying $\character_i\circ x_j = \delta_{ij}\id_{\mathbb{G}_a}=  \liecharacter_i\circ y_j$ for all $j\in I$. Define $\character:=\sum_{i\in I}\character_i$.

For $i\in I$, define $\ol{s_i}:=x_i(-1)y_i(1)x_i(-1)\in G$. It is known that $\ol{s_i}$ lies in the normalizer $N_G(T)$ of $T$ in $G$ and represents the simple reflection $s_i$ in the Weyl group $W$. Moreover, this definition extends to all elements of $W$. More precisely, if $w\in W$ and $\iii=(i_1,\ldots,i_m)$ is a reduced decomposition of $w$, then $\ol{w}:=\ol{s_{i_1}}\cdots\ol{s_{i_m}}$ lies in $N_G(T)$, represents $w$ in $W$ and is independent of the choice of $\iii$.
\subsection{Notation for $G^{\vee}$ (A-model)}\label{Subsection-Amodelnotation}
In Section \ref{Subsection-Bmodelnotation}, we have fixed a simple complex algebraic group $G$. Denote by $G^{\vee}$ its Langlands dual group, and by $T^{\vee}$ the maximal torus dual to $T$. By definition, the roots (resp. coroots) of $(G^{\vee},T^{\vee})$ are the coroots (resp. roots) of $(G,T)$. Since $G$ is of adjoint type, $G^{\vee}$ is simply-connected. The subset $\{\alpha^{\vee}_1,\ldots,\alpha^{\vee}_r\}$ forms a fundamental system for the root system of $(G^{\vee},T^{\vee})$. Denote by $B^{\vee}$ the corresponding Borel subgroup of $G^{\vee}$, and by $B^{\vee}_-$ the opposite Borel subgroup. Denote by $\{\omega^{\vee}_1,\ldots,\omega^{\vee}_r\}$ the set of fundamental weights, i.e., the dual basis of the set $\{\alpha_1,\ldots,\alpha_r\}$ of the simple coroots associated with the above fundamental system. Denote by $Q$ the coroot lattice.

Recall we have also fixed a subset $I_P$ of $I$. Let $P^{\vee}$ be the parabolic subgroup of $G^{\vee}$ satisfying $\lie(P^{\vee})=\lie(B^{\vee})\oplus \bigoplus_{\alpha\in -R^+_P}\mathfrak{g}^{\vee}_{\alpha^{\vee}}$, where $\mathfrak{g}^{\vee}_{\alpha^{\vee}}$ is the one-dimensional root space of $\mathfrak{g}^{\vee}:=\lie(G^{\vee})$ associated with $\alpha^{\vee}$. Define $Q_P:=\sum_{i\in I_P}\mathbb{Z}\cdot \alpha_i\subseteq Q$.

It is known that the Weyl group $N_{G^{\vee}}(T^{\vee})/T^{\vee}$ of $(G^{\vee},T^{\vee})$ is canonically isomorphic to the Weyl group $W$ of $(G,T)$. Because of this, we denote the former group by the same notation. Recall $W_P$ is the subgroup of $W$ generated by the simple reflections $s_i$ for $i\in I_P$. Define $W^P$ to be the set of minimal length coset representatives of the quotient set $W/W_P$.


\section{A-model}\label{Section-amodel}
We will use the notations established in Section \ref{Subsection-Amodelnotation}.

\subsection{Flag varieties and their quantum cohomology}\label{Subsection-Amodelflagvariety}
By a \textit{flag variety} we mean the quotient $\GPd$. It is a smooth projective $G^{\vee}$-variety. Introduce an extra $\mathbb{G}_m$-action on $\GPd$ given by the trivial action. The role of this action will be apparent in Section \ref{Section-mirror-thm}. The $T^{\vee}\times\mathbb{G}_m$-fixed points of $\GPd$ are given by $vP^{\vee}$, $v\in W^P$. For $v\in W^P$, define
\[
\begin{array}{rcl} \\ [-1.0em]
\sigma_v &:= &\pd\left[ \ol{B^{\vee}_-vP^{\vee}/P^{\vee}}\right]\in H_{T^{\vee}\times\mathbb{G}_m}^{2\ell(v)}(\GPd)\\ [1.2em]
\sigma^v &:= &\pd\left[ \ol{B^{\vee}vP^{\vee}/P^{\vee}}\right]\in H_{T^{\vee}\times\mathbb{G}_m}^{\dim_{\mathbb{R}}(\GPd)-2\ell(v)}(\GPd) \\ [.6em]
\end{array}.
\]
It is known that $\{\sigma_v\}_{v\in W^P}$ and $\{\sigma^v\}_{v\in W^P}$ are $H_{T^{\vee}\times\mathbb{G}_m}^{\bullet}(\pt)$-bases of $H_{T^{\vee}\times\mathbb{G}_m}^{\bullet}(\GPd)$, which are dual to each other with respect to the pairing $\int_{\GPd}-\cup -$.

Let $\lambda\in (Q/Q_P)^*$. The one-dimensional $T^{\vee}$-module $\mathbb{C}_{-\lambda}$ of weight $-\lambda$ is naturally a $P^{\vee}$-module so we can define a line bundle $L_{\lambda}:=G^{\vee}\times^{P^{\vee}}\mathbb{C}_{-\lambda}$ on $\GPd$. It is known that $\{[L_{\omega^{\vee}_i}]\}_{i\in I\setminus I_P}$ is a $\mathbb{Z}$-basis of $\pic(\GPd)$ and generates the nef cone. Let $\{\beta_i\}_{i\in I\setminus I_P}\subset H_2(\GPd;\mathbb{Z})$ be its dual basis. Then $\{\beta_i\}_{i\in I\setminus I_P}$ generates the cone of effective curve classes of $\GPd$. Introduce the \textit{quantum parameters} $q_i$, $i\in I\setminus I_P$. 

\begin{definition}
Define the \textit{$T^{\vee}\times\mathbb{G}_m$-equivariant quantum cohomology of $\GPd$}
\[ QH_{T^{\vee}\times\mathbb{G}_m}^{\bullet}(\GPd):=H_{T^{\vee}\times\mathbb{G}_m}^{\bullet}(\GPd)\otimes \mathbb{C}[q_i|~ i\in I\setminus I_P] \]
and the \textit{quantum cup product} $\star$ by
\[ \sigma_u\star\sigma_v:=\sum_{w\in W^P}\sum_{(d_i)\in \mathbb{Z}_{\geqslant 0}^{I\setminus I_P}}\left(\prod_{i\in I\setminus I_P} q_i^{d_i}\right)\left(\int_{\ol{\mathcal{M}}_{0,3}(\GPd,\beta_{\mathbf{d}})} \ev_1^*\sigma_u\cup\ev_2^*\sigma_v \cup\ev_3^*\sigma^w\right)\sigma_w,\]
where 
\begin{itemize}
\item $\beta_{\mathbf{d}}:=\sum_{i\in I\setminus I_P} d_i\beta_i\in H_2(\GPd;\mathbb{Z})$;

\item $\ol{\mathcal{M}}_{0,3}(\GPd,\beta_{\mathbf{d}})$ is the moduli stack of genus zero stable maps to $\GPd$ of degree $\beta_{\mathbf{d}}$ with three marked points;

\item $\ev_1,\ev_2,\ev_3:\ol{\mathcal{M}}_{0,3}(\GPd,\beta_{\mathbf{d}})\rightarrow\GPd$ are the evaluation morphisms; and

\item the integral $\int_{\ol{\mathcal{M}}_{0,3}(\GPd,\beta_{\mathbf{d}})}$ is the $T^{\vee}\times\mathbb{G}_m$-equivariant integral. 
\end{itemize}
\end{definition}

It is known (see e.g., \cite{CoxKatz}) that $(QH_{T^{\vee}\times\mathbb{G}_m}^{\bullet}(\GPd),\star)$ is a graded commutative $H_{T^{\vee}\times\mathbb{G}_m}^{\bullet}(\pt)$-algebra, where the grading is defined by requiring each $q_i$ to have degree $2\langle c_1(\GPd),\beta_i\rangle = 2\sum_{\alpha\in R^+\setminus R^+_P}\alpha^{\vee}(\alpha_i)$.
\subsection{Quantum connection and fundamental solution}\label{Subsection-A-modelconnection}
\begin{definition}\label{vectorbundledef}
Define 
\[\bundle:= QH^{\bullet}_{T^{\vee}\times\mathbb{G}_m}(\GPd)[q_i^{-1}|~i\in I\setminus I_P]\]
regarded as a vector bundle on $\spec H^{\bullet}_{T^{\vee}\times\mathbb{G}_m}(\pt)[q_i^{\pm 1}|~i\in I\setminus I_P] \simeq \mathbb{A}^1_{\hslash}\times \mathfrak{t}^{\vee}\times \mathbb{G}_m^{I\setminus I_P}$.
\end{definition}
\begin{definition}\label{Aconnectiondef}
(Quantum connection) Define a family $\nabla^A$ of connections on the family 
\[ \left\{\bundle|_{\{\hslash\}\times \{h\}\times \mathbb{G}_m^{I\setminus I_P} } \right\}_{(\hslash,h)\in (\mathbb{A}^1\setminus 0)\times \mathfrak{t}^{\vee}}\]
of vector bundles on $\mathbb{G}_m^{I\setminus I_P}$ by
\[
\nabla^A_{\partial_{q_i}} := \frac{\partial}{\partial q_i} +\frac{1}{\hslash q_i}(c_1^{T^{\vee}\times\mathbb{G}_m}(L_{\omega_i^{\vee}})\star -)\qquad i\in I\setminus I_P,
\]
where $L_{\omega_i^{\vee}}$ is defined in Section \ref{Subsection-Amodelflagvariety}, and its $T^{\vee}\times\mathbb{G}_m$-linearization is the restriction of its unique $G^{\vee}\times\mathbb{G}_m$-linearization.
\end{definition} 
\begin{lemma}\label{flatconnectionprop}
For every $(\hslash,h)\in (\mathbb{A}^1\setminus 0)\times\mathfrak{t}^{\vee}$, $\nabla^A$ is a flat connection on $\bundle|_{\{\hslash\}\times \{h\}\times \mathbb{G}_m^{I\setminus I_P} } $.
\end{lemma}
\begin{proof}
This is well-known. See e.g., \cite{Manin}.
\end{proof}
\begin{definition} \label{fundamentalsolutiondef} 
(Fundamental solution) Let $x\in H_{T^{\vee}\times\mathbb{G}_m}^{\bullet}(\GPd)$. Define 
\begin{align}\label{fundamentalsolutiondefeq1}
 & S(\hslash,h,q)(x) \nonumber \\
 := ~& e^{-H(q)/\hslash} x - \sum_{\substack{ v\in W^P\\
 (d_i)\in\mathbb{Z}^{I\setminus I_P}_{\geqslant 0}\setminus \{\mathbf{0}\}}} \left(\prod_{i\in I\setminus I_P} q_i^{d_i} \right)\left(\int_{\overline{\mathcal{M}}_{0,2}(\GPd,\beta_{\mathbf{d}})} \frac{\ev_1^*\left(e^{-H(q)/\hslash}  x\right) }{\hslash +\psi_1}\cup \ev_2^*\sigma^v\right) \sigma_v,
\end{align}
where $H(q):=\sum_{i\in I\setminus I_P} (\log q_i)c_1^{T^{\vee}\times\mathbb{G}_m}(L_{\omega_i^{\vee}})$ and $\psi_1$ is the $\psi$-class associated with the first marked point.
\end{definition}
\begin{remark}\label{fundamentalsolutionremark1}
A priori, each component of $S(\hslash,h,q)(x)$ (say with respect to the Schubert basis $\{\sigma_v\}_{v\in W^P}$) is a formal power series in $\hslash^{-1}$, $q_i$, $\log q_i$ and the equivariant parameters. But by Lemma \ref{fundamentalsolutionconverges} below, it is in fact a (multi-valued) holomorphic function on an open subset.
\end{remark}

\begin{lemma} \label{fundamentalsolutionlemma} 
For every $i\in I\setminus I_P$ and $x\in H_{T^{\vee}\times \mathbb{G}_m}^{\bullet}(\GPd)$, we have
\[ \nabla^A_{\partial_{q_i}} \left(S(\hslash,h,q)(x)\right) =0 \]
as a formal power series.
\end{lemma}
\begin{proof}
This is well-known. See e.g., \cite[Chapter 10]{CoxKatz}.
\end{proof}
\begin{lemma} \label{fundamentalsolutionconverges} 
There exists an open neighbourhood $U$ of the origin $0\in\mathfrak{t}^{\vee}$ such that for every $x\in H_{T^{\vee}\times\mathbb{G}_m}^{\bullet}(\GPd)$, the formal section $(\hslash,h,\widetilde{q})\mapsto S(\hslash,h,\exp(\widetilde{q}))(x)$ of the pull-back of $\bundle$ by the covering map 
\[(\mathbb{A}^1\setminus 0)\times \mathfrak{t}^{\vee}\times \lie(\mathbb{G}_m^{I\setminus I_P}) \rightarrow (\mathbb{A}^1\setminus 0)\times \mathfrak{t}^{\vee}\times \mathbb{G}_m^{I\setminus I_P}\]
is holomorphic on the open subset
\[\{(\hslash,h)\in (\mathbb{A}^1\setminus 0)\times \mathfrak{t}^{\vee}|~h\in \hslash U\}\times \lie(\mathbb{G}_m^{I\setminus I_P}).\]
\end{lemma}
\begin{proof}
We apply a result from Appendix \ref{Subsection-frobenius-method} where we take $\FFF$ to be the restriction of the vector bundle $H_{T^{\vee}\times\mathbb{G}_m}^{\bullet}(\GPd)$ to the open subset $\mathcal{U}:=\{(\hslash,h)\in (\mathbb{A}^1\setminus 0)\times \mathfrak{t}^{\vee}|~h\in \hslash U\}$ ($U$ to be specified), $J$ to be $I\setminus I_P$, and the system \eqref{frobeniusmethodeq1} to be $\nabla^A_{q_i\partial_{q_i}}=0$. In particular, each $A_{j,\mathbf{0}}$ is equal to $-\frac{1}{\hslash} c_1^{T^{\vee}\times\mathbb{G}_m}(L_{\omega_j^{\vee}})\cup - $. Observe that $S(\hslash,h,q)$ does have the form \eqref{frobeniusmethodeq1.5} with $S_{\mathbf{0}}=\id$ and $S_{\nu}$ ($\nu\ne \mathbf{0}$) given by 
\[y\mapsto -\sum_{v\in W^P}\left(\int_{\overline{\mathcal{M}}_{0,2}(\GPd,\beta_{\mathbf{d}})} \frac{\ev_1^*y }{\hslash +\psi_1}\cup \ev_2^*\sigma^v\right)\sigma_v,\]
which is a priori a formal power series in $\hslash^{-1}$ and the equivariant parameters. By Lemma \ref{fundamentalsolutionlemma} and Lemma \ref{Frobeniusmethodconvergence}, it suffices to show that $U$ can be chosen such that $S_{\nu}$ is holomorphic on $\mathcal{U}$ for all $\nu$.

By the recurrence relation \eqref{Frobeniusmethodconvergenceeq2}, it suffices to show that there exists $U$ such that for every $\nu\ne\mathbf{0}$, the determinant of the linear map
\[ X\mapsto \langle \nu, e\rangle X + X\circ A_{\mathbf{0}} - A_{\mathbf{0}}\circ X\]
does not vanish on $\mathcal{U}$. By linear algebra, every eigenvalue of this linear map is equal to the difference $\lambda_1-\lambda_2$ for some eigenvalues $\lambda_1$ and $\lambda_2$ of $\langle \nu, e\rangle\id + A_{\mathbf{0}}$ and $A_{\mathbf{0}}$ respectively. It follows that, by the localization formula, it is of the form $\langle \nu, e\rangle+\frac{\varphi}{\hslash}$ for some linear form $\varphi$ on $\mathfrak{t}^{\vee}$. Note that there are only finitely many possibilities for $\varphi$. Since $\langle \nu, e\rangle\geqslant 1$, we can indeed find $U$ such that $\langle \nu, e\rangle+\frac{\varphi}{\hslash}$ does not vanish on $\mathcal{U}$ for all such $\varphi$. We are done. 
\end{proof}
\begin{remark}\label{fundamentalsolutionremark2}
The $J$-function $J_F(s)$ from Conjecture \ref{GammaconjI} is by definition the ``last row'' of the fundamental solution matrix for the dual quantum connection restricted to the anti-canonical direction. In our case, we have
\[ J_{\GPd}(s)= \sum_{v\in W^P}\langle S(-1,0,q(s))(\sigma^v),1\rangle \sigma_v,\]
where $q(s):=(q_i(s))_{i\in I\setminus I_P}$ with $q_i(s):=s^{\sum_{\alpha\in R^+\setminus R^+_P}\alpha^{\vee}(\alpha_i)}$. In fact, $J_{\GPd}(s)$ plays no role in the proof of Theorem \ref{mainG} because what we prove is an equivalent statement formulated without this function. See Section \ref{Section-final}.
\end{remark}
\subsection{Schubert positive points}\label{Subsection-Schubert-positive-point}
Let $q\in \mathbb{G}_m^{I\setminus I_P}$. Denote by $QH^{\bullet}(\GPd)_q$ the quantum cohomology with equivariant parameters specialized at $0$ and quantum parameters specialized at $q$, and by $\star_q$ the ring structure on $QH^{\bullet}(\GPd)_q$ induced by $\star$. 
\begin{definition}\label{conjectureOevaluedef}
Let $q\in\mathbb{G}_m^{I\setminus I_P}$. Define
\[ \EO^q:=\max\{ |\lambda||~\lambda\text{ is an eigenvalue of }c_1(\GPd)\star_q -\}.\]
\end{definition}
In what follows, we take the coefficient ring to be $\mathbb{R}$ so that $QH^{\bullet}(\GPd)_q$ is a $|W^P|$-dimensional $\mathbb{R}$-algebra.
\begin{proposition}\label{existSchubertpositive}
For every $q\in\mathbb{R}_{>0}^{I\setminus I_P}$, there exists an $\mathbb{R}$-point $\spp_q$ of $\spec QH^{\bullet}(\GPd)_q$ such that 
\begin{enumerate}
\item (Schubert positivity) $\sigma_v(\spp_q)>0$ for all $v\in W^P$; and

\item $c_1(\GPd)(\spp_q)= \EO^q$.
\end{enumerate}
\end{proposition}
\begin{proof}
The following proof is not due to us. See Remark \ref{existSchubertpositiveremark} below. 

Put $A:= QH^{\bullet}(\GPd)_q$. Let $\mathbf{c}:=(c_v)_{v\in W^P}\in\mathbb{R}_{>0}^{W^P}$. Define $a_{\mathbf{c}}:=\sum_{v\in W^P}c_v\sigma_v\in A$. Consider the operator $M_{\mathbf{c}}:=a_{\mathbf{c}}\star_q -$ on $A$. Since $q\in\mathbb{R}_{>0}^{I\setminus I_P}$ and $\star_q$ is enumerative with respect to the Schubert basis $\{\sigma_v\}_{v\in W^P}$, the matrix representing $M_{\mathbf{c}}$ with respect to this basis is non-negative. By \cite[Lemma 9.3]{RietschJAMS} (see also \cite[Lemma 9.4]{LamRietsch}), $M_{\mathbf{c}}$ is moreover indecomposable, i.e., if $V\subseteq A$ is a vector subspace that is preserved by $M_{\mathbf{c}}$ and spanned by a subset of $\{\sigma_{v}\}_{v\in W^P}$, then $V=\{0\}$ or $A$. (Strictly speaking, the author only considered the case where $\mathbf{c}=(1)_{v\in W^P}$ but their arguments obviously carry over the present situation.) By Perron-Frobenius theorem, $M_{\mathbf{c}}$ has an eigenvalue $E_{\mathbf{c}}\in\mathbb{R}_{>0}$ such that it has maximum modulus among all eigenvalues of $M_{\mathbf{c}}$, and the corresponding eigenspace $V_{\mathbf{c}}$ is spanned by a vector $v_{\mathbf{c}}\in\sum_{v\in W^P}\mathbb{R}_{>0}\cdot \sigma_v$. Since $A$ is a commutative ring, it follows that, for every $x\in A$, we have $x\star_q v_{\mathbf{c}}\in V_{\mathbf{c}}$, and hence it is equal to $\lambda_{\mathbf{c}}(x) v_{\mathbf{c}}$ for a unique $\lambda_{\mathbf{c}}(x)\in \mathbb{R}$. It is easy to see that this defines an $\mathbb{R}$-algebra homomorphism $\lambda_{\mathbf{c}}:A\rightarrow\mathbb{R}$.  

Now let $N$ be a positive integer. Consider the element $x_N:=\sum_{k=0}^Nc_1^{\star_q k}\in A$, where we have put $c_1:=c_1(\GPd)$ for simplicity. It is equal to $a_{\mathbf{c}}$ for some $\mathbf{c}=(c_v)_{v\in W^P}\in\mathbb{R}^{W^P}$. By the quantum Chevalley formula \cite[Theorem 10.1]{FW}, we have $c_v\geqslant 0$ for all $v\in W^P$, and these inequalities are all strict if $N$ is sufficiently large. Choose $N=N_0$ such that the latter condition holds, so that we have an $\mathbb{R}$-algebra homomorphism $\lambda_{\mathbf{c}_0}:A\rightarrow\mathbb{R}$, by the discussion in the previous paragraph ($\mathbf{c}_0\in\mathbb{R}_{>0}^{W^P}$ is the vector such that $a_{\mathbf{c}_0}=x_{N_0}$). Define $\spp_q$ to be the $\mathbb{R}$-point of $\spec A$ associated with $\lambda_{\mathbf{c}_0}$.

\textit{Verification of (1).} Let $v\in W^P$. By definition, we have
\[
\left\{
\begin{array}{rcl}
\sigma_v(\spp_q)&=&\lambda_{\mathbf{c}_0}(\sigma_v)\\ [.3em]
\sigma_v\star_q v_{\mathbf{c}_0}&=& \lambda_{\mathbf{c}_0}(\sigma_v)v_{\mathbf{c}_0}\\ [.3em]
v_{\mathbf{c}_0}&\in& \sum_{w\in W^P}\mathbb{R}_{>0}\cdot \sigma_w
\end{array}.
\right.
\]
Since $\star_q$ is enumerative with respect to $\{\sigma_w\}_{w\in W^P}$, $q\in\mathbb{R}_{>0}^{I\setminus I_P}$ and $v_{\mathbf{c}_0}\in\sum_{w\in W^P}\mathbb{R}_{>0}\cdot\sigma_w$, we have $\sigma_v\star_q v_{\mathbf{c}_0}\in \left(\sum_{w\in W^P}\mathbb{R}_{\geqslant 0}\cdot\sigma_w\right)\setminus\{0\}$, and hence $\sigma_v(\spp_q)=\lambda_{\mathbf{c}_0}(\sigma_v)>0$ as desired.

\textit{Verification of (2).} We have seen that the operator $x_{N_0}\star_q-$ is indecomposable with respect to $\{\sigma_v\}_{v\in W^P}$. This implies that the operator $c_1\star_q-$ is also indecomposable with respect to the same basis. By Perron-Frobenius theorem, $c_1\star_q-$ has an eigenvalue $E\in\mathbb{R}_{>0}$ with one-dimensional eigenspace $V$ and maximum modulus among all other eigenvalues. By definition, $E=\EO^q$. Then $E':=\sum_{k=0}^{N_0}(\EO^q)^k\in\mathbb{R}_{>0}$ is an eigenvalue of $x_{N_0}\star_q-$ with eigenspace $V$ and maximum modulus among all other eigenvalues, and hence we must have $E'=E_{\mathbf{c}_0}$ and $V=V_{\mathbf{c}_0}$. Therefore,
\[ c_1(\spp_q)=\lambda_{\mathbf{c}_0}(c_1)=\text{ eigenvalue of }c_1\star_q-|_{V_{\mathbf{c}_0}}=\text{ eigenvalue of }c_1\star_q-|_{V}=\EO^q.\]
\end{proof}

\begin{remark}\label{existSchubertpositiveremark} $~$
\begin{enumerate}[(i)]
\item Rietsch \cite[Section 9]{RietschJAMS} constructed $\spp_q$ and verified (1) in the way described in the above proof in order to prove a structural result about the totally non-negative part of the centralizer of a principal nilpotent element for type A. They only considered the vector $\mathbf{c}=(1)_{v\in W^P}$, which is sufficient for their need.

\item Lam and Rietsch \cite[Section 9]{LamRietsch} used the same arguments for the same purpose when they generalized Rietsch's result to arbitrary Lie group type.

\item Cheong and Li \cite[Proposition 4.2]{CheongLi} simplified Rietsch's arguments (namely the proof of the indecomposability of $M_{\mathbf{c}}$) and applied them to prove \textit{Conjecture $\mathcal{O}$} \cite[Conjecture 3.1.2]{GGI}. The introduction of the element $x_N$ and the verification of (2) are due to them. 
\end{enumerate}
\end{remark}
\begin{remark}\label{existSchubertpositiveremark1}
In fact, the $\mathbb{R}$-point $\spp_q$ is uniquely characterized by the Schubert positivity. See \cite[Section 9]{RietschJAMS} for more details.
\end{remark}
\section{B-model}\label{Section-bmodel}
We will use the notations established in Section \ref{Subsection-Bmodelnotation}.
\subsection{Rietsch mirror}\label{Subsection-Rietschmirror}

\begin{definition}(\cite[Section 1]{BK}) \label{geometriccrystaldef}
The \textit{parabolic geometric crystal} associated with $(G,P)$ is a quadruple $(\X,\fm,\pim,\gammam)$ consisting of 
\begin{enumerate}
\item a smooth affine variety $\X$;

\item a regular function $\fm\in \mathcal{O}(\X)$, called the \textit{decoration};

\item a morphism $\pim:\X\rightarrow\ZL$, called the \textit{highest weight map}; and

\item a morphism $\gammam:\X\rightarrow T$, called the \textit{weight map},
\end{enumerate}
where 
\[ \X:= B^-\cap U\ZL\ol{w_P} U\]
and 
\[ \fm(x):= \character(u_1)+\character(u_2),\quad \pim(x):=t, \quad \gammam(x):=t_0\]
for $x=u_0t_0=u_1t\ol{w_P}u_2\in\X$ with $t_0\in T$, $t\in\ZL$, $u_0\in U^-$ and $u_1,u_2\in U$.
\end{definition}
\begin{remark}\label{geometriccrystalremark2}
It is an exercise to check that $\fm$, $\pim$ and $\gammam$ are well-defined.
\end{remark}
\begin{remark}\label{geometriccrystalremark}
The definition from \cite{BK} contains additional data including regular functions $\{\varphi_i\}_{i\in I}$, $\{\varepsilon_i\}_{i\in I}$ and rational $\mathbb{G}_m$-actions $\{e_i\}_{i\in I}$ on $\X$. These data will also be used in this paper, but their definitions will be postponed until Section \ref{Subsection-Weylgroupaction}.
\end{remark}

\begin{definition}\label{volumeformdef}
(Fiberwise volume form \cite[Section 6.6]{LamTemplier}) Define a fiberwise volume form $\vol$ on $\X$ relative to $\ZL$ as follows. Consider a $\ZL$-morphism $\X\rightarrow \ZL\times G/B$ defined by $x\mapsto (\pim(x),x^{-1}w_0^PB)$. One can check that it is an isomorphism onto $\ZL\times \mathcal{R}_{w_0^P}^{w_0}$, where $\mathcal{R}_{w_0^P}^{w_0}:= (B^-w_0^PB/B)\cap (Bw_0B/B)$. The projection $G/B\rightarrow G/P$ induces an isomorphism of $\mathcal{R}_{w_0^P}^{w_0}$ onto its image, which we denote by $\UU$. By \cite[Lemma 5.4]{KLS}, the complement of $\UU$ in $G/P$ has pure codimension one, and the associated multiplicity-free divisor $D$ is anti-canonical. It follows that there exists a unique volume form $\omega_{\UU}$ on $\UU$, up to a non-zero scalar multiple, that has simple pole along every irreducible component of $D$. Define $\vol$ to be the pull-back of $\omega_{\UU}$ by the following composition:
\[ \X\xrightarrow{\sim}\ZL\times \mathcal{R}_{w_0^P}^{w_0}\xrightarrow{\sim}\ZL\times\UU\xrightarrow{\pr_2}\UU.\]
\end{definition}
\begin{remark}\label{volumeformremark}
The above fiberwise volume form $\vol$, which is defined up to a non-zero scalar multiple, will be rescaled in Definition \ref{volumeformrescale}.
\end{remark}
We now recall Lam-Templier's definition \cite{LamTemplier} of the Rietsch mirror \cite{Rietsch}.
\begin{definition}\label{Rietschmirrordef}
The \textit{Rietsch mirror} of $\GPd$ is a quintuple $(\X,\fm,\pim,\gammam,\vol)$ consisting of the parabolic geometric crystal $(\X,\fm,\pim,\gammam)$ associated with $(G,P)$ defined in Definition \ref{geometriccrystaldef} and the fiberwise volume form $\vol$ on $\X$ defined in Definition \ref{volumeformdef}.
\end{definition}
\begin{remark}\label{Rietschmirrorremark}
In the context of mirror symmetry, the decoration $\fm$ is called the \textit{superpotential}.
\end{remark}
\begin{definition}\label{Rietschmirrorgmactiondef}
($\mathbb{G}_m$-action) Following \cite[Section 6.21]{LamTemplier}, we define a $\mathbb{G}_m$-action on $\X$, $\mathbb{A}^1$, $\ZL$ and $T$ by
\[
\begin{array}{rcll}
c\cdot x&:=& \rho^{\vee}(c)x \rho^{\vee}(c)^{-1} &\qquad x\in\X\\ [.3em]
c\cdot a&:=& ca &\qquad a\in \mathbb{A}^1\\ [.3em]
c\cdot t &:=& (2\rho^{\vee}-2\rho_P^{\vee})(c)t &\qquad t\in \ZL\\ [.3em]
c\cdot t &:=& t&\qquad t\in T\\ [.3em]
\end{array} 
\]
for $c\in\mathbb{G}_m$, where $\rho^{\vee}:=\frac{1}{2}\sum_{\alpha\in R^+}\alpha^{\vee}$ and $\rho^{\vee}_P:=\frac{1}{2}\sum_{\alpha\in R^+_P}\alpha^{\vee}$.
\end{definition}
\begin{lemma}\label{Rietschmirrorgmactionlemma}
$\fm$, $\pim$, $\gammam$ are $\mathbb{G}_m$-equivariant and $\vol$ is $\mathbb{G}_m$-invariant.
\end{lemma}
\begin{proof}
This is \cite[Proposition 6.24 \& Lemma 6.26]{LamTemplier}.
\end{proof}
\subsection{Brieskorn lattice, Gauss-Manin connection and Jacobi algebra}\label{Subsection-B-modelconnection}
\begin{definition}\label{Brieskorndef}
(Brieskorn lattice) Define a $\SSS[\hslash] \otimes\mathcal{O}(\ZL)$-module
\begin{align*}
&~ \Bmodule \\
\qquad :=&~ \coker\left( \SSS[\hslash]\otimes \Omega^{top -1}(\X/\ZL)\xrightarrow{\partial}  \SSS[\hslash]\otimes \Omega^{top}(\X/\ZL)\right),
\end{align*}
where
\begin{itemize}
\item $\Omega^{i}(\X/\ZL)$ is the space of $i$-forms on $\X$ relative to $\ZL$; and 

\item $\partial$ is defined by 
\[ \partial(z\otimes\omega):= z\otimes\left(\hslash d\omega + d\fm\wedge \omega\right)-\sum_i zh_i\otimes (\gammam^*\langle h^i,\mcf_T\rangle)\wedge\omega,\]
where $\{h_i\}\subset \mathfrak{t}$ and $\{h^i\}\subset \mathfrak{t}^*$ are dual bases and $\mcf_T\in\Omega^1(T;\mathfrak{t})$ is the Maurer-Cartan form of $T$.
\end{itemize}
\end{definition}
\begin{definition}\label{Bconnectiondef}
(Gauss-Manin connection) For $i\in I\setminus I_P$, define
\[\nabla^B_{\partial_{t_i}} :\Bmodule[\hslash^{-1}] \rightarrow \Bmodule[\hslash^{-1}] \]
by
\begin{equation}\label{mirrorremarkeq1}
\nabla^B_{\partial_{t_i}}[z\otimes \omega] := \frac{1}{\hslash} \left[ z\otimes \left(\hslash \mathcal{L}_{\widetilde{\partial_{t_i}}}\omega + (\mathcal{L}_{\widetilde{\partial_{t_i}}} \fm)\omega\right) - \sum_i zh_i\otimes (\iota_{\widetilde{\partial_{t_i}}}\gammam^*\langle h^i,\mcf_T\rangle)\omega \right]
\end{equation}
for $[z\otimes \omega]\in \Bmodule[\hslash^{-1}]$, where $\partial_{t_i}$ is the vector field on $\ZL$ corresponding to the $i$-th coordinate vector field on $\mathbb{G}_m^{I\setminus I_P}$ under the isomorphism \eqref{Bmodelnotationeq1}, and $\widetilde{\partial_{t_i}}$ is a lift of $\partial_{t_i}$ with respect to $\pim$.
\end{definition}
\begin{definition}\label{Jacobidef} (Jacobi algebra) Define $\jacobi$ to be the coordinate ring of the scheme-theoretic zero locus of the relative 1-form
\[ \pr_{\X}^*d\fm - \langle \pr_{\mathfrak{t}^{\vee}}, (\gammam\circ \pr_{\X})^*\mcf_T \rangle\in \Omega^1(\X\times \mathfrak{t}^{\vee}/ \ZL \times \mathfrak{t}^{\vee}),\]
where $\pr_{\X}:\X \times \mathfrak{t}^{\vee}\rightarrow \X$ and $\pr_{\mathfrak{t}^{\vee}}:\X\times \mathfrak{t}^{\vee}\rightarrow \mathfrak{t}^{\vee}$ are the projections. Note that the fiber product $0\times_{\mathfrak{t}^{\vee}}\spec\jacobi$ is nothing but the fiberwise critical locus of $\fm$ relative to $\ZL$. We denote it by $\crit$.
\end{definition}
\begin{remark}
The operators $\partial_{t_i}\mapsto \hslash\nabla^B_{\partial_{t_i}}$ define a $D_{\hslash,\ZL}$-module structure on the Brieskorn lattice $\Bmodule$, and the resulting $D_{\hslash,\ZL}$-module is isomorphic to the zeroth cohomology of the \textit{weighted geometric crystal $D_{\hslash}$-module} $WGr_{(G,P)}^{1/\hslash}$ defined by Lam and Templier \cite[Section 11.10]{LamTemplier}.
\end{remark}
\begin{remark}\label{mirrorremark} 
By identifying $\Omega^{top}(\X/\ZL)$ with $\mathcal{O}(\X)$ using a fiberwise volume form on $\X$, we get 
\[ \Bmodule/\hslash\Bmodule \simeq \jacobi\]
as $\SSS\otimes \mathcal{O}(\ZL)$-modules. For our purpose, we will take the fiberwise volume form to be $\vol$ defined in Definition \ref{volumeformdef}.
\end{remark}

\subsection{Torus charts}\label{Subsection-Toruscharts} The main reference for this subsection is \cite[Section 3]{BK}.
\begin{definition}\label{Bruhatdef}
Define
\[ B^-_{w_P}:=B^-\cap U\ol{w_P}U\quad \text{ and }\quad  U^{w_P}:= U\cap B^-w_PB^-.\]
\end{definition}
\begin{lemma}\label{geometriccrystaltrivial}
The morphism
\[
\begin{array}{ccc}
\ZL\times B^-_{w_P}&\rightarrow & \X \\ [.5em]
(t,x) &\mapsto & tx
\end{array}
\]
is an isomorphism of $\ZL$-schemes.
\end{lemma}
\begin{proof}
Obvious.
\end{proof}

\begin{definition}\label{toruschartforBruhatdef}
Let $\iii=(i_1,\ldots,i_{\ell})$ be a reduced decomposition of $w_P$. 
\begin{enumerate}
\item Define 
\[ \ttm_{\iii}:\mathbb{G}_m^{\ell}\rightarrow B^-\]
by 
\[ \ttm_{\iii}(a_1,\ldots,a_{\ell}) := x_{-i_1}(a_1)\cdots x_{-i_{\ell}}(a_{\ell}),\]
where $x_{-i}(a):=y_i(a)\alpha_i^{\vee}(a^{-1})$.

\item Define 
\[ \ttp_{\iii}:\mathbb{G}_m^{\ell}\rightarrow U\]
by
\[ \ttp_{\iii}(a_1,\ldots,a_{\ell}) := x_{i_1}(a_1)\cdots x_{i_{\ell}}(a_{\ell}).\]
\end{enumerate}
\end{definition}
\begin{lemma}\label{toruschartforBruhatlemma}
Let $\iii=(i_1,\ldots,i_{\ell})$ be a reduced decomposition of $w_P$.
\begin{enumerate}
\item $\ttm_{\iii}$ is an open immersion into $B^-_{w_P}$.

\item $\ttp_{\iii}$ is an open immersion into $U^{w_P}$.
\end{enumerate}
\end{lemma}
\begin{proof}
This is a special case of \cite[Proposition 4.5]{BZ} by observing that $B^-_{w_P}$ and $U^{w_P}$ are equal to $L^{w_P,e}$ and $L^{e,w_P}$ from \textit{loc. cit.} respectively. 
\end{proof}
\begin{definition}\label{twistmapdef}
(Twist map \cite[Definition 4.6]{BZ}) Define a morphism 
\[ \eta^{w_P}:U^{w_P}\rightarrow B^-_{w_P}\]
as follows. Let $x\in U^{w_P}$. Then $x\in B^-w_PB^-$, and hence $x\ol{w_P^{-1}}\in B^-w_PB^-w_P^{-1}\subseteq B^-U$ so that we can write $x\ol{w_P^{-1}} = bu$ uniquely, with $b\in B^-$ and $u\in U$. But we also have $x\in U$, and hence $b=x\ol{w_P^{-1}}u^{-1}\in U\ol{w_P^{-1}} U$. It follows that $b\in B^-\cap U\ol{w_P^{-1}} U=B^-_{w_P^{-1}}$. We define 
\[ \eta^{w_P}(x) := \iota(b)\in \iota( B^-_{w_P^{-1}} ) = B^-_{w_P},\]
where $\iota:G\xrightarrow{\sim} G$ is the anti-automorphism characterized by 
\[ \iota(x_i(a))=x_i(a),\quad \iota(t)=t^{-1}\quad \text{and}\quad \iota(y_i(a))=y_i(a)\]
for $i\in I$, $a\in \mathbb{A}^1$ and $t\in T$.
\end{definition}
\begin{lemma}\label{twistmapisom}
$\eta^{w_P}$ is an isomorphism.
\end{lemma}
\begin{proof}
This is a special case of \cite[Theorem 4.7]{BZ} by observing that $\eta^{w_P}$ is equal to $\psi^{e,w_P}$ from \textit{loc. cit.} 
\end{proof}

\begin{definition}\label{toruschartforgeometriccrystaldef}
Let $\iii=(i_1,\ldots,i_{\ell})$ be a reduced decomposition of $w_P$. Define an open immersion
\[ \XXp_{\iii}:\ZL\times \mathbb{G}_m^{\ell}\hookrightarrow \X\]
by
\[ \XXp_{\iii}(t,a_1,\ldots,a_{\ell}) := t\cdot(\eta^{w_P}\circ \ttp_{\iii})(a_1,\ldots,a_{\ell}),\]
where $\ttp_{\iii}$ and $\eta^{w_P}$ are defined in Definition \ref{toruschartforBruhatdef}(2) and Definition \ref{twistmapdef} respectively.
\end{definition}

\begin{lemma}\label{fLusztig}
For every reduced decomposition $\iii=(i_1,\ldots,i_{\ell})$ of $w_P$, we have
\[ (\fm\circ\XXp_{\iii})(t,a_1,\ldots,a_{\ell}) = a_1+\cdots+a_{\ell}+\sum_{i\in I\setminus I_P}\alpha_i(t) P_i(\mathbf{a}),\]
where each $P_i(\mathbf{a})$ is a Laurent polynomial in $\mathbf{a}=(a_1,\ldots,a_{\ell})$ with positive coefficients.
\end{lemma}
\begin{proof}
This is \cite[Corollary 6.11]{LamTemplier}\footnote{Note that their twist map \cite[Lemma 6.2]{LamTemplier} is an isomorphism $U^{w_P^{-1}}\xrightarrow{\sim}B^-_{w_P}$, which is equal to the composition of ours and the isomorphism $U^{w_P^{-1}}\xrightarrow{\sim}U^{w_P}$ induced by $\iota$. It is easy to translate their result to obtain our Lemma \ref{fLusztig}.}.
\end{proof}
\begin{lemma}\label{gammaLusztig}
For every reduced decomposition $\iii=(i_1,\ldots,i_{\ell})$ of $w_P$, we have
\[ (\gammam\circ\XXp_{\iii})(t,a_1,\ldots,a_{\ell}) = t \cdot \prod_{k=1}^{\ell}\beta^{\vee}_k(a_k),\]
where $\beta^{\vee}_k:= - s_{i_1}\cdots s_{i_{k-1}}(\alpha^{\vee}_{i_k})\in \co(T)$ (the cocharacter lattice of $T$). 
\end{lemma}
\begin{proof}
This was proved in \cite[Theorem 4.1.20]{Chhaibi}. A similar result was also proved in \cite[Claim 7.12]{BK}. For the reader's convenience, we provide the details in Appendix \ref{Subsection-proofs-from-bmodel}.
\end{proof}
The following fact about the coroots $\beta^{\vee}_k$ from the last lemma will be useful.
\begin{lemma}\label{betadescription}
We have 
\[ \{\beta^{\vee}_1,\ldots,\beta^{\vee}_{\ell}\}=\{\alpha^{\vee}\}_{\alpha\in -(R^+\setminus R^+_P)}.  \]
\end{lemma}
\begin{proof}
It suffices to prove the following more general result: For every reduced decomposition $\iii=(i_1,\ldots,i_m)$ of an element $w\in W$, we have
\begin{equation}\label{betadescriptioneq1}
\{\beta^{\vee}_k:= - s_{i_1}\cdots s_{i_{k-1}}(\alpha^{\vee}_{i_k})\}_{k=1}^m = \{\alpha^{\vee}\}_{\alpha\in wR^+\cap (-R^+)}.
\end{equation} 
This is proved by induction on $\ell(w)$. At each inductive step, we write $w=w's_{i_m}$. Since $\ell(w)=\ell(w')+1$, we have $w'(\alpha_{i_m})\in R^+$. It follows that both the new sets from \eqref{betadescriptioneq1} are obtained from the old sets by adding the same element $-w'(\alpha_{i_m}^{\vee})$. We are done.
\end{proof}

Recall the fiberwise volume form $\vol$ defined in Definition \ref{volumeformdef}.

\begin{lemma}\label{omegaLusztig}
For every reduced decomposition $\mathbf{i}=(i_1,\ldots,i_{\ell})$ of $w_P$, we have
\[ (\XXp_{\iii})^*\vol = c \frac{da_1\wedge\cdots\wedge da_{\ell}}{a_1\cdots a_{\ell}}\]
for some non-zero scalar $c$.
\end{lemma}
\begin{proof}
See Appendix \ref{Subsection-proofs-from-bmodel}.
\end{proof}
\begin{definition}\label{volumeformrescale}
(Rescaling of $\vol$) Fix a reduced decomposition $\iii_0$ of $w_P$. We rescale $\vol$ to the unique volume form satisfying 
\[(\XXp_{\iii_0})^*\vol = \frac{da_1\wedge\cdots\wedge da_{\ell}}{a_1\cdots a_{\ell}}.\]
(Lemma \ref{omegaLusztig} guarantees that this is possible.)
\end{definition}

\subsection{Totally positive part}\label{Subsection-Positivepart} The main reference for this subsection is \cite[Section 3]{BK}.
\begin{definition}\label{nonnegativepartofsubgroupofG}$~$
\begin{enumerate}
\item Define $T_{>0}$ to be the submonoid of $T$ with unit generated by $\alpha_i^{\vee}(a)$ for $i\in I$ and $a\in\mathbb{R}_{>0}$. 

\item Define
\[ \ZLp:=\ZL\cap T_{>0}.\]

\item Define $G_{\geqslant 0}$ to be the submonoid of $G$ with unit generated by $\alpha_i^{\vee}(a)$, $x_i(a)$ and $y_i(a)$ for $i\in I$ and $a\in\mathbb{R}_{>0}$.

\end{enumerate}
\end{definition}
\begin{remark}\label{nonnegativepartofsubgroupofGrmk}
Since $G$ is of adjoint type, we have
\[ T_{>0}=\{t\in T|~\alpha_i(t)\in\mathbb{R}_{>0}\text{ for all }i\in I\}. \]
\end{remark}
\begin{definition}\label{positivepartofBruhat}
Define 
\[ (B^-_{w_P})_{>0} :=B^-_{w_P}\cap G_{\geqslant 0}\quad \text{ and }\quad (U^{w_P})_{>0} :=U^{w_P}\cap G_{\geqslant 0}, \]
where $B^-_{w_P}$ and $U^{w_P}$ are defined in Definition \ref{Bruhatdef}.
\end{definition}

\begin{lemma}\label{toruschartforBruhatpositive}
For every reduced decomposition $\iii$ of $w_P$, the open immersion $\ttm_{\iii}$ (resp. $\ttp_{\iii}$) defined in Definition \ref{toruschartforBruhatdef} maps $\mathbb{R}_{>0}^{\ell}$ onto $(B^-_{w_P})_{>0}$ (resp. $(U^{w_P})_{>0}$).
\end{lemma}
\begin{proof}
This is a special case of \cite[Proposition 4.5]{BZ} by observing that $(B^-_{w_P})_{>0}$ and $(U^{w_P})_{>0}$ are equal to $L^{w_P,e}_{>0}$ and $L^{e,w_P}_{>0}$ from \textit{loc. cit.} respectively. 
\end{proof}
\begin{lemma}\label{twistmappositive}
The twist map $\eta^{w_P}$ defined in Definition \ref{twistmapdef} maps $(U^{w_P})_{>0}$ onto $(B^-_{w_P})_{>0}$.
\end{lemma}
\begin{proof}
This is a special case of \cite[Theorem 4.7]{BZ} by observing that $\eta^{w_P}$ is equal to $\psi^{e,w_P}$ from \textit{loc. cit.}, and that $(B^-_{w_P})_{>0}$ and $(U^{w_P})_{>0}$ are equal to $L^{w_P,e}_{>0}$ and $L^{e,w_P}_{>0}$ from \textit{loc. cit.} respectively. 
\end{proof}
\begin{definition}\label{positivepartdef}
Define the \textit{totally positive part} $\Xp$ of $\X$ to be the image of $\ZLp\times(B^-_{w_P})_{>0}$ under the isomorphism from Lemma \ref{geometriccrystaltrivial}.
\end{definition}
\begin{lemma}\label{toruschartforgeometriccrystalpositive} 
For every reduced decomposition $\iii$ of $w_P$, we have 
\[\XXp_{\iii}(\ZLp\times\mathbb{R}_{>0}^{\ell})=\Xp,\]
where $\XXp_{\iii}$ is defined in Definition \ref{toruschartforgeometriccrystaldef}.
\end{lemma}
\begin{proof}
This follows from Lemma \ref{toruschartforBruhatpositive} and Lemma \ref{twistmappositive}.
\end{proof}

\begin{lemma}\label{gammapreservepositivepart}
We have $\gammam(\Xp)\subseteq T_{>0}$.
\end{lemma}
\begin{proof}
This follows from Lemma \ref{gammaLusztig}, Remark \ref{nonnegativepartofsubgroupofGrmk} and Lemma \ref{toruschartforgeometriccrystalpositive}. Alternatively, this follows from \cite[Lemma 2.3(b)]{Lusztig_positivity}.
\end{proof}
\begin{definition} \label{positivepartorientation}
(Orientation) Define an orientation on the fibers of $\Xp$ over $\ZLp$ to be the one induced by the standard orientation on $\mathbb{R}_{>0}^{\ell}$ via $\XXp_{\iii_0}$ (see Lemma \ref{toruschartforgeometriccrystalpositive}), where $\iii_0$ is the reduced decomposition of $w_P$ fixed in Definition \ref{volumeformrescale}. In other words, the fiberwise volume form $\vol$ (after rescaling) is an orientation form.
\end{definition}
\subsection{Rational Weyl group action}\label{Subsection-Weylgroupaction} The main references for this subsection are \cite{BKI, BK}.
\begin{definition}(\cite[before Remark 2.22]{BK})\label{geometriccrystaladditionaldatadef}
Let $i\in I$ be given. Define regular functions $\varphi_i$, $\varepsilon_i\in \mathcal{O}(\X)$ and a rational map $e_i:\mathbb{G}_m\times\X\dashrightarrow \X$ by
\[ \varphi_i(x) := \liecharacter_i(u_0) ,\quad \varepsilon_i(x) := \liecharacter_i(u_0) \alpha_i(t_0)= \varphi_i(x) (\alpha_i\circ\gammam)(x) \]
and
\[ e_i(c,x):= e_i^c(x) := x_i\left(\frac{c-1}{\varphi_i(x)}\right)\cdot x \cdot x_i\left(\frac{c^{-1}-1}{\varepsilon_i(x)}\right),\]
for $c\in \mathbb{G}_m$ and $x=u_0t_0\in\X$ with $t_0\in T$ and $u_0\in U^-$.
\end{definition}
\begin{lemma}\label{gmactionwelldefined}
For every $i\in I$, the rational map $e_i$ from Definition \ref{geometriccrystaladditionaldatadef} defines a regular $\mathbb{G}_m$-action on $\X\setminus\{\varphi_i=0\}= \X\setminus\{\varepsilon_i=0\}$.
\end{lemma}
\begin{proof} 
This is straightforward. For the reader's convenience, we provide the details in Appendix \ref{Subsection-proofs-from-bmodel}.
\end{proof}
\begin{remark}\label{gmactionwelldefinedremark0}
The regular functions $\varphi_i$ and $\varepsilon_i$ are not identically zero, so $\X\setminus\{\varphi_i=0\}= \X\setminus\{\varepsilon_i=0\}$ is a non-empty open subset. See Lemma \ref{Wpreservepositivepart}.
\end{remark}
\begin{remark}\label{gmactionwelldefinedremark}
The above $\mathbb{G}_m$-action should not be confused with the one from Definition \ref{Rietschmirrorgmactiondef}.
\end{remark}
\begin{definition}\label{simplereflectionongeometriccrystal}
Let $i\in I$ be given. Define a rational map $s_i:\X\dashrightarrow \X$ by
\[ s_i(x) := e_i^{\frac{1}{(\alpha_i\circ\gammam)(x)}}(x),\qquad x\in \X.\]
\end{definition}
\begin{lemma}\label{Weylgroupongeometriccrystal}
The rational maps $s_i$ ($i\in I$) defined in Definition \ref{simplereflectionongeometriccrystal} generate a rational $W$-action on $\X$.
\end{lemma}
\begin{proof}
This is \cite[Proposition 2.3]{BKI}. Let us sketch the proof. We have to show that for every sequence $(i_1,\ldots,i_m)$ of elements of $I$,
\[ s_{i_1}\cdots s_{i_m} = e\in W~\Longrightarrow ~s_{i_1}\circ\cdots\circ s_{i_m}=\id_{\X}.\]
By \cite[Theorem 3.8]{BKI}, we have
\begin{equation}\label{Weylgroupongeometriccrystaleq1}
s_{i_1}\cdots s_{i_m} = e\in W~\Longrightarrow ~e_{i_1}^{\beta_1(t)}\circ\cdots\circ e_{i_m}^{\beta_m(t)}=\id_{\X}\quad ~\text{ for all }t\in T,
\end{equation}
where $\beta_k:=s_{i_m}\cdots s_{i_{k+1}}(\alpha_{i_k})$ (see \cite[Lemma 2.1]{BKI} for the meaning of this theorem). The result will follow if we can show
\[ (e_{i_1}^{\beta_1(t)}\circ\cdots\circ e_{i_m}^{\beta_m(t)})(x)=(s_{i_1}\circ\cdots\circ s_{i_m})(x) \]
whenever $t=\gammam(x)^{-1}$, without assuming $ s_{i_1}\cdots s_{i_m} = e$. This can be proved by induction on $m$, using $\gammam\circ s_i=s_i\circ\gammam$ ($i\in I$), which is proved in Lemma \ref{Wpreservegamma} below.
\end{proof}
In what follows, we verify some properties of the $W$-action from Lemma \ref{Weylgroupongeometriccrystal}.
\begin{lemma}\label{Wpreservef}
$\fm$ is $W$-invariant.
\end{lemma}
\begin{proof} 
This is straightforward. We provide the details for the reader's convenience. It suffices to verify $\fm\circ s_i=\fm$ for all $i\in I$. By definition, we have
\[ \fm(e_i^c(x)) = \fm(x) + \frac{c-1}{\varphi_i(x)} + \frac{c^{-1}-1}{\varepsilon_i(x)}.\]
We have to show $\frac{c-1}{\varphi_i(x)} + \frac{c^{-1}-1}{\varepsilon_i(x)}=0$ if $c=\frac{1}{(\alpha_i\circ\gammam)(x)}$. Indeed, in this case we have $\varepsilon_i(x)=c^{-1} \varphi_i(x)$, and hence
\[\frac{c-1}{\varphi_i(x)} + \frac{c^{-1}-1}{\varepsilon_i(x)}=\frac{c-1}{\varphi_i(x)} +\frac{1-c}{ \varphi_i(x)} =0.\]
\end{proof}
\begin{lemma}\label{Wpreservepi}
$\pim$ is $W$-invariant.
\end{lemma}
\begin{proof}
This is clear from the definition.
\end{proof}

\begin{lemma}\label{Wpreservegamma}
$\gammam$ is $W$-equivariant.
\end{lemma}
\begin{proof}
It suffices to verify $\gammam\circ s_i=s_i\circ\gammam$ for all $i\in I$. This follows from the equality $\gammam\circ e_i^c=\alpha^{\vee}_i(c)\cdot\gammam$ ($i\in I$ and $c\in\mathbb{G}_m$), which is easy (see e.g., the proof of Lemma \ref{gmactionwelldefined}).
\end{proof}

\begin{lemma}\label{Wpreserveomega}
For every $w\in W$, we have $w^*\vol=(-1)^{\ell(w)}\vol$.
\end{lemma}
\begin{proof}
For $P=B$, a proof, which seems specific to this case, was given in \cite[Theorem 5.1.4]{Chhaibi}. We will prove the general case in Appendix \ref{Subsection-proofs-from-bmodel}.
\end{proof}

\begin{lemma}\label{Wpreservepositivepart}
$\Xp$ lies in the domain of definition of the $W$-action and is preserved by it.
\end{lemma}
\begin{proof}
For $P=B$, a proof, which seems specific to this case, was given in \cite[Proposition 4.4.30]{Chhaibi}. We will prove the general case in Appendix \ref{Subsection-proofs-from-bmodel}.
\end{proof}
\section{Mirror theorem}\label{Section-mirror-thm}

\subsection{Statement}\label{Subsection-mirror-thm-statement}
Recall the materials from Section \ref{Subsection-A-modelconnection} and Section \ref{Subsection-B-modelconnection}.
\begin{definition}\label{mirrormapdef}
Define the \textit{mirror map}
\[ \mirrormap: \ZL\xrightarrow{\sim} \mathbb{G}_m^{I\setminus I_P}\]
to be the isomorphism \eqref{Bmodelnotationeq1}, i.e., $\mirrormap(t):=(\alpha_i(t))_{i\in I\setminus I_P}$.
\end{definition}
\begin{theorem}\label{mirrorthm} (\cite[Theorem 1.2]{Chow3})
After making the identifications
\[ \SSS[\hslash]\simeq H^{\bullet}_{T^{\vee}\times\mathbb{G}_m}(\pt)\quad\text{ and }\quad \mathcal{O}(\ZL)\simeq \mathbb{C}[q_i^{\pm 1}|~i\in I\setminus I_P]\]
using the canonical isomorphism and $\mathcal{O}(\mirrormap^{-1})$ respectively, there exists a $\SSS[\hslash]\otimes \mathcal{O}(\ZL)$-linear map
\[ \Mir: \Bmodule \rightarrow \Amodule \]
such that
\begin{enumerate}
\item it is bijective;

\item it intertwines $\nabla^B_{\partial_{t_i}}$ (Definition \ref{Bconnectiondef}) and $\nabla^A_{\partial_{q_i}}$ (Definition \ref{Aconnectiondef}) for all $i\in I\setminus I_P$;

\item $\Mir([\vol])=1$; and

\item its semi-classical limit 
\[\Mir^{\hslash=0}:=\Mir\otimes_{\mathbb{C}[\hslash]}\mathbb{C} :\jacobi \rightarrow QH_{T^{\vee}}^{\bullet}(\GPd)[q_i^{-1}|~i\in I\setminus I_P]\]
is an isomorphism of $\sym^{\bullet}(\mathfrak{t})\otimes\mathcal{O}(\ZL)$-algebras (see Remark \ref{mirrorremark}).
\end{enumerate}
\hfill$\square$
\end{theorem}

\begin{remark} The Brieskorn lattice depends on the Rietsch mirror, and our Rietsch mirror differs conventionally and notationally from that used in \cite{Chow3}. See however Appendix D therein for an identification of these two versions.
\end{remark}
\subsection{First Chern class theorem}\label{Subsection-first-chern-class-thm}

\begin{theorem}\label{firstchernclassthm}
We have $\Mir([\fm\vol])=c_1^{T^{\vee}\times\mathbb{G}_m}(\GPd)$.
\end{theorem}
\begin{proof}
Recall the $\mathbb{G}_m$-action defined in Definition \ref{Rietschmirrorgmactiondef}. Let $V$ and $\widetilde{V}$ be the vector fields generating the actions on $\ZL$ and $\X$ respectively. By definition, $V=\sum_{i\in I\setminus I_P}(2\rho^{\vee}-2\rho^{\vee}_P)(\alpha_i)t_i\partial_{t_i}$. Since the anti-canonical line bundle of $\GPd$ is isomorphic to $\bigotimes_{i\in I\setminus I_P} L_{\omega^{\vee}_i}^{\otimes(2\rho^{\vee}-2\rho_P^{\vee})(\alpha_i)}$, we have $\hslash\nabla^A_{\mirrormap_*V}(1)=c_1^{T^{\vee}\times\mathbb{G}_m}(\GPd)$ (see Definition \ref{Aconnectiondef}). By Lemma \ref{Rietschmirrorgmactionlemma}, we have
\[\mathcal{L}_{\widetilde{V}}\fm=\fm,\quad  \mathcal{L}_{\widetilde{V}}\gammam^*\langle - ,\mcf_T\rangle = 0,\quad  \mathcal{L}_{\widetilde{V}}\vol = 0,\] 
and that $\widetilde{V}$ is a lift of $V$. It follows that, by Definition \ref{Bconnectiondef}, $\hslash\nabla^B_V([\vol])=[\fm\vol]$. Therefore, by Theorem \ref{mirrorthm},
\[ \Mir([\fm\vol])= \Mir( \hslash\nabla^B_V([\vol]) )  =\hslash\nabla^A_{\mirrormap_*V}(1) =c_1^{T^{\vee}\times\mathbb{G}_m}(\GPd) .\]
\end{proof}
\begin{corollary}\label{criticalvalue=eigenvalue}
We have $\fm\circ \spec(\Mir^{\hslash=0}) = c_1^{T^{\vee}}(\GPd)$. \hfill$\square$
\end{corollary}
\begin{remark}\label{firstchernclassthmremark}
By the fact that every Artinian ring is the product of its localizations at its maximal ideals, Corollary \ref{criticalvalue=eigenvalue} implies that for every $t\in\ZL$, the set of critical values of $\fm|_{\pim^{-1}(t)}$ and the set of eigenvalues of the operator $c_1(\GPd)\star_{\mirrormap(t)}-$ on $QH^{\bullet}(\GPd)_{\mirrormap(t)}$, both counted with multiplicities, are equal. When $G$ is of type A, this result has been proved by Li, Rietsch, Yang and Zhang \cite{LRYZ}. Their proof does not rely on the existence of a mirror isomorphism.
\end{remark}
\subsection{Description of mirror isomorphism}\label{Subsection-descriptionofmirrorisom}
We will need the following description of the limit $\Mir^{\hslash,h=0}:=\Mir\otimes_{\SSS[\hslash]}\mathbb{C}$. Recall the elements $e_i\in\mathfrak{g}_{\alpha_i}$ fixed in Section \ref{Subsection-Bmodelnotation}. Define $F\in\mathfrak{g}^*$ to be the unique element such that $F(e_i)=1$ for all $i\in I$, and $F$ is zero on other root spaces as well as $\mathfrak{t}$. Define 
\[ \UF:=\{ u\in U^-|~u\cdot F = F\}.\]
Recall the following results from the literature.
\begin{enumerate}
\item Rietsch \cite[Theorem 4.1]{Rietsch} proved that $\crit = \UF\times_G U\ZL\ol{w_P}U$ as subschemes of $B^-$. Let 
\[ \widetilde{\Phi}^0_{R}: \mathcal{O}(\UF) \rightarrow \mathcal{O}(\crit) \]
be the ring map induced by the inclusion
\[ \crit = \UF\times_G U\ZL\ol{w_P}U\hookrightarrow \UF.\]
 
\item Yun and Zhu \cite{YunZhu} constructed a ring isomorphism\footnote{In fact, what they constructed is a map $\mathcal{O}(B^{\vee}_{e^T(0)})\rightarrow H_{-\bullet}(\ag_G)$. To obtain our map from theirs, apply the transpose $g\mapsto g^T$ to $U_F^-$, and interchange the roles of $G$ and $G^{\vee}$. Here, the transpose of $G$ is the unique anti-automorphism of $G$ characterized by 
\[ x_i(a)^T=y_i(a),\quad t^T=t\quad \text{and}\quad y_i(a)^T=x_i(a)\]
for $i\in I$, $a\in \mathbb{A}^1$ and $t\in T$.}
\[ \Phi^0_{YZ}:\mathcal{O}(\UF)\xrightarrow{\sim} H_{-\bullet}(\ag),\]
where $\ag$ is the affine Grassmannian of $G^{\vee}$.

\item Discovered by Peterson \cite{Peterson} and proved by Lam and Shimozono \cite{LamShimozono}, there is a ring map 
\[ \Phi^0_{PLS}: H_{-\bullet}(\ag)\rightarrow QH^{\bullet}(\GPd)[q_i^{-1}|~i\in I\setminus I_P],\]
which is surjective after localization and has an explicit description in terms of the affine and quantum Schubert bases.
\end{enumerate}
  
\begin{theorem}\label{mirrorisomfactorthrucentralizer}
The following diagram is commutative.
\begin{equation}\nonumber
\begin{tikzpicture}
\tikzmath{\x1 = 8; \x2 = 3;}
\node (A) at (0,0) {$\mathcal{O}(\UF)$} ;
\node (B) at (\x1,0) {$H_{-\bullet}(\ag)$} ;
\node (C) at (0,-\x2) {$\mathcal{O}(\crit)$} ;
\node (D) at (\x1,-\x2) {$QH^{\bullet}(\GPd)[q_i^{-1}|~i\in I\setminus I_P]$} ;

\path[->, font=\small] (A) edge node[above]{$\Phi^0_{YZ}$} (B);
\path[->, font=\small] (A) edge node[left]{$\widetilde{\Phi}^0_R$} (C);
\path[->, font=\small] (B) edge node[right]{$\Phi^0_{PLS}$} (D);
\path[->, font=\small]
(C) edge node[above]{$ \Mir^{\hslash,h=0}$} (D);
\end{tikzpicture}
\end{equation}
\end{theorem}
\begin{proof}
This follows from \cite[Theorem 1.4]{Chow3} by taking $-\otimes_{\SSS[\hslash]}\mathbb{C}$.
\end{proof}
\begin{remark}\label{mirrorisomfactorthrucentralizerremark}
The superscript 0 in the notation of the maps introduced above indicates that they are the non-equivariant limits of their $T^{\vee}$-equivariant analogs. In fact, Theorem \ref{mirrorisomfactorthrucentralizer} also has a $T^{\vee}$-equivariant analog. We have worked in this setting because this is all we will need.
\end{remark}
\subsection{Totally positive critical point}\label{Subsection-totally-positive-critical-point}
Recall the totally positive part $\Xp$ of $\X$ defined in Definition \ref{positivepartdef}. For $t\in\ZL$, define $\Xt:=\pim^{-1}(t)\subseteq \X$ and $\Xtp:=\Xt\cap \Xp$. Note that $\Xtp\neq\emptyset$ if and only if $t\in\ZLp$. ($\ZLp$ is defined in Definition \ref{nonnegativepartofsubgroupofG}(2).)
\begin{lemma}\label{Schubertpositivegoestopositivepart}
For every $t\in\ZLp$, the isomorphism 
\[ t\times_{\ZL}\spec(\Mir^{\hslash,h=0}): \spec QH^{\bullet}(\GPd)_{\mirrormap(t)}\rightarrow \critical(\fm|_{\Xt})\]
sends the point $\spp_{\mirrormap(t)}$ from Proposition \ref{existSchubertpositive} to a point $x_t$ in $\Xtp$.
\end{lemma}
\begin{proof}
By \cite[Proposition 11.3]{LamRietsch} (see also Theorem \ref{LamRietschthm}), $\spec(\Phi^0_{PLS}\circ\Phi^0_{YZ})$ sends $\spp_{\mirrormap(t)}$ to a point in $U_{\geqslant 0}^-$, the submonoid of $U^-$ with unit generated by $y_i(a)$ for $i\in I$ and $a\in\mathbb{R}_{>0}$. By Theorem \ref{mirrorisomfactorthrucentralizer}, this point is equal to $x_t:=\spec(\Mir^{\hslash,h=0})(\spp_{\mirrormap(t)})\in\X$. Since $\spp_{\mirrormap(t)}$ lies over $\mirrormap(t)$ and $\Mir^{\hslash,h=0}$ is linear with respect to $\mathcal{O}(\mirrormap^{-1})$, it follows that $x_t$ lies over $t$, and hence we have $x_t\in\Xt\subseteq Ut\ol{w_P}U$. It follows that 
\[ x_t\in  U_{\geqslant 0}^-\cap Ut\ol{w_P}U\subseteq U_{\geqslant 0}^-\cap Bw_P B.\]
Let $\iii=(i_1,\ldots,i_{\ell})$ be a reduced decomposition of $w_P$. By \cite[Proposition 2.7 \& Corollary 2.8]{Lusztig_positivity}, there exist $a_1,\ldots,a_{\ell}\in\mathbb{R}_{>0}$ such that $x_t=y_{i_1}(a_1)\cdots y_{i_{\ell}}(a_{\ell})$. By $y_i(a)=x_{-i}(a)\alpha_i^{\vee}(a)$ (see Definition \ref{toruschartforBruhatdef}(1)) and by moving all factors $\alpha_{i_k}^{\vee}(a_k)$ to the left, we can find $t'\in T$ and $a'_1,\ldots,a'_{\ell}\in \mathbb{R}_{>0}$ such that 
\begin{equation}\label{Schubertpositivegoestopositiveparteq1}
 y_{i_1}(a_1)\cdots y_{i_{\ell}}(a_{\ell}) = t' \cdot x_{-i_1}(a'_1)\cdots x_{-i_{\ell}}(a'_{\ell}) .
\end{equation}
Observe that the LHS of \eqref{Schubertpositivegoestopositiveparteq1} lies in $Ut\ol{w_P} U$, and the RHS of \eqref{Schubertpositivegoestopositiveparteq1} lies in $Ut'\ol{w_P}U$, and hence we have $t'=t$. By Lemma \ref{toruschartforBruhatpositive}, $x_{-i_1}(a'_1)\cdots x_{-i_{\ell}}(a'_{\ell})=\ttm_{\iii}(a'_1,\ldots,a'_{\ell})$ belongs to $(B^-_{w_P})_{>0}$. Therefore, 
\[  x_t =  y_{i_1}(a_1)\cdots y_{i_{\ell}}(a_{\ell}) = t\cdot \ttm_{\iii}(a'_1,\ldots,a'_{\ell})\in \Xtp.\]
\end{proof}
\begin{corollary}\label{fhascriticalpoint}
$\fm|_{\Xtp}$ has a critical point, namely $x_t$. \hfill$\square$
\end{corollary}
\begin{lemma}\label{criticalvalue=conjectureOeigenvalue}
For every $t\in\ZLp$, we have 
\[ \fm(x_t) = \EO^{\mirrormap(t)},\]
where $x_t$ is the point from Lemma \ref{Schubertpositivegoestopositivepart} and $\EO^{\mirrormap(t)}$ is the constant defined in Definition \ref{conjectureOevaluedef}.
\end{lemma}
\begin{proof}
This follows from Proposition \ref{existSchubertpositive} and Corollary \ref{criticalvalue=eigenvalue}.
\end{proof}

\section{Flat sections}\label{Section-flatsections}
\subsection{A-side}\label{Subsection-Aside}
\begin{definition}\label{gammafunctiondef}
Define the gamma function 
\[\Gamma(z):=\int_0^{\infty} e^{-t}~t^{z-1}~dt,\qquad \real(z)>0.\]
\end{definition}

\begin{lemma}\label{gammafunctionlemma}
We have $\Gamma(1)=1$ and $\Gamma(1+z)=z\Gamma(z)$, and hence $\Gamma$ extends to a meromorphic function on the complex plane, and there exist $a_1,a_2,\ldots\in\mathbb{C}$ such that 
\[ \Gamma(1+z)=1+a_1z+a_2z^2+\cdots,\qquad |z|<1.\]
\end{lemma}
\begin{proof}
This is well-known.
\end{proof}
\begin{definition}\label{vectorbundle0def}
Define 
\[\bundlez:= H^{\bullet}_{T^{\vee}}(\GPd)\]
regarded as a vector bundle on $\spec H^{\bullet}_{T^{\vee}}(\pt)\simeq\mathfrak{t}^{\vee}$.
\end{definition}
\begin{definition}\label{gammahatdef}
Define 
\[ \widehat{\Gamma}_{\GPd} :=~ \prod_{i=1}^{\ell}\Gamma(1+\delta_i), \]
where $\delta_1,\ldots,\delta_{\ell}$ are the $T^{\vee}$-equivariant Chern roots of the tangent bundle of $\GPd$. A priori, we regard it as a section of $\bundlez$ on a formal neighbourhood of $0\in\mathfrak{t}^{\vee}$.
\end{definition}
\begin{lemma}\label{gammahatlemma}
$\widehat{\Gamma}_{\GPd}$ is a holomorphic section of $\bundlez$ on an open neighbourhood of $0\in\mathfrak{t}^{\vee}$.
\end{lemma}
\begin{proof}
See Appendix \ref{Subsection-proofs-from-flatsections}.
\end{proof}
\begin{definition}\label{DDDdef} Fix an open connected neighbourhood $\DDD$ of $0\in\mathfrak{t}^{\vee}$ such that 
\begin{enumerate}
\item $\DDD$ is $W$-invariant;

\item $\DDD$ is contained in the open neighbourhood from Lemma \ref{fundamentalsolutionconverges}; and

\item $\DDD$ is contained in the open neighbourhood from Lemma \ref{gammahatlemma}.
\end{enumerate}
\end{definition}
For a vector bundle $E$, denote by $\holosection(E)$ the space of its holomorphic sections.  
\begin{definition}\label{weirdoperatordef} Let $\hslash\in\mathbb{R}_{>0}$. Following the literature (e.g., \cite{IZ}), we define a linear map
\[ \hslash^{-\mu}\hslash^{c_1}: \holosection(\bundlez|_{\DDD})\rightarrow \holosection(\bundlez|_{\hslash\DDD})\]
to be the composition $\hslash^{-\mu'}\circ \hslash^{\frac{\ell}{2}}\circ\hslash^{c_1}$ of three linear maps defined as follows:
\begin{enumerate}[(i)]
\item $\hslash^{c_1}:=\exp((\log\hslash) c^{T^{\vee}}_1(\GPd)\cup -)$; 

\item $\hslash^{\frac{\ell}{2}}:=$ multiplication by $\hslash^{\frac{\ell}{2}}$; and

\item $\hslash^{-\mu'}$ sends a section $s\in \holosection(\bundlez|_{\DDD})$ to a section $s'\in \holosection(\bundlez|_{\hslash\DDD})$ defined by
\[ s'(h):= \sqrt{\hslash}^{-1}\cdot s(\sqrt{\hslash}\cdot h),\]
where the two dots $\cdot$ denote the $\mathbb{G}_m$-actions (induced by the standard gradings\footnote{In particular, $\sqrt{\hslash}\cdot h$ is equal to $\hslash^{-1}$ times $h$ with respect to the scalar multiplication.}) on the bundle $\bundlez$ and the base $\mathfrak{t}^{\vee}$ respectively.
\end{enumerate}
\end{definition}
Put $\CCp:=\{z\in\mathbb{C}|~\real(z)>0\}$.
\begin{definition}\label{iadef}
For $\hslash\in\mathbb{R}_{>0}$, $h\in\hslash\DDD$ , $q\in \CCp^{I\setminus I_P}$ and $y\in H_{T^{\vee}\times\mathbb{G}_m}^{\bullet}(\GPd)$, define
\[ \ia(\hslash,h,q,y):= \hslash^{\frac{\ell}{2}} \int_{\GPd} S(\hslash,h,q)\left(\hslash^{-\mu} \hslash^{c_1} \widehat{\Gamma}_{\GPd}\right)\cup y,\]
where $S(\hslash,h,q)$ is defined in Definition \ref{fundamentalsolutiondef} and the branch for each $\log q_i$ involved in its definition is taken to be the one containing the real line.
\end{definition}
\begin{lemma}\label{iaok}
For every $\hslash\in\mathbb{R}_{>0}$ and $y\in H_{T^{\vee}\times\mathbb{G}_m}^{\bullet}(\GPd)$, the function $(h,q)\mapsto \ia(\hslash,h,q,y)$ is holomorphic on $\hslash\DDD\times \CCp^{I\setminus I_P}$.
\end{lemma}
\begin{proof}
This follows from Lemma \ref{fundamentalsolutionconverges} and Lemma \ref{gammahatlemma}.
\end{proof}
\begin{remark}
Up to a factor, $\ia(\hslash,0,q,1)$ is the \textit{quantum cohomology central charge} of the structure sheaf $\mathcal{O}_{\GPd}$ \cite{IZ}.
\end{remark}
\begin{lemma}\label{iaWinv}
For every $\hslash\in\mathbb{R}_{>0}$, $h\in\hslash\DDD$, $q\in\CCp^{I\setminus I_P}$ and $w\in W$, we have 
\[\ia(\hslash,w(h),q,1)=\ia(\hslash,h,q,1).\]
\end{lemma}
\begin{proof}
Since the tangent bundle of $\GPd$ is $G^{\vee}$-linearized, $\widehat{\Gamma}_{\GPd}$ is a $W$-equivariant section of $\bundlez|_{\DDD}$. Moreover, we have the equality $w(S(\hslash,h,q)(x))=S(\hslash,w(h),q)(w(x))$ because all line bundles on $\GPd$ are $G^{\vee}$-linearized, and there are natural $G^{\vee}$-actions on $\overline{\mathcal{M}}_{0,2}(\GPd,\beta_{\mathbf{d}})$ for which the evaluation morphisms and the $\psi$-classes are $G^{\vee}$-equivariant. The result follows.  
\end{proof}
\begin{lemma}\label{ialimit}
For every $\hslash\in\mathbb{R}_{>0}$, $h\in\hslash\DDD$ and $\lambda\in\mathfrak{t}$ such that $\real(\alpha^{\vee}(h))<0$ if $\alpha\in R^+$ and $\lambda(\alpha_i)\in\mathbb{R}_{>0}$ (resp. $=0$) if $i\in I\setminus I_P$ (resp. $i\in I_P$), we have
\[ \lim_{\mathbb{R}_{>0}\ni~ s\to 0^+} s^{-\frac{\lambda(h)}{\hslash}} \ia(\hslash,h,q_{\lambda}(s),1) = \hslash^{-\frac{(2\rho^{\vee}-2\rho_P^{\vee})(h)}{\hslash}}\prod_{\alpha\in -(R^+\setminus R^+_P)}\Gamma\left(\frac{\alpha^{\vee}(h)}{\hslash}\right),\]
where $q_{\lambda}(s):=\left(s^{\lambda(\alpha_i)}\right)_{i\in I\setminus I_P}$, $2\rho^{\vee}-2\rho_P^{\vee}:=\sum_{\alpha\in R^+\setminus R^+_P}\alpha^{\vee}$ and $\Gamma$ is defined in Definition \ref{gammafunctiondef}.
\end{lemma}
\begin{proof}
Consider \eqref{fundamentalsolutiondefeq1} with $x=\hslash^{-\mu} \hslash^{c_1} \widehat{\Gamma}_{\GPd}$. The term $\prod_{i\in I\setminus I_P} q_i^{d_i} $ is equal to $s^{\sum_{i\in I\setminus I_P} d_i\lambda(\alpha_i)}$, and the term $e^{-H(q)/\hslash}$ is equal to $\exp\left( -(\frac{\log s}{\hslash})\sum_{i\in I\setminus I_P}\lambda(\alpha_i) c_1^{T^{\vee}\times\mathbb{G}_m}(L_{\omega_i^{\vee}})\right)$. By our assumption on $\lambda$, the first expression goes to 0 as $s\to 0^+$, and the restriction of the second expression to a $T^{\vee}\times\mathbb{G}_m$-fixed point $wP^{\vee}\in \GPd$ ($w\in W^P$) is equal to $s^{\frac{\lambda(w^{-1}(h))}{\hslash}}$ (recall $L_{\omega_i^{\vee}}=G^{\vee}\times^{P^{\vee}}\mathbb{C}_{-\omega_i^{\vee}}$). Moreover, we have $-\lambda+w(\lambda)\in -\sum_{i\in I}\mathbb{R}_{\geqslant 0}\cdot\alpha^{\vee}_i$, and it is non-zero unless $w=e$ (recall $w\in W^P$). It follows that, by our assumption on $h$,
\[ \lim_{s\to 0^+} s^{\frac{-\lambda(h)+\lambda(w^{-1}(h))}{\hslash}} = \left\{ 
\begin{array}{cc}
1 & w=e\\ [.5 em]
0 & \text{otherwise}
\end{array}
\right. .\]
Therefore, if we expand the integral $\ia(\hslash,h,q_{\lambda}(s),1)$ by localization, only the restriction of the leading term $e^{-H(q)/\hslash}x$ to $eP^{\vee}$ contributes to the limit $\displaystyle\lim_{s\to 0^+} s^{-\frac{\lambda(h)}{\hslash}} \ia(\hslash,h,q_{\lambda}(s),1)$, and this contribution is equal to $\hslash^{\frac{\ell}{2}}$ times the restriction of $\hslash^{-\mu} \hslash^{c_1} \widehat{\Gamma}_{\GPd}$ to $eP^{\vee}$ times the contribution by the tangent space, i.e., $\frac{1}{\mathbf{e}(T_{eP^{\vee}}(\GPd))(h)}=\frac{1}{\prod_{\alpha\in -(R^+\setminus R^+_P)}\alpha^{\vee}(h)}$. The rest of the proof is a straightforward computation, which we leave to the reader.
\end{proof}
Recall the vector bundle $\bundle$ defined in Definition \ref{vectorbundledef}.
\begin{definition}\label{sadef}
Define a section $\sa$ of $\bundle|_{\{(\hslash,h)\in\mathbb{R}_{>0}\times\mathfrak{t}^{\vee}|~ h\in\hslash\DDD\}\times \CCp^{I\setminus I_P}}$ by
\[ \sa(\hslash,h,q) := \hslash^{\frac{\ell}{2}}S(\hslash,h,q)\left(\hslash^{-\mu} \hslash^{c_1} \widehat{\Gamma}_{\GPd}\right) = \sum_{v\in W^P} \ia(\hslash,h,q,\sigma^v)\sigma_v.\]

\end{definition}
\begin{lemma}\label{saflat}
For every $i\in I\setminus I_P$, we have $\nabla^A_{\partial_{q_i}}\sa=0$.
\end{lemma}
\begin{proof}
This follows from Lemma \ref{fundamentalsolutionlemma}.
\end{proof}
\subsection{B-side}\label{Subsection-Bside}
Recall the Rietsch mirror $(\X,\fm,\pim,\gammam,\vol)$ defined in Definition \ref{Rietschmirrordef}. For $t\in\ZL$, define $\Xt:=\pim^{-1}(t)$, $\fmt:=\fm|_{\Xt}$ and $\gammamt:=\gammam|_{\Xt}$. Recall also the totally positive part $\Xp$ of $\X$ defined in Definition \ref{positivepartdef}. For $t\in\ZLp$, define $\Xtp:=\Xt\cap\Xp$.

\begin{definition}\label{ibdef}
For $\hslash\in\mathbb{R}_{>0}$, $h\in\mathfrak{t}^{\vee}$ , $t\in\ZLp$ and $\omega\in \SSS[\hslash]\otimes\Omega^{top}(\X/\ZL)$, define
\[ \ib(\hslash,h,t,\omega):= \int_{\Xtp} e^{-\fmt/\hslash} \gammamt^{h/\hslash} \omega_{(-\hslash,h,t)}.\]
Here, 
\begin{enumerate}
\item the orientation on $\Xtp$ is specified in Definition \ref{positivepartorientation};

\item $\gammamt^{h/\hslash} :=\exp\left(\frac{1}{\hslash}\langle h,\log\circ\gammamt\rangle\right)$, where $\log:T_{>0}\xrightarrow{\sim}\mathfrak{t}_{\mathbb{R}}$ is the inverse of the exponential map restricted to $\mathfrak{t}_{\mathbb{R}}:=\{x\in\mathfrak{t}|~\forall i\in I,~\alpha_i(x)\in\mathbb{R}\}$ (we have $\gammamt(\Xtp)\subseteq T_{>0}$ by Lemma \ref{gammapreservepositivepart}); and

\item $\omega_{(-\hslash,h,t)}\in \Omega^{top}(\Xt)$ is $\omega$ specialized at $(-\hslash,h,t)$.
\end{enumerate}  
\end{definition}
\begin{lemma}\label{ibok}
For every $\hslash\in \mathbb{R}_{>0}$ and $\omega\in \SSS[\hslash]\otimes\Omega^{top}(\X/\ZL)$, the function $(h,t)\mapsto \ib(\hslash,h,t,\omega)$ is well-defined and has an analytic continuation on the open subset $\mathfrak{t}^{\vee}\times\ZLpp$ of $ \mathfrak{t}^{\vee}\times \ZL$, where $\ZLpp:=\{t\in \ZL|~ \forall i\in I\setminus I_P,~\real(\alpha_i(t))>0\}$.
\end{lemma}
\begin{proof}
Take a reduced decomposition $\iii$ of $w_P$ and identify $\Xp$ with $\ZLp\times \mathbb{R}_{>0}^{\ell}$ using $\XXp_{\iii}$ (see Lemma \ref{toruschartforgeometriccrystalpositive}) so that $\ib(\hslash,h,t,\omega)$ becomes an integral over $\mathbb{R}_{>0}^{\ell}$. By Lemma \ref{fLusztig}, Lemma \ref{gammaLusztig} and Lemma \ref{omegaLusztig}, the latter integral is of the form
\begin{equation}\label{ibokeq1}
 e^{\frac{\langle h,\log t\rangle}{\hslash}}\int_{\mathbb{R}_{>0}^{\ell}} e^{-\frac{a_1+\cdots+a_{\ell}+\sum_{i\in I\setminus I_P}\alpha_i(t) P_i(\mathbf{a})}{\hslash}} \prod_{k=1}^{\ell} a_k^{\frac{\beta_k^{\vee}(h)}{\hslash}-1} g(\hslash,h,t,\mathbf{a}) da_1\cdots da_{\ell},
\end{equation}
where each $P_i(\mathbf{a})$ is a Laurent polynomial in $\mathbf{a}=(a_1,\ldots,a_{\ell})$ with positive coefficients, $\beta^{\vee}_1,\ldots,\beta^{\vee}_{\ell}$ come from Lemma \ref{gammaLusztig} and $g\in \SSS\otimes\mathcal{O}(\ZL)[\hslash,a_1^{\pm 1},\ldots, a_{\ell}^{\pm 1}]$. 

By general measure theory, our function is well-defined and holomorphic if we can bound every iterated partial derivative, with respect to $h$ and $t$, of the integrand in \eqref{ibokeq1}, at least near a given point $(h_0,t_0)\in\mathfrak{t}^{\vee}\times\ZLpp$, by an integrable function depending on $a_1,\ldots,a_{\ell}$ only. These partial derivatives are $\SSS[\hslash^{\pm 1}]\otimes\mathcal{O}(\ZL)$-linear combinations of functions of the form
\begin{equation}\label{ibokeq2}
e^{-\frac{a_1+\cdots+a_{\ell}+\sum_{i\in I\setminus I_P}\alpha_i(t) P_i(\mathbf{a})}{\hslash}} a_1^{b_1+\frac{\beta_1^{\vee}(h)}{\hslash}} \cdots a_{\ell}^{b_{\ell}+\frac{\beta_{\ell}^{\vee}(h)}{\hslash}} (\log a_1)^{c_1}\cdots (\log a_{\ell})^{c_{\ell}},
\end{equation}
where $b_1,\ldots,b_{\ell}\in\mathbb{Z}$ and $c_1,\ldots,c_{\ell}\in\mathbb{Z}_{\geqslant 0}$. Choose $t'\in\ZLp$ and $d^{\pm}_1,\ldots,d^{\pm}_{\ell}\in\mathbb{R}$ such that $\real(\alpha_i(t_0))>\alpha_i(t')$ for all $i\in I\setminus I_P$, and $d^-_k<b_k+\frac{\real(\beta_k^{\vee}(h_0))}{\hslash}<d^+_k$ for all $k=1,\ldots,\ell$. Then the set of $(h,t)\in \mathfrak{t}^{\vee}\times\ZLpp$ satisfying the above two conditions, with $(h_0,t_0)$ replaced by $(h,t)$, is an open neighbourhood of $(h_0,t_0)$, and for every point $(h,t)$ in this neighbourhood, we have
\begin{align*}
    \text{absolute value of}~\eqref{ibokeq2}~& \leqslant \sum_{(\epsilon_k)\in\{-,+\}^{\ell}} e^{-\frac{a_1+\cdots+a_{\ell}+\sum_{i\in I\setminus I_P}\alpha_i(t') P_i(\mathbf{a})}{\hslash}} a_1^{d^{\epsilon_1}_1}\cdots a_{\ell}^{d^{\epsilon_{\ell}}_{\ell}} |\log a_1|^{c_1}\cdots |\log a_{\ell}|^{c_{\ell}}\\
    &=\sum_{(\epsilon_k)\in\{-,+\}^{\ell}} e^{-\frac{\left(\fm\circ\XXp_{\iii}\right)(t',\mathbf{a})}{\hslash}} a_1^{d^{\epsilon_1}_1}\cdots a_{\ell}^{d^{\epsilon_{\ell}}_{\ell}} |\log a_1|^{c_1}\cdots |\log a_{\ell}|^{c_{\ell}}.
\end{align*}
It remains to show that the integral of each summand of the RHS of the last inequality over $\mathbb{R}_{>0}^{\ell}$ is finite. By Corollary \ref{fhascriticalpoint}, the function $\fm|_{\Xtpp}$ has a critical point. It follows that the function on $\mathbb{R}^{\ell}_{>0}$ defined by $\mathbf{a}\mapsto\frac{1}{\hslash}\left(\fm\circ\XXp_{\iii}\right)(t',\mathbf{a})$ has a critical point. The result then follows from Lemma \ref{convexhull} below.
\end{proof}
 
\begin{lemma}\label{convexhull}
Let $S$ be a finite subset of $\mathbb{Z}^{\ell}$ spanning $\mathbb{R}^{\ell}$ and $f(\mathbf{a}):=\sum_{\mathbf{v}\in S}f_{\mathbf{v}}a_1^{v_1}\cdots a_{\ell}^{v_{\ell}}$ be a Laurent polynomial in $\mathbf{a}=(a_1,\ldots,a_{\ell})$, where each $f_{\mathbf{v}}\in\mathbb{R}_{>0}$. Suppose $f$ has a critical point in $\mathbb{R}_{>0}^{\ell}$. Then for every $c_1,\ldots,c_{\ell}\in\mathbb{Z}_{\geqslant 0}$ and $d_1,\ldots,d_{\ell}\in\mathbb{R}$, 
\[ \int_{\mathbb{R}_{>0}^{\ell}}e^{-f(\mathbf{a})} a_1^{d_1}\cdots a_{\ell}^{d_{\ell}} |\log a_1|^{c_1}\cdots |\log a_{\ell}|^{c_{\ell}}da_1\cdots da_{\ell}<+\infty. \]
\end{lemma}
\begin{proof}
See Appendix \ref{Subsection-proofs-from-flatsections}.
\end{proof}

\begin{lemma}\label{ibdescend}
$\ib(\hslash,h,t,\omega)$ does not depend on $\omega$ but the class $[\omega]\in \Bmodule$ it represents.
\end{lemma}
\begin{proof}
By Definition \ref{Brieskorndef}, we have to show
\[ \int_{\Xtp} e^{-\fmt/\hslash} \gammamt^{h/\hslash} \left(-\hslash d\omega + d \fmt\wedge\omega - (\gammamt^* \langle h, \mcf_T \rangle )\wedge \omega\right)=0\]
for all $\omega\in \Omega^{top-1}(\Xt)$. (Recall we are specializing at $(-\hslash,h,t)$.) The LHS is nothing but the integral $-\hslash \int_{\Xtp} d\left(e^{-\fmt/\hslash} \gammamt^{h/\hslash}\omega \right)$, and hence the result follows from Stokes' theorem.
\end{proof}
\begin{remark}
Up to a factor, $\ib(\hslash,0,t,[\vol])$ is the \textit{LG central charge} of $\Xp$ \cite{IZ}.
\end{remark}
\begin{lemma}\label{ibWinv}
For every $\hslash\in\mathbb{R}_{>0}$, $h\in\mathfrak{t}^{\vee}$, $t\in\ZLp$ and $w\in W$, we have 
\[\ib(\hslash,w(h),t,[\vol])=\ib(\hslash,h,t,[\vol]) .\]
\end{lemma}
\begin{proof}
Consider the rational $W$-action on $\X$ introduced in Lemma \ref{Weylgroupongeometriccrystal}. By Lemma \ref{Wpreservepi} and Lemma \ref{Wpreservepositivepart}, it induces a $W$-action on $\Xtp$. In particular, we have a diffeomorphism $w:\Xtp\xrightarrow{\sim}\Xtp$. By applying it to the first integral, we get
\begin{align*}
&\ib(\hslash,w(h),t,[\vol])\\
=~&\int_{\Xtp} e^{-\fmt/\hslash} \gammamt^{w(h)/\hslash} (\vol)_{(-\hslash,w(h),t)}\\
=~&\pm \int_{\Xtp} e^{-w^*\fmt/\hslash} (\gammamt\circ w)^{w(h)/\hslash} (w^*\vol)_{(-\hslash,w(h),t)}\\
=~&\pm \int_{\Xtp} e^{-\fmt/\hslash} \gammamt^{h/\hslash} \left((-1)^{\ell(w)}(\vol)_{(-\hslash,h,t)}\right)\\
=~& \pm (-1)^{\ell(w)}\ib(\hslash,h,t,[\vol]),
\end{align*}
where the third equality follows from Lemma \ref{Wpreservef}, Lemma \ref{Wpreservegamma}, Lemma \ref{Wpreserveomega} and the fact that $\vol$ is by definition independent of $h$. Here, the sign $\pm $ is $+$ (resp. $-$) if $w$ preserves (resp. reverses) the orientation. Since $\vol$ is an orientation form (see Definition \ref{positivepartorientation}), the sign cancels with $(-1)^{\ell(w)}$. The result follows.
\end{proof}
\begin{lemma}\label{iblimit}
For every $\hslash\in\mathbb{R}_{>0}$, $h\in\mathfrak{t}^{\vee}$ and $\lambda\in\mathfrak{t}$ such that $\real(\alpha^{\vee}(h))<0$ if $\alpha\in R^+$ and $\lambda(\alpha_i)\in\mathbb{R}_{>0}$ (resp. $=0$) if $i\in I\setminus I_P$ (resp. $i\in I_P$), we have
\[ \lim_{\mathbb{R}_{>0}\ni~ s\to 0^+} s^{-\frac{\lambda(h)}{\hslash}} \ib(\hslash,h, t_{\lambda}(s),[\vol]) = \hslash^{-\frac{(2\rho^{\vee}-2\rho_P^{\vee})(h)}{\hslash}}\prod_{\alpha\in -(R^+\setminus R^+_P)}\Gamma\left(\frac{\alpha^{\vee}(h)}{\hslash}\right),\]
where $t_{\lambda}(s)$ is characterized by $(\alpha_i\circ t_{\lambda})(s)=s^{\lambda(\alpha_i)}$ for all $i\in I\setminus I_P$, $2\rho^{\vee}-2\rho_P^{\vee}:=\sum_{\alpha\in R^+\setminus R^+_P}\alpha^{\vee}$ and $\Gamma$ is defined in Definition \ref{gammafunctiondef}.
\end{lemma}
\begin{proof}
Let $\iii_0$ be the reduced decomposition of $w_P$ fixed in Definition \ref{volumeformrescale}. By identifying $\Xp$ with $\ZLp\times \mathbb{R}_{>0}^{\ell}$ using $\XXp_{\iii_0}$ (see Lemma \ref{toruschartforgeometriccrystalpositive}) and applying Lemma \ref{fLusztig} and Lemma \ref{gammaLusztig}, we get
\begin{align*}
&\ib(\hslash,h, t_{\lambda}(s),[\vol])\\
=~&e^{\frac{1}{\hslash}\langle h,(\log\circ t_{\lambda})(s)\rangle}\int_{\mathbb{R}_{>0}^{\ell}} e^{-\frac{a_1+\cdots+a_{\ell}+\sum_{i\in I\setminus I_P}s^{\lambda(\alpha_i)} P_i(\mathbf{a})}{\hslash}} \prod_{k=1}^{\ell} a_k^{\frac{\beta_k^{\vee}(h)}{\hslash}-1} da_1\cdots da_{\ell},
\end{align*}
where $P_i(\mathbf{a})$ and $\beta_k^{\vee}$ come from Lemma \ref{fLusztig} and Lemma \ref{gammaLusztig} respectively. By our assumption on $\lambda$, we have
$e^{\frac{1}{\hslash}\langle h,(\log\circ t_{\lambda})(s)\rangle} = s^{\frac{\lambda(h)}{\hslash}}$ and $\displaystyle\lim_{s\to 0^+}s^{\lambda(\alpha_i)}=0$ for all $i\in I\setminus I_P$. Moreover, we have $\{\beta^{\vee}_1,\ldots, \beta^{\vee}_{\ell}\}=\{\alpha^{\vee}\}_{\alpha\in -(R^+\setminus R^+_P)}$ by Lemma \ref{betadescription}, and hence $\real(\beta^{\vee}_k(h))>0$ for all $k$, by our assumption on $h$. Therefore,
\begin{align*}
\lim_{\mathbb{R}_{>0}\ni~ s\to 0^+} s^{-\frac{\lambda(h)}{\hslash}} \ib(\hslash,h, t_{\lambda}(s),[\vol]) =~& \int_{\mathbb{R}^{\ell}_{>0}} e^{-\frac{a_1+\cdots +a_{\ell}}{\hslash}} \prod_{k=1}^{\ell} a_k^{\frac{\beta^{\vee}_k(h)}{\hslash} - 1} da_1\cdots da_{\ell}\\
=~& \prod_{k=1}^{\ell} \hslash^{\frac{\beta_k^{\vee}(h)}{\hslash}}  \Gamma\left(\frac{\beta^{\vee}_k(h)}{\hslash}\right) \\
=~& \hslash^{-\frac{(2\rho^{\vee}-2\rho_P^{\vee})(h)}{\hslash}}\prod_{\alpha\in -(R^+\setminus R^+_P)}\Gamma\left(\frac{\alpha^{\vee}(h)}{\hslash}\right).
\end{align*}
\end{proof}
\begin{lemma}\label{spa}
For fixed $t\in\ZLp$ and $[\omega]\in\Bmodule$, we have
\[ \left|\left| e^{\frac{\EO^{\mirrormap(t)}}{\hslash}}\ib(\hslash,0,t,[\omega])\right|\right| ~\overset{\hslash\to 0}{=\joinrel=} ~O(\hslash^{\frac{\ell}{2}}),\]
where $\EO^{\mirrormap(t)}$ is the constant defined in Definition \ref{conjectureOevaluedef} and $\overset{\hslash \to 0}{=\joinrel=}O(\hslash^{\frac{\ell}{2}})$ means that there exist $\hslash_0, C\in\mathbb{R}_{>0}$ such that the expression is smaller than $C\hslash^{\frac{\ell}{2}}$ for all $0<\hslash<\hslash_0$.
\end{lemma}
\begin{proof}
By Lemma \ref{fLusztig}, $\fmt|_{\Xtp}$ becomes convex after the coordinate change $(x_1,\ldots,x_{\ell})\in\mathbb{R}^{\ell}\mapsto \XXp_{\iii}(t,e^{x_1},\ldots,e^{x_{\ell}})\in\Xtp$ ($\iii$ is a reduced decomposition of $w_P$). By Corollary \ref{fhascriticalpoint}, $\fmt|_{\Xtp}$ has a critical point. It follows that this critical point is unique, non-degenerate and is a global minimum point. By Lemma \ref{criticalvalue=conjectureOeigenvalue}, the critical value is equal to $\EO^{\mirrormap(t)}$. The result now follows from the well-known \textit{stationary phase approximation}. See e.g., \cite[Proposition 2.35]{SPA} for a proof. 
\end{proof}
Recall the vector bundle $\bundle$ defined in Definition \ref{vectorbundledef}, the mirror map $\mirrormap$ defined in Definition \ref{mirrormapdef} and the mirror isomorphism $\Mir$ from Theorem \ref{mirrorthm}. Note that $\mirrormap$ restricts to an isomorphism $\ZLpp\xrightarrow{\sim}\CCp^{I\setminus I_P}$.

\begin{definition}\label{sbdef}
Define a section $\sb$ of $\bundle|_{\mathbb{R}_{>0}\times \mathfrak{t}^{\vee}\times \CCp^{I\setminus I_P}}$ by
\[ \sb(\hslash,h,q) := \sum_{v\in W^P} \ib(\hslash,h,\mirrormap^{-1}(q),\Mir^{-1}(\sigma^v))\sigma_v .\]
It is well-defined by Lemma \ref{ibok}.
\end{definition}
\begin{lemma}\label{sbflat}
For every $i\in I\setminus I_P$, we have $\nabla^A_{\partial_{q_i}}\sb=0$.
\end{lemma}
\begin{proof}
We have 
\[ \nabla^A_{\partial_{q_i}} \sb = \sum_{v\in W^P} \left[\left(\frac{\partial}{\partial q_i}\ib(\hslash,h,\mirrormap^{-1}(q),\Mir^{-1}(\sigma^v))\right) \sigma_v + \ib(\hslash,h,\mirrormap^{-1}(q),\Mir^{-1}(\sigma^v))\nabla^A_{\partial_{q_i}}\sigma_v\right].\]
By a straightforward argument, it suffices to show 
\[ \frac{\partial}{\partial q_i}\ib(\hslash,h,\mirrormap^{-1}(q),\Mir^{-1}(\sigma^v)) = \ib(\hslash,h,\mirrormap^{-1}(q), \Mir^{-1}(\nabla^A_{\partial_{q_i}}\sigma^v)).\]
Since $\Mir$ intertwines $\nabla^A_{\partial_{q_i}}$ and $\nabla^B_{\partial_{t_i}}$ (Theorem \ref{mirrorthm}), it suffices to show 
\[ \frac{\partial}{\partial t_i}\ib(\hslash,h,t,[\omega]) = \ib\left(\hslash,h,t, \nabla^B_{\partial_{t_i}} \left[\omega\right] \right)\] 
for all $\omega\in \SSS[\hslash]\otimes\Omega^{top}(\X/\ZL)$. By the holomorphicity, we may assume $t\in\ZLp$. Then 
\begin{align*}
&\frac{\partial}{\partial t_i}\ib(\hslash,h,t,[\omega])\\ 
=~& \int_{\Xtp}  e^{-\fmt/\hslash} \gammamt^{h/\hslash} \left[\left(-\frac{1}{\hslash}\frac{\partial}{\partial t_i} \fmt + \frac{1}{\hslash}\frac{\partial}{\partial t_i}\gammamt^*\langle h,\mcf_T\rangle \right)\omega_{(-\hslash,h,t)}+ \frac{\partial}{\partial t_i} \omega_{(-\hslash,h,t)} \right]\\
=~& \ib\left(\hslash,h,t,\nabla^B_{\partial_{t_i}} \left[\omega\right] \right),
\end{align*}
where the last equality follows from \eqref{mirrorremarkeq1} in Definition \ref{Bconnectiondef}.
\end{proof}
\subsection{Proof of Theorem \ref{mainMS}}\label{Subsection-A=B}
Our goal is to prove 
\[\sa = \sb~\text{ on }~\{(\hslash,h)\in\mathbb{R}_{>0}\times\mathfrak{t}^{\vee}|~ h\in\hslash\DDD\}\times \CCp^{I\setminus I_P}, \]
where $\sa$ and $\sb$ are the flat sections defined in Definition \ref{sadef} and Definition \ref{sbdef} respectively.

Fix $\hslash\in \mathbb{R}_{>0}$. By Lemma \ref{iaok} and Lemma \ref{ibok}, the functions $\sa(\hslash,-,-)$ and $\sb(\hslash,-,-)$ are holomorphic on $\hslash\DDD\times\CCp^{I\setminus I_P}$. Hence, it suffices to prove that they agree on $\hslash\DDD\times S$ for some non-empty subset $S\subset\mathbb{R}^{I\setminus I_P}_{>0}$ whose closure contains an open subset. Let $\lambda\in\mathfrak{t}$ be a vector satisfying $\lambda(\alpha_i)\in\mathbb{Z}_{>0}$ (resp. $=0$) if $i\in I\setminus I_P$ (resp. $i\in I_P$). Define $q_{\lambda}:\mathbb{R}_{>0}\rightarrow\mathbb{R}^{I\setminus I_P}_{>0}$ by $q_{\lambda}(s):=(s^{\lambda(\alpha_i)})_{i\in I\setminus I_P}$. Define $S$ to be the union of $\image(q_{\lambda})$ over all possible $\lambda$. Then $\overline{S}=(0,1]^{I\setminus I_P}\cup [1,\infty)^{I\setminus I_P}$, and hence it suffices to prove that for every $\lambda$, 
\begin{equation}\label{sa=sbeq1}
\sa(\hslash,h,q_{\lambda}(s)) = \sb(\hslash,h,q_{\lambda}(s))\qquad h\in\hslash\DDD,~s\in\mathbb{R}_{>0}.
\end{equation}

Put $g(h,s):= \sa(\hslash,h,q_{\lambda}(s)) - \sb(\hslash,h,q_{\lambda}(s))$. By Lemma \ref{saflat} and Lemma \ref{sbflat}, we have
\begin{equation}\label{sa=sbeq2}
s\frac{\partial}{\partial s} g(h,s) + \frac{1}{\hslash} 
\left(\sum_{i\in I\setminus I_P}\lambda(\alpha_i)c_1^{T^{\vee}\times\mathbb{G}_m}(L_{\omega_i^{\vee}})\right)\star_{q_{\lambda}(s)} g(h,s) = 0.
\end{equation}

Define $\UUU$ to be the set of $h'\in\DDD$ satisfying
\begin{enumerate}
\item $\alpha^{\vee}(h')\ne 0$ whenever $\alpha\in R^+$; and

\item $\lambda(w_1^{-1}(h'))-\lambda(w_2^{-1}(h'))\not\in\mathbb{Z}_{>0}$ whenever $w_1, w_2\in W^P$.
\end{enumerate}
Note that $\UUU$ is connected. We want to apply Lemma \ref{Frobeniusmethodexistence} to 
\[\FFF:=H_{T^{\vee}\times\mathbb{G}_m}^{\bullet}(\GPd)|_{\{\hslash\}\times\hslash\UUU}\quad\text{and}\quad A(s):=-\frac{1}{\hslash}\left(\sum_{i\in I\setminus I_P}\lambda(\alpha_i)c_1^{T^{\vee}\times\mathbb{G}_m}(L_{\omega_i^{\vee}})\right)\star_{q_{\lambda}(s)}-.\]
By condition (1) above and localization, we can take the required global frame $\{v_{0,i}\}$ to be $\{v_{0,w}(h):=\pd[wP^{\vee}]\}_{w\in W^P}$. In this case, the eigenfunctions $\lambda_i$ are $h\mapsto \frac{\lambda(w^{-1}(h))}{\hslash}$, and hence condition (2) implies that Lemma \ref{Frobeniusmethodexistence} is indeed applicable. It follows that 
\[ g(h,s)= \sum_{w\in W^P} A_w(h) g_w(h,s)\qquad h\in\hslash\UUU, ~s\in \mathbb{R}_{>0},\]
where each $A_w$ is a holomorphic function on $\hslash\UUU$ and 
\[ g_w(h,s)= s^{\frac{\lambda(w^{-1}(h))}{\hslash}}(v_{0,w}(h) + v_{1,w}(h) s+ v_{2,w}(h)s^2+\cdots )\]
is a solution to \eqref{sa=sbeq2}. 

Define 
\[ \UUU^-:=  \{h'\in\UUU|~\real(\alpha^{\vee}(h'))<0~\text{ for all }\alpha\in R^+\}.\]
By our assumptions on $\DDD$ (see Definition \ref{DDDdef}) and $\lambda$, $\UUU$ is preserved by the $W$-action, and hence $W\cdot \UUU^-\subseteq\UUU$. Moreover, for every $w_1, w_2\in W^P$ and $h\in w_2(\hslash\UUU^-)$, we have 
\[ \lim_{s\to 0^+} s^{\frac{\lambda(w_1^{-1}(h))-\lambda(w_2^{-1}(h))}{\hslash}} = \left\{ 
\begin{array}{cc}
1 & w_1=w_2\\ [.5 em]
0 & \text{otherwise}
\end{array}
\right. .\]
Therefore, for every $w\in W^P$ and $h\in w(\hslash\UUU^-)$,
\begin{align*}
A_w(h) &= \lim_{s\to 0^+} s^{-\frac{\lambda(w^{-1}(h))}{\hslash}} \int_{\GPd} g(h,s)\\
&=\lim_{s\to 0^+} s^{-\frac{\lambda(w^{-1}(h))}{\hslash}}\left(\ia(\hslash,h,q_{\lambda}(s),1)-\ib(\hslash,h,(\mirrormap^{-1}\circ q_{\lambda})(s),[\vol])\right).
\end{align*}
It follows that, by Lemma \ref{ialimit} and Lemma \ref{iblimit}, we have $A_e(h)=0$ for all $h\in \hslash\UUU^-$. By Lemma \ref{iaWinv} and Lemma \ref{ibWinv}, we have, for every $w\in W^P$ and $h\in w(\hslash\UUU^-)$,
\begin{align*}
A_w(h) = ~&   \lim_{s\to 0^+} s^{-\frac{\lambda(w^{-1}(h))}{\hslash}} \left(\ia(\hslash,h,q_{\lambda}(s),1)-\ib(\hslash,h,(\mirrormap^{-1}\circ q_{\lambda})(s),[\vol])\right)\\
= ~& \lim_{s\to 0^+} s^{-\frac{\lambda(w^{-1}(h))}{\hslash}} \left(\ia(\hslash,w^{-1}(h),q_{\lambda}(s),1)-\ib(\hslash,w^{-1}(h),(\mirrormap^{-1}\circ q_{\lambda})(s),[\vol])\right)\\
=~ & A_e(w^{-1}(h)) \\
=~& 0 .
\end{align*}
Since $A_w$ is holomorphic, $\hslash\UUU$ is connected and $w(\hslash\UUU^-)$ is open in $\hslash\UUU$, it follows that $A_w(h)=0$ for all $h\in \hslash\UUU$. Therefore, $g(h,s)=0$ for all $h\in\hslash\UUU$ and $s\in \mathbb{R}_{>0}$. Since $\hslash\DDD$ is connected, $\hslash\UUU$ is open in $\hslash\DDD$, we conclude $g(h,s)=0$ for all $h\in\hslash\DDD$ and $s\in \mathbb{R}_{>0}$. This proves \eqref{sa=sbeq1} and thus Theorem \ref{mainMS}.

By taking $\int_{\GPd}-\cup y$, we obtain
\begin{corollary}\label{ia=ib}
For every $\hslash\in\mathbb{R}_{>0}$, $h\in\hslash\DDD$ , $t\in\ZLp$ and $y\in H_{T^{\vee}\times\mathbb{G}_m}^{\bullet}(\GPd)$, we have
\[ \ia(\hslash,h,\mirrormap(t),y) = \ib(\hslash,h,t,\Mir^{-1}(y)). \]
\hfill$\square$
\end{corollary}

This corollary will be used to prove Theorem \ref{mainG} in the next section.
\section{Proof of Gamma conjecture I for flag varieties} \label{Section-final}
Let $\bundle$ be the vector bundle defined in Definition \ref{vectorbundledef} and $\EO^q$ the number defined in Definition \ref{conjectureOevaluedef}.
\begin{definition}\label{principalasymptoticdef}
$~$ 
\begin{enumerate}
\item Define a connection $\nabla$ on $\bundle|_{\mathbb{R}_{>0}\times\{0\}\times\{1\}}$ by
\[ \nabla_{\partial_{\hslash}}:= \frac{\partial}{\partial\hslash} - \frac{1}{\hslash^2}(c_1(\GPd)\star_{q=1} - )+\frac{\mu}{\hslash},\]
where 
\[ \mu:= \sum_{k\geqslant 0} \left(\frac{k-\ell}{2}\right)\id_{H^k(\GPd)}\in\End(H^{\bullet}(\GPd))\quad (\ell:=\dim_{\mathbb{C}}\GPd).\]

\item (See \cite[Proposition 3.3.1]{GGI}.) Define 
\[\quad\qquad
\mathcal{A}_{\GPd}:=  \left\{ s\in \Gamma(\mathbb{R}_{>0};\bundle|_{\mathbb{R}_{>0}\times\{0\}\times\{1\}})\left|
\begin{array}{l}
\nabla_{\partial_{\hslash}}s=0~\text{ and }\\ [.7em]
\exists~ m\in\mathbb{Z}~,~ \left|\left|e^{\frac{\EO^{q=1}}{\hslash}} s(\hslash)\right|\right|~\overset{\hslash \to 0}{=\joinrel=}~O(\hslash^{m})
\end{array}
\right.\right\},\]
where $\overset{\hslash \to 0}{=\joinrel=}O(\hslash^{m})$ means that there exist $\hslash_0, C\in\mathbb{R}_{>0}$ such that the expression is smaller than $C\hslash^{m}$ for all $0<\hslash<\hslash_0$.
 
\end{enumerate}
\end{definition} 

By \cite[Corollary 3.6.9]{GGI}, Theorem \ref{mainG} follows from
\begin{proposition}\label{propositiona}
$\EO^{q=1}$ is an eigenvalue of $c_1(\GPd)\star_{q=1}-$ with multiplicity one.
\end{proposition}
\begin{proposition}\label{propositionb}
$\mathcal{A}_{\GPd}$ contains the section $S(\hslash,0,1)\left(\hslash^{-\mu}\hslash^{c_1}\widehat{\Gamma}_{\GPd}\right)$. (See Section \ref{Subsection-Aside}.)
\end{proposition} 
Proposition \ref{propositiona} follows from the proof of Proposition \ref{existSchubertpositive} (more precisely the verification of (2) therein). It corresponds to part (1) and (3) of \textit{Property $\mathcal{O}$} \cite[Definition 3.1.1]{GGI}, a property conjectured to be satisfied for arbitrary Fano manifolds \cite[Conjecture 3.1.2]{GGI}. For the case of $\GPd$, this conjecture has been proved by Cheong and Li \cite{CheongLi}. The proof presented here is an exposition of theirs, which relies on some arguments of Rietsch \cite{RietschJAMS}. See Remark \ref{existSchubertpositiveremark} for more details.

It remains to prove Proposition \ref{propositionb}. 

\bigskip

\begin{proof}[Proof of Proposition \ref{propositionb}]
It is well-known that $ \nabla_{\partial_{\hslash}}\left(S(\hslash,0,1)\left(\hslash^{-\mu}\hslash^{c_1}x\right)\right)=0$ for all $x\in H^{\bullet}(\GPd)$. See e.g., \cite[Proposition 2.4]{IZ}. By Definition \ref{sadef}, we have 
\[ S(\hslash,0,1)\left(\hslash^{-\mu}\hslash^{c_1} \widehat{\Gamma}_{\GPd}\right)= \hslash^{-\frac{\ell}{2}} \sa(\hslash,0,1).\]
Hence it remains to show
\[ \left|\left| e^{\frac{\EO^{q=1}}{\hslash}} \hslash^{-\frac{\ell}{2}}\sa(\hslash,0,1) \right|\right|~ \overset{\hslash\to 0}{=\joinrel=} ~O(\hslash^{m})\]
for some $m\in\mathbb{Z}$. We have $\sa(\hslash,0,1)=\sum_{v\in W^P}\ia(\hslash,0,1,\sigma^v)\sigma_v$. By Corollary \ref{ia=ib}, we have $\ia(\hslash,0,1,\sigma^v)=\ib(\hslash,0,1,\Mir^{-1}(\sigma^v))$. The estimate then follows from Lemma \ref{spa}.
\end{proof}

The proof of Theorem \ref{mainG} is complete.
\appendix
\section{Results on differential equations}\label{Subsection-frobenius-method}
We need the following two standard results. For the reader's convenience, we provide the proofs.

Let $\FFF$ be a holomorphic vector bundle on a complex manifold $Y$ and $J$ a finite set. Suppose we are given a family $\{A_{j,\nu}\}_{(j,\nu)\in J\times \mathbb{Z}_{\geqslant 0}^J}$ of holomorphic sections of $\End(\FFF)$ satisfying
\begin{enumerate}
\item the set $\{ (j,\nu)\in J\times \mathbb{Z}_{\geqslant 0}^J|~A_{j,\nu}\ne 0\}$ is finite; and

\item $[A_{j_1,\mathbf{0}},A_{j_2,\mathbf{0}}]=0$ for all $j_1, j_2\in J$.
\end{enumerate} 
Let $q=\{q_j\}_{j\in J}$ be a family of formal parameters. Consider the following system of differential equations
\begin{equation}\label{frobeniusmethodeq1}
\left(q_j\frac{\partial}{\partial q_j}-\sum_{\nu\in \mathbb{Z}_{\geqslant 0}^J} q^{\nu} A_{j,\nu}\right)v(y,q)=0,\quad j\in J,~ y\in Y.
\end{equation}
Here, $q^{\nu}:=\prod_{j\in J} q_j^{\langle \nu, e_j\rangle}$, where $e_j$ ($j\in J$) are the coordinate covectors.

Let 
\begin{equation}\label{frobeniusmethodeq1.5}
 S:= \left( \sum_{\nu\in \mathbb{Z}_{\geqslant 0}^J} q^{\nu}S_{\nu}\right)\circ \exp\left( \sum_{j\in J} (\log q_j)A_{j,\mathbf{0}}\right)\in\End(\FFF)[[q_j,\log q_j|~j\in J]],
\end{equation}
where each $S_{\nu}$ is a holomorphic section of $\End(\FFF)$.

\begin{lemma}\label{Frobeniusmethodconvergence} 
Suppose $Sx$ is a formal solution to \eqref{frobeniusmethodeq1} for all holomorphic section $x$ of $\FFF$. Then the formal power series $\sum_{\nu\in \mathbb{Z}_{\geqslant 0}^J} q^{\nu} S_{\nu}$ converges to a holomorphic section of $\End(\FFF)\times \mathbb{C}^J$ over $Y\times\mathbb{C}^J$ and $Sx$ is a numerical (multi-valued) solution to \eqref{frobeniusmethodeq1} for all $x$.
\end{lemma}
\begin{proof}
Since the problem is local in $Y$, it suffices to verify the convergence over $U\times \mathbb{C}^J$ for every open subset $U$ of $Y$ with compact closure. 

We have
\begin{align*}
q_j\frac{\partial Sx}{\partial q_j} & = \sum_{\nu\in \mathbb{Z}_{\geqslant 0}^J} q^{\nu}\left( \langle \nu,e_j\rangle S_{\nu} + S_{\nu}\circ A_{j,\mathbf{0}}\right)\circ\exp\left( \sum_{j\in J} (\log q_j)A_{j,\mathbf{0}}\right) x \\
\left( \sum_{\nu\in \mathbb{Z}_{\geqslant 0}^J} q^{\nu} A_{j,\nu}\right)Sx & = \sum_{\nu\in \mathbb{Z}_{\geqslant 0}^J} q^{\nu}\left(A_{j,\mathbf{0}}\circ S_{\nu} + \sum_{\substack{ \nu_1+\nu_2=\nu\\ \nu_1\neq 0}} A_{j,\nu_1}\circ S_{\nu_2} \right) \circ \exp\left( \sum_{j\in J} (\log q_j)A_{j,\mathbf{0}}\right)x.
\end{align*}
Since $Sx$ is a formal solution to \eqref{frobeniusmethodeq1} for all $x$, we have, for every $j\in J$ and $\nu\in \mathbb{Z}_{\geqslant 0}^J$,
\begin{equation}\label{Frobeniusmethodconvergenceeq1}
 \langle \nu,e_j\rangle S_{\nu} + S_{\nu}\circ A_{j,\mathbf{0}} - A_{j,\mathbf{0}}\circ S_{\nu} = \sum_{\substack{ \nu_1+\nu_2=\nu\\ \nu_1\neq 0}} A_{j,\nu_1}\circ S_{\nu_2} .
\end{equation}
Summing these equalities over $j$, we get 
\begin{equation}\label{Frobeniusmethodconvergenceeq2}
 \langle \nu,e\rangle S_{\nu} + S_{\nu}\circ A_{\mathbf{0}} - A_{\mathbf{0}}\circ S_{\nu} = \sum_{\substack{ \nu_1+\nu_2=\nu\\ \nu_1\neq 0}} A_{\nu_1}\circ S_{\nu_2},
\end{equation}
where $e:=\sum_{j\in J}e_j$ and $A_{\nu}:=\sum_{j\in J}A_{j,\nu}$. 

Now give $\FFF$ a Hermitian metric. Take an integer $N$ greater than the norm of the operator $X\mapsto X\circ A_{\mathbf{0}}(y) -A_{\mathbf{0}}(y)\circ X$ on $\End(\FFF_y)$ for all $y\in U$, and an integer $m\geqslant 0$ satisfying $A_{\nu}\neq 0\implies \nu\in [0,m]^J$. Put $M:=1+\sup_{\nu\neq \mathbf{0}, y\in U} ||A_{\nu}(y)||$, which is finite because $U$ has compact closure. Equation \eqref{Frobeniusmethodconvergenceeq2} implies 
\[ ||S_{\nu}(y)||\leqslant \frac{M\left(\sum_{\nu_1\in [0,m]^J\setminus \{\mathbf{0}\}}|| S_{\nu-\nu_1}(y)|| \right)}{\langle \nu,e\rangle -N}\]
for all $y\in U$ and $\nu\in\mathbb{Z}_{\geqslant 0}^J$ satisfying $\langle \nu,e\rangle > N$. By induction, there exists $C>0$ such that 
\[ ||S_{\nu}(y)|| \leqslant C\frac{M^{\langle \nu,e\rangle}}{ \left(\left\lfloor \frac{\langle \nu,e\rangle -N}{(m+1)^{|J|}}\right\rfloor\right)!}\]
whenever $y\in U$ and $\langle \nu,e\rangle >N$. It follows that there exists a polynomial $f$ of degree $|J|-1$ such that for every $R>1$ and $\{q_j\}_{j\in J}\in \mathbb{C}^J$ satisfying $|q_j|\leqslant R$ for all $j$, the series $\sum_{\langle \nu,e\rangle >N} |q^{\nu}|\sup_{y\in U}||S_{\nu}(y)||$ is bounded by $\sum_{k=0}^{\infty}\frac{f(k)}{k!} (MR)^{N+(m+1)^{|J|}(k+1)}$, which is finite. This verifies the convergence of the series $\sum_{\nu\in\mathbb{Z}_{\geqslant 0}^{J}}q^{\nu}S_{\nu}$.

Finally, that $Sx$ is a numerical solution follows from \eqref{Frobeniusmethodconvergenceeq1}. The proof is complete.
\end{proof}
Now we restrict ourselves to the case $J=\{j_0\}$. Put $s:= q_{j_0}$ and $A_k:=A_{j_0,(k)}$ for $k\in\mathbb{Z}_{\geqslant 0}$. The system \eqref{frobeniusmethodeq1} becomes the differential equation 
\begin{equation}\label{frobeniusmethodeq3}
\left(s\frac{\partial}{\partial s}-\sum_{k\geqslant 0} s^kA_k\right)v(y,s) =0,\quad y\in Y.
\end{equation}
Suppose $\FFF$ has a global frame $\{v_{0,1},\ldots,v_{0,N}\}$ such that for every $i=1,\ldots,N$,
\begin{equation}\nonumber\label{frobeniusmethodeq0}
A_0v_{0,i} = \lambda_i v_{0,i}
\end{equation}
for some holomorphic function $\lambda_i$ on $Y$. 
\begin{lemma}\label{Frobeniusmethodexistence} 
Suppose $\lambda_{i_1}(y)-\lambda_{i_2}(y)\not\in \mathbb{Z}_{>0}$ for all $y\in Y$ and $1\leqslant i_1,i_2\leqslant N$. Then for every $i$, there exists a section $v_i$ of $\FFF\times\mathbb{R}_{>0}$ over $Y\times \mathbb{R}_{>0}$ satisfying
\begin{enumerate}
\item it is holomorphic in $y\in Y$ and smooth in $s\in \mathbb{R}_{>0}$;

\item it is a numerical solution to \eqref{frobeniusmethodeq3}; and

\item it has an expansion
\begin{equation}\label{frobeniusmethodeq2}
 v_i(y,s) = s^{\lambda_i(y)}(v_{0,i}(y)+v_{1,i}(y)s+v_{2,i}(y)s^2+\cdots )\qquad y\in Y, s\in \mathbb{R}_{>0}
\end{equation}
for some holomorphic sections $v_{1,i}, v_{2,i},\ldots $ on $Y$.
\end{enumerate}
Moreover, $\{v_1,\ldots,v_N\}$ forms a basis of the space of solutions to \eqref{frobeniusmethodeq3} over the ring of holomorphic functions on $Y$.
\end{lemma}
\begin{proof}
The condition on $\lambda_i$ implies that the operator $X\mapsto k X+ X\circ A_{0}-A_{0}\circ X$ is invertible for all positive integer $k$. It follows that we can solve the recurrence relation \eqref{Frobeniusmethodconvergenceeq1} to get $S$, starting with $S_0=\id$. The desired solution $v_i$ is just the restriction of $Sv_{0,i}$ to $Y\times \mathbb{R}_{>0}$. Properties (1), (2) and (3) follow immediately from Lemma \ref{Frobeniusmethodconvergence}.

It remains to verify that $\{v_1,\ldots,v_N\}$ is a basis of the space of solutions to \eqref{frobeniusmethodeq3}. Let $y\in Y$. By the uniqueness result in ODE theory, it suffices to show that $\{v_1(y,s_0),\ldots,v_N(y,s_0)\}$ is a basis of $\FFF_y$ for some $s_0\in\mathbb{R}_{>0}$. Let $\ol{v}_i(s):=s^{-\lambda_i(y)}v_i(y,s)$. Then $\ol{v}_i(0):=\lim_{s\to 0^+} \ol{v}_i(s)=v_{0,i}(y)$. It follows that $\{\ol{v}_1(0),\ldots,\ol{v}_N(0)\}$ is a basis of $\FFF_y$, and hence the same is true for $\{\ol{v}_1(s),\ldots,\ol{v}_N(s)\}$ whenever $s\in \mathbb{R}_{>0}$ is small enough. Since $v_i(y,s)$ is a non-zero scalar multiple of $\ol{v}_i(s)$, the result follows. 
\end{proof}
\section{Proofs from Section \ref{Section-bmodel}}\label{Subsection-proofs-from-bmodel}
\begin{proof}[Proof of Lemma \ref{gammaLusztig}]
Following \cite{BZ}, we write $x=[x]_+[x]_0[x]_-$ for $x\in UTU^-$ with $[x]_+\in U$, $[x]_0\in T$ and $[x]_-\in U^-$. By recalling the definitions of $\XXp_{\iii}$, $\eta^{w_P}$ and $\ttp_{\iii}$ (Definition \ref{toruschartforgeometriccrystaldef}, Definition \ref{twistmapdef} and Definition \ref{toruschartforBruhatdef}(2)), we see that 
\begin{align*}
(\gammam\circ\XXp_{\iii})(t,a_1,\ldots,a_{\ell})  & = t \cdot[\ol{w_P}~ \iota(\ttp_{\iii}(a_1,\ldots,a_{\ell}))]_0 \\
& = t\cdot [\ol{w_P}~ x_{i_{\ell}}(a_{\ell})\cdots x_{i_1}(a_1)]_0 .
\end{align*}

It suffices to prove the following more general result: For every reduced decomposition $\mathbf{j}=(j_m,\ldots,j_1)$ of an element $w\in W$ (note the unusual ordering) and $b_1,\ldots,b_m\in \mathbb{G}_m$, we have
\[ [\ol{w}~ x_{j_1}(b_1)\cdots x_{j_m}(b_m)]_0 = \prod_{k=1}^m \gamma_k^{\vee}(b_k),\]
where $\gamma_k^{\vee}:= - s_{j_m}\cdots s_{j_{k+1}}(\alpha_{j_k}^{\vee})$\footnote{$\ol{w}~ x_{j_1}(b_1)\cdots x_{j_m}(b_m)\in UTU^-$ because $x_{j_1}(b_1)\cdots x_{j_m}(b_m)\in B^-w^{-1}B^-$.}.

We prove it by induction on $m=\ell(w)$. Write $w=s_{j_m}w_1$. Then $(j_{m-1},\ldots,j_1)$ is a reduced decomposition of $w_1$. Put $z:=\ol{w_1}~ x_{j_1}(b_1)\cdots x_{j_{m-1}}(b_{m-1})$ and write $[z]_-= y_{j_m}(c)u$, where $c:=\liecharacter_{j_m}([z]_-)$. Then
\begin{align*}
&\ol{w}~ x_{j_1}(b_1)\cdots x_{j_m}(b_m) \\
=~&  \ol{s_{j_m}}~ z~ x_{j_m}(b_m)\\
=~ & \ol{s_{j_m}}~ [z]_+[z]_0[z]_- x_{j_m}(b_m)\\
=~& \left( \ol{s_{j_m}}~[z]_+~ \ol{s_{j_m}}^{-1}\right) s_{j_m}([z]_0)\left( \ol{s_{j_m}} ~y_{j_m}(c)~
x_{j_m}(b_m)\right) \left(x_{j_m}(b_m)^{-1} ~u~ x_{j_m}(b_m)\right).
\end{align*}
Observe that $x_{j_1}(b_1)\cdots x_{j_{m-1}}(b_{m-1})\in B^-w_1^{-1}B^-$, and hence $z\in w_1B^-w_1^{-1}B^-=U(w_1)B^-$, where $U(w_1):= U\cap w_1U^-w_1^{-1}$. Since $\ell(w)=\ell(w_1)+1$, we have $\character_{j_m}|_{U(w_1)}\equiv 0$, and hence $\ol{s_{j_m}}~[z]_+~ \ol{s_{j_m}}^{-1}\in U$. It is clear that $x_{j_m}(b_m)^{-1} ~u~ x_{j_m}(b_m)\in U^-$ by the definition of $u$. By playing with $2\times 2$ matrices, we see that $[\ol{s_{j_m}} ~y_{j_m}(c)~
x_{j_m}(b_m)]_0 = \alpha_{j_m}^{\vee}(b_m^{-1})$. Therefore, by induction,
\begin{align*}
[\ol{w}~ x_{j_1}(b_1)\cdots x_{j_m}(b_m)]_0 &~= ~s_{j_m}([z]_0)\alpha_{j_m}^{\vee}(b_m^{-1}) \\
&~=~\left( \prod_{k=1}^{m-1} s_{j_m}\left((- s_{j_{m-1}}\cdots s_{j_{k+1}}(\alpha_{j_k}^{\vee}))(b_k)\right) \right) \cdot (-\alpha_{j_m}^{\vee})(b_m) \\
&~= ~\prod_{k=1}^m \gamma_k^{\vee}(b_k)
\end{align*}
as desired. 
\end{proof}
\bigskip
\begin{proof}[Proof of Lemma \ref{omegaLusztig}]
By the definitions of $\vol$ and $\XXp_{\iii}$, it suffices to show that the pull-back of $\omega_{\UU}$ (see Definition \ref{volumeformdef}) by the composite morphism
\[ \mathbb{G}_m^{\ell} \xhookrightarrow{~\ttp_{\iii}~}  U^{w_P}\xrightarrow[\simeq]{~\eta^{w_P}~}  B^-_{w_P} \xrightarrow[\simeq]{\zeta~:~ x~\mapsto~ x^{-1}P}   \UU\]
is equal to a non-zero scalar multiple of $\frac{da_1\wedge\cdots\wedge da_{\ell}}{a_1\cdots a_{\ell}}$.

The following arguments are due to Lam \cite[Proposition 2.11]{Lam}, and we provide the details for the reader's convenience. By \cite[Proposition 7.2]{Rietsch}, there exists a volume form $\omega_{U^{w_P}}$ on $U^{w_P}$ such that 
\begin{equation}\label{omegaLusztigeq1}
(\ttp_{\iii'})^*\omega_{U^{w_P}} = \pm \frac{da_1\wedge\cdots\wedge da_{\ell}}{a_1\cdots a_{\ell}}
\end{equation}
for all reduced decomposition $\iii'$ of $w_P$. Hence it suffices to show that $\omega':= \left((\eta^{w_P})^{-1}\circ \zeta^{-1}\right)^*\omega_{U^{w_P}}$ is a non-zero scalar multiple of $\omega_{\UU}$. By \cite[Lemma 2.10]{Lam}, \eqref{omegaLusztigeq1} implies that $\omega'$ has at worst simple pole along every irreducible component of the boundary divisor $(G/P)\setminus \UU$. It follows that the rational function $\omega'/\omega_{\UU} $ on $G/P$ has no poles and hence must be a non-zero constant. The result follows.
\end{proof}

\bigskip
\begin{proof}[Proof of Lemma \ref{gmactionwelldefined}]
Let $c\in\mathbb{G}_m$ and $x\in\X\setminus\{\varphi_i=0\}$. Write $x=u\gammam(x)$ with $u\in U^-$, and then $u=u'y_i(\liecharacter_i(u))=u'y_i(\varphi_i(x))$. We have 
\begin{align*}
e_i^c(x) =~& \left( x_i\left(\frac{c-1}{\varphi_i(x)}\right)\cdot u'\cdot  x_i\left(\frac{-c+1}{\varphi_i(x)}\right) \cdot y_i(c^{-1}\varphi_i(x)) \right)\times \\
& \left( y_i(-c^{-1}\varphi_i(x)) \cdot x_i\left(\frac{c-1}{\varphi_i(x)}\right)\cdot y_i(\varphi_i(x))\cdot \gammam(x)  \cdot x_i\left(\frac{c^{-1}-1}{\varepsilon_i(x)}\right)\right).
\end{align*}
It is straightforward to see that the first factor in the last expression lies in $U^-$ and the second factor is equal to $\alpha^{\vee}_i(c)\cdot\gammam(x)\in T$. Observe that $\liecharacter_i$ vanishes at $x_i\left(\frac{c-1}{\varphi_i(x)}\right)\cdot u'\cdot  x_i\left(\frac{-c+1}{\varphi_i(x)}\right)$. It follows that 
\[ \varphi_i\circ e_i^c = c^{-1}\cdot \varphi_i,\quad \gammam\circ e_i^c = \alpha^{\vee}_i(c)\cdot\gammam\]
and 
\[ \varepsilon_i\circ e_i^c = (\varphi_i\circ e_i^c)\cdot (\alpha_i\circ\gammam\circ e_i^c)= c\cdot \varepsilon_i. \]
This shows that $e_i$ is regular on $\mathbb{G}_m\times(\X\setminus\{\varphi_i=0\})$ and takes values in $\X\setminus\{\varphi_i=0\}$.

It remains to show that $c\mapsto e^c_i$ defines a $\mathbb{G}_m$-action. Let $c_1,c_2\in \mathbb{G}_m$ and $x\in \X\setminus\{\varphi_i=0\}$. We have
\begin{align*}
(e_i^{c_1}\circ e_i^{c_2})(x)=~& x_i\left(\frac{c_1-1}{\varphi_i(e_i^{c_2}(x))}\right) \cdot x_i\left(\frac{c_2-1}{\varphi_i(x)}\right)\cdot x\cdot x_i\left(\frac{c_2^{-1}-1}{\varepsilon_i(x)}\right)\cdot x_i\left(\frac{c_1^{-1}-1}{\varepsilon_i(e_i^{c_2}(x))}\right)  \\
=~& x_i\left(\frac{\frac{c_1-1}{c_2^{-1}} +c_2-1}{\varphi_i(x)}\right)\cdot x\cdot x_i\left(\frac{c_2^{-1}-1+\frac{c_1^{-1}-1}{c_2}}{\varepsilon_i(x)}\right)\\
=~& x_i\left(\frac{c_1c_2-1}{\varphi_i(x)}\right)\cdot x\cdot x_i\left(\frac{(c_1c_2)^{-1}-1}{\varepsilon_i(x)}\right)\\
=~& e_i^{c_1c_2}(x). 
\end{align*}
The proof is complete.
\end{proof}
\bigskip

\begin{proof}[Proof of Lemma \ref{Wpreserveomega}]
It suffices to show that for every $i\in I$, $s_i^*\vol=-\vol$ holds fiberwise. Let $t\in\ZL$. Put $\Xt:=\pim^{-1}(t)$, $\gammamt:=\gammam|_{\Xt}$, $e_{i,t}^c:=e_i^c|_{\Xt}$ and $s_{i,t}:=s_i|_{\Xt}$. Denote by $\omega_t\in\Omega^{top}(\Xt)$ the restriction of $\vol$ to $\Xt$. We have an isomorphism $\zeta_t:\Xt\xrightarrow{\sim}\UU$ defined by $x\mapsto x^{-1}P$, and by definition, we have $\omega_t=\zeta_t^*\omega_{\UU}$, where $\UU$ and $\omega_{\UU}$ come from Definition \ref{volumeformdef}.

\begin{lemma}\label{omegaisweightvector}
$\omega_t$ is a weight vector with respect to the $\mathbb{G}_m$-action $c\mapsto e^c_{i,t}$.
\end{lemma}
\begin{proof}
Put $\Xti:=\Xt\setminus \{\varphi_{i}=0\}$ and $V:=\Omega^{top}(\Xti)$. Define $S$ to be the set of $\omega\in V$ that are nowhere vanishing. Note $\omega_{t}|_{\Xti}\in S$. We are going to prove that every element of $S$ is a weight vector. The $\mathbb{G}_m$-action $c\mapsto e^c_{i,t}$ induces a linear $\mathbb{G}_m$-action on $V$ preserving $S$. There exists a sequence of sub-$\mathbb{G}_m$-modules
\[ 0=V_0\subseteq V_1\subseteq V_2\subseteq \cdots \subseteq V\]
such that each $V_n$ is finite dimensional and $V=\bigcup_{n=0}^{\infty}V_n$. We are done if we can show that for every $n$, $V_n\cap S$ is contained in a finite union of one-dimensional vector subspaces.

Consider the map $\mathcal{O}(\Xti)\rightarrow V$ defined by $\varphi\mapsto \varphi\omega_{t}|_{\Xti}$. It is an isomorphism of vector spaces (a priori not necessarily of $\mathbb{G}_m$-modules) sending $\mathcal{O}(\Xti)^{\times}$ to $S$. Since $\Xt$ is isomorphic to $\UU$, which is a divisor complement of a Schubert cell, $\mathcal{O}(\Xti)$ is isomorphic to the localization of a polynomial algebra $A:=\mathbb{C}[x_1,\ldots,x_N]$ by a non-zero polynomial $f$. Our goal becomes showing that every finite dimensional vector subspace $W$ of $A[f^{-1}]$ contains only finitely many non-homothetic units. By multiplying a power of $f$, we may assume $W\subseteq A$. Observe that every unit of $A[f^{-1}]$ lying in $A$ is of the form $c f_1^{e_1}\cdots f_k^{e_k}$, where $c\in \mathbb{C}^{\times}$, $f_1,\ldots, f_k$ are the irreducible divisors of $f$ and $e_1,\ldots,e_k\in\mathbb{Z}_{\geqslant 0}$. Up to homothety, there are only finitely many of them that lie in $W$, since the exponents $e_i$ are bounded by $\sup_{g\in W}\deg g$, which is finite. We are done.
\end{proof}
 
\begin{lemma}\label{gammanonconst}
$\alpha_i\circ \gammamt$ is non-constant.
\end{lemma}
\begin{proof}
Suppose $\alpha_i\circ \gammamt$ is constant. By Lemma \ref{gammaLusztig}, we have $\beta_k^{\vee}(\alpha_i)=0$ for all $1\leqslant k\leqslant\ell$. By Lemma \ref{betadescription}, which says $\{\beta_k^{\vee}\}_{k=1}^{\ell}=\{\alpha^{\vee}\}_{\alpha\in -(R^+\setminus R^+_P)}$, we have $\alpha^{\vee}(\alpha_i)=0$ for all $\alpha\in R\setminus R_P$.

Define 
\[ J':=\{j\in I|~\alpha^{\vee}(\alpha_j)=0~\text{ for all }\alpha\in R\setminus R_P\}\]
and $J_0\subseteq J_1\subseteq J_2\subseteq \cdots \subseteq I$ inductively by $J_0:=\{i\}$ and 
\[ J_{r+1}:=\{j\in I|~\alpha_j^{\vee}(\alpha_{j'})\neq 0~\text{ for some }j'\in J_r\}\qquad r\geqslant 0.\]
Notice that $J'\subseteq I_P$. Since $G$ is simple, we have $J_r=I$ for sufficiently large $r$. We are done if we can show that for every $r\geqslant 0$,
\[ J_r\subseteq J'\Longrightarrow J_{r+1}\subseteq J'\]
because this will force $I=J'$, and hence $I=I_P$ which we have excluded at the beginning (Section \ref{Subsection-Bmodelnotation}).

Suppose $J_r\subseteq J'$ for some $r\geqslant 0$. Let $j\in J_{r+1}$. Suppose $j\not\in J'$. Then there exists $\alpha\in R\setminus R_P$ such that $\alpha^{\vee}(\alpha_j)\ne 0$, and hence there exists $\beta\in R$ such that $\beta^{\vee}$ is equal to $\alpha^{\vee}+\alpha^{\vee}_j$ or $\alpha^{\vee}-\alpha^{\vee}_j$. On the other hand, since $j$ belongs to $J_{r+1}$, there exists $j'\in J_r$ such that $\alpha^{\vee}_j(\alpha_{j'})\neq 0$. By the assumption $J_r\subseteq J'$, we have $j'\in J'$, and hence $\alpha^{\vee}(\alpha_{j'})=0$. It follows that $\beta^{\vee}(\alpha_{j'})\neq 0$, and hence $\beta\in R_P$ (by $j'\in J'$). Thus, we must have $\alpha_j\in R\setminus R_P$, which implies $\alpha^{\vee}_j(\alpha_{j'})=0$ (by $j'\in J'$ again), a contradiction.
\end{proof}

Let $k\in\mathbb{Z}$ be the weight of $\omega_{t}$ (Lemma \ref{omegaisweightvector}). By the equality $\gammam\circ e^c_{i} =\alpha^{\vee}_i(c)\cdot \gammam$ (see the proof of Lemma \ref{gmactionwelldefined}), Lemma \ref{gammanonconst} and a straightforward computation, we have $s_{i,t}^*\omega_{t}= -(\alpha_i\circ\gammamt)^{-k} \omega_{t}$. Hence it remains to prove $k=0$. By the proof of Lemma \ref{omegaLusztig}, $\omega_{\UU}$ is \textit{dlog}, i.e., there exist rational functions $\varphi_1,\ldots,\varphi_{\ell}$ on $\UU$ such that $\omega_{\UU}=\frac{d\varphi_1\wedge\cdots \wedge d\varphi_{\ell}}{\varphi_1\cdots \varphi_{\ell}}$. Since $s_{i,t}$ is a birational equivalence on $\Xt$, $\omega'_{\UU}:= (\zeta_t\circ s_{i,t}\circ\zeta_t^{-1})^*\omega_{\UU}$ is also dlog. By \cite[Lemma 2.10]{Lam}, $\omega'_{\UU}$ has at worst simple pole along every irreducible component of the boundary divisor $(G/P)\setminus \UU$. It follows that $-(\alpha_i\circ\gammamt\circ\zeta_t^{-1})^{-k}=\omega'_{\UU}/\omega_{\UU}$ has no poles along these irreducible components. Since $\gammamt$ is regular on $\Xt$, $(\alpha_i\circ\gammamt)^{-k}$ must be constant. By Lemma \ref{gammanonconst}, we conclude $k=0$ as desired.
\end{proof}
\bigskip 
\begin{proof}[Proof of Lemma \ref{Wpreservepositivepart}]
It suffices to show that for every $i\in I$, $\Xp$ lies in the domain of definition of $s_i$ and is preserved by it. This will follow if we can verify the statement with $s_i$ replaced by $e_i^c$ for $c\in \mathbb{R}_{>0}$, since $(\alpha_i\circ\gammam)(\Xp)\subseteq \mathbb{R}_{>0}$ by Lemma \ref{gammapreservepositivepart}.

Let $x\in \Xp$. Take a reduced decomposition $\iii=(i_1,\ldots,i_{\ell})$ of $w_P$. By Lemma \ref{toruschartforBruhatpositive} and the definition of $\Xp$ (Definition \ref{positivepartdef}), there exist $t\in\ZLp$ and $a_1,\ldots,a_{\ell}\in\mathbb{R}_{>0}$ such that 
\[ x=t\cdot \ttm_{\iii}(a_1,\ldots,a_{\ell}) = t\cdot x_{-i_1}(a_1)\cdots x_{-i_{\ell}}(a_{\ell}).\]
The last expression is equal to $y_{i_1}(a'_1)\cdots y_{i_{\ell}}(a'_{\ell})\cdot t'$ for some $t'\in T_{>0}$ and $a'_1,\ldots, a'_{\ell}\in\mathbb{R}_{>0}$.

Let $i\in I$. Define $K:=\{1,\ldots,\ell\}$ and $K_i:=\{k\in K|~i_k=i\}$. 
\begin{lemma}\label{Wpreservepositivepartlemma}
$K_i\neq \emptyset$.
\end{lemma}
\begin{proof}
Let $\{\omega_j\}_{j\in I}$ be the dual basis of $\{\alpha^{\vee}_j\}_{j\in I}$. Recall the rays generated by these vectors are the edges of the dominant Weyl chamber $\Lambda$. Suppose $K_i=\emptyset$. Then we have $w_P\omega_i=\omega_i$, or equivalently
\begin{equation}\label{Wpreservepositivepartlemmaeq1}
w_0^P\omega_i=w_0\omega_i.
\end{equation}
There exists $i^{\star}\in I$ such that $w_0\omega_i=-\omega_{i^{\star}}$. By assumption, we have $I_P\neq I$ (see Section \ref{Subsection-Bmodelnotation}). Pick an element $j\in I\setminus I_P$. Equality \eqref{Wpreservepositivepartlemmaeq1} implies that $\omega_j=w_0^P\omega_j$ and $-\omega_{i^{\star}}=w_0^P\omega_i$ generate two edges of the Weyl chamber $w_0^P\cdot \Lambda$. But this contradicts the well-known fact that the angle between any two edges of a Weyl chamber in an irreducible root system must be acute.
\end{proof}
 
Now let $c\in\mathbb{R}_{>0}$. Observe that $\varphi_i(x)=\liecharacter_i(y_{i_1}(a'_1)\cdots y_{i_{\ell}}(a'_{\ell}))=\sum_{k\in K_i}a'_k$ and $\varepsilon_i(x)= \alpha_i(t')\varphi_i(x)=\alpha_i(t')\sum_{k\in K_i}a'_k$. It follows that, by Lemma \ref{Wpreservepositivepartlemma}, $\varphi_i(x), \varepsilon_i(x)>0$, and hence $e^c_i(x)$ is well-defined. Using the identities
\[ x_i(A)\cdot y_i(B) = y_i\left(\frac{B}{1+AB}\right)\cdot \alpha_i^{\vee}(1+AB)\cdot x_i\left(\frac{A}{1+AB}\right)\]
and 
\[ x_i(A)\cdot y_j(B) = y_j(B)\cdot x_i(A)\qquad i\ne j,\]
we obtain
\begin{align}\label{Wpreservepositiveparteq1}
e_i^c(x) =~& x_i\left(\frac{c-1}{\varphi_i(x)}\right)\cdot x \cdot x_i\left(\frac{c^{-1}-1}{\varepsilon_i(x)}\right)\nonumber \\ 
=~& x_i\left(\frac{c-1}{\varphi_i(x)}\right)\cdot y_{i_1}(a'_1)\cdots y_{i_{\ell}}(a'_{\ell})\cdot t'\cdot x_i\left(\frac{c^{-1}-1}{\varepsilon_i(x)}\right) \\
=~& \left(\prod_{k\in K} y_{i_k}(a''_k)\alpha_i^{\vee}(b_k)\right)\cdot t'\cdot x_i\left(\frac{c_{\ell}}{\alpha_i(t')}+\frac{c^{-1}-1}{\alpha_i(t')\sum_{k\in K_i}a'_k}\right),
\end{align}
where 
\[ a''_k:= 
\left\{
\begin{array}{cc}
\dfrac{a'_k}{1+a'_k c_{k-1}} & k\in K_i\\ [1.5em]
a'_k&\text{otherwise}
\end{array}
\right.
,\quad
b_k:= 
\left\{
\begin{array}{cc}
1+a'_k c_{k-1} & k\in K_i\\ [.8em]
1&\text{otherwise}
\end{array}
\right.~,\]
\[ c_0:=\frac{c-1}{\sum_{k\in K_i}a'_k}\quad\text{ and }\quad c_k:=\left\{
\begin{array}{cc}
\dfrac{c_{k-1}}{1+a'_k c_{k-1}} & k\in K_i\\ [1.2em]
c_{k-1}&\text{otherwise}
\end{array}
\right. .\]

For $k\in K$, define $A_{\star k}:=\sum_{s\in K_i, s\star k} a'_s$ for $\star\in\{ <,\leqslant,>,\geqslant \}$. By induction, we have
\[ c_k = \frac{c-1}{cA_{\leqslant k}+A_{>k}}\qquad k\in K,\]
and hence
\[ a''_k = \frac{a'_k (cA_{<k}+A_{\geqslant k})}{c A_{\leqslant k}+A_{>k}}>0\quad\text{ and }\quad  b_k = \frac{c A_{\leqslant k} +A_{>k}}{cA_{<k}+A_{\geqslant k}} > 0\qquad k\in K_i.\]
In particular, we have $c_{\ell}=\frac{c-1}{c\sum_{k\in K_i}a'_k}$, and hence $\frac{c_{\ell}}{\alpha_i(t')}+\frac{c^{-1}-1}{\alpha_i(t')\sum_{k\in K_i}a'_k}=0$. Therefore, by \eqref{Wpreservepositiveparteq1}, we can write $e_i^c(x)=t''\cdot x_{-i_1}(a'''_1)\cdots x_{-i_{\ell}}(a'''_{\ell})$ for some $t''\in T_{>0}$ and $a'''_1,\ldots,a'''_{\ell}\in\mathbb{R}_{>0}$. But since $\pim\circ e^c_i=\pim$ (obvious), we have $t''=t\in\ZLp$, and hence $e_i^c(x)\in\Xp$, by Lemma \ref{toruschartforBruhatpositive}, as desired.
\end{proof}
\section{Proofs from Section \ref{Section-flatsections}}\label{Subsection-proofs-from-flatsections}
\begin{proof}[Proof of Lemma \ref{gammahatlemma}]
The following proof works for any reasonable $T^{\vee}$-varieties, $T^{\vee}$-equivariant vector bundles and formal power series whose radius of convergence is positive.

Define 
\[ g(x):=\log \Gamma(1+x)=b_1x+b_2x^2+\cdots .\]
Note that it has positive radius of convergence because both $\log(1+y)$ and $\Gamma(1+x)$ do. In particular, there exists $\rho>0$ such that 
\begin{equation}\label{gammahatlemmaproofeq1}
\lim_{k\to \infty}|b_k|\rho^k = 0.
\end{equation}
Introduce formal variables $x_1,\ldots,x_{\ell}$. Then
\[ \Gamma(1+x_1)\cdots \Gamma(1+x_{\ell})=\exp\left(\sum_{k=1}^{\infty}b_k\left(\sum_{i=1}^{\ell}x_i^k\right)\right)\in \mathbb{C}[[x_1,\ldots,x_{\ell}]].\]
Hence it suffices to show that the formal power series $\sum_{k=1}^{\infty}b_k\left(\sum_{i=1}^{\ell}\delta_i^k\right)$ defines a holomorphic section of $\bundlez$ on an open neighbourhood of $0\in\mathfrak{t}^{\vee}$. (Recall $\delta_1,\ldots,\delta_{\ell}$ are the $T^{\vee}$-equivariant Chern roots of the tangent bundle of $\GPd$.)

For $\nu=(\nu_i)_{i=1}^{\ell}\in\mathbb{Z}_{\geqslant 0}^{\ell}$, define $|\nu|_1:=\sum_{i=1}^{\ell}\nu_i$ and $|\nu|_2:=\sum_{i=1}^{\ell}i\nu_i$. We can write $\sum_{i=1}^{\ell}x_i^k=\sum_{\nu\in N_k}c_{\nu}s_1^{\nu_1}\cdots s_{\ell}^{\nu_{\ell}}$, where 
\begin{itemize}
\item $N_k:=\{\nu\in\mathbb{Z}_{\geqslant 0}^{\ell}|~|\nu|_2=k\}$; and

\item $s_j$ is the $j$-th elementary symmetric polynomial in $x_1,\ldots,x_{\ell}$.
\end{itemize}
It is known that $c_{\nu}=(-1)^{k+|\nu|_1}\frac{k}{|\nu|_1}\cdot \frac{(|\nu|_1)!}{(\nu_1)!\cdots (\nu_{\ell})!}$. Observe that $\frac{(|\nu|_1)!}{(\nu_1)!\cdots (\nu_{\ell})!}\leqslant (~\!\underbrace{1+\cdots+1}_{\ell~\text{ 1's}}~\!)^{\nu_1+\cdots+\nu_{\ell}}=\ell^{|\nu|_1}$, and hence
\begin{equation}\label{gammahatlemmaproofeq2}
|c_{\nu}|\leqslant \frac{k\cdot \ell^{|\nu|_1}}{|\nu|_1}.
\end{equation}

Let $\nu\in \mathbb{Z}_{\geqslant 0}^{\ell}$ and $y\in H^{\bullet}_{T^{\vee}}(\GPd)$ be a homogeneous element. Define
\[ \mathcal{I}_{\nu,y}:=\int_{\GPd}c_1^{\nu_1}\cup\cdots\cup c_{\ell}^{\nu_{\ell}}\cup y \in H_{T^{\vee}}^{2|\nu|_2+\deg y -2\ell}(\pt), \]
where $c_j:=c_j^{T^{\vee}}(\mathcal{T}_{\GPd})$. Put $d(y):=\frac{1}{2}\deg y-\ell$. Write $\mathcal{I}_{\nu,y}=\sum_{\eta\in H_{|\nu|_2+d(y)}} d^{\eta}_{\nu,y}h_1^{\eta_1}\cdots h_r^{\eta_r}$, where $H_m:= \left\{\eta=(\eta_j)_{j=1}^r\in\mathbb{Z}_{\geqslant 0}^r\left|~ |\eta|_1:=\sum_{j=1}^r\eta_j=m\right.\right\}$ and $h_1,\ldots,h_r$ are the equivariant parameters. Then $d^{\eta}_{\nu,y}=\frac{\partial_{h_1}^{\eta_1}\cdots \partial_{h_r}^{\eta_r}\mathcal{I}_{\nu,y}}{(\eta_1)!\cdots (\eta_r)!}$. Note that the RHS of the last equality is a constant polynomial, and so we can compute it by applying the localization formula and evaluating the resulting expression at a generic point of $\mathfrak{t}^{\vee}$ which depends only on the $T^{\vee}$-equivariant geometry of $\GPd$. It is then straightforward to see that $|d^{\eta}_{\nu,y}|\leqslant \frac{C\cdot (|\nu|_1+a)^{|\eta|_1}}{(\eta_1)!\cdots (\eta_r)!} R^{|\nu|_1}$ for some constants $a,C,R>1$ that are independent of $\nu$ and $\eta$. Using $\frac{x^m}{m!}<e^x$ for $x>0$, we get
\begin{equation}\label{gammahatlemmaproofeq3}
|d^{\eta}_{\nu,y}|\leqslant C(e^r R)^{|\nu|_1+a}.
\end{equation}

Let us go back to the power series $\sum_{k=1}^{\infty}b_k\left(\sum_{i=1}^{\ell}\delta_i^k\right)$. We have
\[ \int_{\GPd}\sum_{k=1}^{\infty}b_k\left(\sum_{i=1}^{\ell}\delta_i^k\right)\cup y=\sum_{k=1}^{\infty}\sum_{\nu\in N_k}\sum_{\eta\in H_{|\nu|_2+d(y)}} b_kc_{\nu}d^{\eta}_{\nu,y} h_1^{\eta_1}\cdots h_r^{\eta_r}.\]
Using $\frac{|\nu|_2}{\ell}\leqslant |\nu|_1\leqslant |\nu|_2$ and the estimates \eqref{gammahatlemmaproofeq2} and \eqref{gammahatlemmaproofeq3}, we have, for every $h_1,\ldots,h_r\in\mathbb{C}$ satisfying $|h_j|<\epsilon:= \frac{1}{2}\rho (e^r\ell R)^{-1}$ (where $\rho$ satisfies \eqref{gammahatlemmaproofeq1}),
\begin{align*}
&~ \sum_{k=1}^{\infty}\sum_{\nu\in N_k}\sum_{\eta\in H_{|\nu|_2+d(y)}} |b_kc_{\nu}d^{\eta}_{\nu,y} h_1^{\eta_1}\cdots h_r^{\eta_r}|\\
\leqslant &~ C(e^rR)^a\ell\epsilon^{d(y)}\sum_{k=1}^{\infty}|b_k| (e^r\ell R\epsilon)^k\left(\sum_{\nu\in N_k}\sum_{\eta\in H_{|\nu|_2+d(y)}} 1\right).
\end{align*}
Observe that $\sum_{\nu\in N_k}\sum_{\eta\in H_{|\nu|_2+d(y)}}1$ is equal to $|N_k|\cdot|H_{k+d(y)}|$, which is bounded by a polynomial in $k$, and hence the RHS of the last inequality is finite by \eqref{gammahatlemmaproofeq1}. We are done. 
\end{proof}

\bigskip
\begin{proof}[Proof of Lemma \ref{convexhull}]
This is well-known. We provide the details for the reader's convenience. 

First notice that $|\log x|<x+\frac{1}{x}$ for all $x\in\mathbb{R}_{>0}$ so we may assume $c_1=\cdots=c_{\ell}=0$.

Define $g:\mathbb{R}^{\ell}\rightarrow\mathbb{R}$ by 
\[ g(\mathbf{x}):= f(e^{x_1},\ldots,e^{x_{\ell}})=\sum_{\mathbf{v}\in S}f_{\mathbf{v}}e^{\langle \mathbf{x},\mathbf{v}\rangle}\qquad \mathbf{x}=(x_1,\ldots,x_{\ell})\in\mathbb{R}^{\ell}.\]
We claim that the interior of the convex hull $\conv(S)$ of $S$ contains the origin. Suppose not. Then there exists $\mathbf{x}_0\in\mathbb{R}^{\ell}\setminus\{\mathbf{0}\}$ such that $\langle\mathbf{x}_0,\mathbf{v}\rangle\leqslant 0$ for all $\mathbf{v}\in S$. It follows that $\lim_{s\to +\infty} g(s\mathbf{x}_0)$ exists. But by our assumptions, $g$ is convex and has a critical point, and hence it is unbounded at infinity, a contradiction.

Now, by taking the normal fan of $\conv(S)$, we can cover $\mathbb{R}^{\ell}$ with finitely many polyhedral cones such that for each of these cones $C$, there is $\mathbf{v}\in S$ such that the linear function $\mathbf{x}\mapsto\langle\mathbf{x},\mathbf{v}\rangle$ is positive on $C\setminus\{\mathbf{0}\}$. It follows that there exists $M\in\mathbb{R}$ such that 
\[ g(\mathbf{x})\geqslant M+\sum_{k=1}^{\ell}\left( x_k^2+(d_k+1)x_k\right)\qquad \mathbf{x}\in\mathbb{R}^{\ell}.\]
Therefore, 
\begin{align*}
&\int_{\mathbb{R}_{>0}^{\ell}} e^{-f(\mathbf{a})}a_1^{d_1}\cdots a_{\ell}^{d_{\ell}}da_1\cdots da_{\ell}\\
=~ &\int_{\mathbb{R}^{\ell}} e^{-g(\mathbf{x})+\sum_{k=1}^{\ell}(d_k+1)x_k}dx_1\cdots dx_{\ell}\\
\leqslant ~& \int_{\mathbb{R}^{\ell}} e^{-M-|\mathbf{x}|^2}dx_1\cdots dx_{\ell} \\
<~&+\infty.
\end{align*}
\end{proof}
\begin{remark}
    The critical point in question is called the \textit{conifold point} of $f$ \cite{Conifold}.
\end{remark}
\section{Exposition of Lam-Rietsch's theorem}\label{Subsection-LamRietschexposition}
We give an exposition of a result of Lam and Rietsch \cite[Proposition 11.3]{LamRietsch}, which is used in the proof of Lemma \ref{Schubertpositivegoestopositivepart}. Recall Yun-Zhu's isomorphism $\Phi^0_{YZ}$ and Peterson-Lam-Shimozono's homomorphism $\Phi^0_{PLS}$ introduced in Section \ref{Subsection-descriptionofmirrorisom}. Define $U_{\geqslant 0}^-$ to be the submonoid of $U^-$ with unit generated by $y_i(a)$ for $i\in I$ and $a\in\mathbb{R}_{>0}$ ($y_i(a)$ is defined in Section \ref{Subsection-Bmodelnotation}).

\begin{theorem}\label{LamRietschthm}(\cite[Proposition 11.3]{LamRietsch}) Let $q\in \mathbb{R}_{>0}^{I\setminus I_P}$. Suppose $z_q^P$ is an $\mathbb{R}$-point in the scheme $\spec QH^{\bullet}(\GPd)_q$ that is \textit{Schubert positive} in the sense that $\sigma_v(z_q^P)>0$ for all $v\in W^P$. Then $\spec(\Phi^0_{PLS}\circ\Phi^0_{YZ})$ sends $z_q^P$ to a point in $U_{\geqslant 0}^-$.  
\end{theorem}

Before the proof, let us do some preparation.

Let $\widetilde{G}$ denote the universal cover of $G$. Objects associated with $G$ have analogs for $\widetilde{G}$, and we denote them in the obvious way. Define $\BFt:=\{b\in\widetilde{B}^-|~b\cdot F=F\}$. (We may also define $\BF$ in the same way but it is just $\UF$ because $G$ is of adjoint type.) Define 
\[ U^-_{>0}:= U^-_{\geqslant 0}\cap Bw_0B\quad\text{ and }\quad \widetilde{U}^-_{>0}:= \widetilde{U}^-_{\geqslant 0}\cap \widetilde{B}w_0\widetilde{B} .\]
Let $\{\omega_i\}_{i\in I}$ be the set of fundamental weights. Define $\Gamma :=W\cdot \{\omega_i\}_{i\in I}$, regarded as a subset of the character lattice of $\widetilde{T}$. Define a collection $\{\Delta^{\gamma}\}_{\gamma\in\Gamma}$ of regular functions on $\widetilde{G}$ as follows. For $i\in I$, denote by $V(\omega_i)$ the $i$-th fundamental representation of $\widetilde{G}$. Pick a non-zero highest weight vector $v_i\in V(\omega_i)$, and let $v_i^*\in V(\omega_i)^*$ be the unique element such that $\langle v_i^*,v_i\rangle = 1$ and $v_i^*$ vanishes on other weight vectors. Define $\Delta^{\omega_i}\in\mathcal{O}(\widetilde{G})$ by $\Delta^{\omega_i}(g):=\langle v_i^*,g^T\cdot v_i\rangle$, where $g\mapsto g^T$ is the transpose of $\widetilde{G}$, i.e., the unique anti-automorphism of $\widetilde{G}$ characterized by 
\[ x_j(a)^T=y_j(a),\quad t^T=t\quad \text{and}\quad y_j(a)^T=x_j(a)\]
for $j\in I$, $a\in \Ga$ and $t\in \widetilde{T}$. Now for arbitrary $\gamma\in\Gamma$, we can find $w\in W$ such that $w^{-1}\gamma=\omega_i$ for some $i\in I$. Note that $i$ is unique. Define $\Delta^{\gamma}\in\mathcal{O}(\widetilde{G})$ by $\Delta^{\gamma}(g):=\langle v_i^*, g^T\cdot v_{\gamma}\rangle$, where $v_{\gamma}$ is an element of the $\gamma$-weight space $V(\omega_i)_{\gamma}$ that we will specify in Remark \ref{LamRietschthmremark_finish_def} below.

Let $G^{\vee}_{ad}$ be the quotient of $G^{\vee}$ by its center and $\ag_{ad}$ the corresponding affine Grassmannian. It is known that the Pontryagin ring $H_{-\bullet}(\ag_{ad})$ has an additive basis $\{\xi_{\alpha}\}_{\alpha\in A_{ad}}$ consisting of \textit{affine Schubert classes}. There is a subset $A\subseteq A_{ad}$ such that $\{\xi_{\alpha}\}_{\alpha\in A}$ is a basis of $H_{-\bullet}(\ag)$, and there is a distinguished element $0\in A$ such that $\xi_{0}=1$.

There are canonical isomorphisms of rings 
\begin{equation}\label{LamRietschthmeq1}
\mathcal{O}(\BFt)\simeq \mathcal{O}(\UF)\otimes \mathcal{O}(Z(\widetilde{G}))\quad\text{ and }\quad H_{-\bullet}(\ag_{ad})\simeq H_{-\bullet}(\ag)\otimes  H_0(\ag_{ad}).
\end{equation}
Note that $G^{\vee}_{ad}$ is Langlands dual to $\widetilde{G}$, and hence we have the corresponding Yun-Zhu's isomorphism
\[ \widetilde{\Phi}^0_{YZ}:\mathcal{O}(\BFt) \xrightarrow{\sim} H_{-\bullet}(\ag_{ad}).\]
We collect below some facts about $\Phi^0_{YZ}$, $\widetilde{\Phi}^0_{YZ}$ and $\Phi^0_{PLS}$, as well as some others that we will need for the proof of Theorem \ref{LamRietschthm}.

\begin{enumerate}
\item (\cite[Theorem 1.1]{YunZhu}) $\Phi^0_{YZ}$ and $\widetilde{\Phi}^0_{YZ}$ are graded Hopf algebra isomorphisms, where the gradings on the sources are induced by the conjugation of the cocharacter $-2\rho^{\vee}:=-\sum_{\alpha\in R^+}\alpha^{\vee}$, and the coalgebra structures on the sources and targets are induced by the group multiplications and the homology coproducts respectively. 

\item (\cite[Proposition 3.3]{YunZhu}) After composing the isomorphisms from \eqref{LamRietschthmeq1}, we have
\[ \widetilde{\Phi}^0_{YZ} = \Phi^0_{YZ}\otimes\phi,\]
where $\phi:\mathcal{O}(Z(\widetilde{G}))\xrightarrow{\sim} H_0(\ag_{ad})$ is the canonical isomorphism. (Both group schemes $Z(\widetilde{G})$ and $\spec H_0(\ag_{ad})$ are canonically isomorphic to the quotient of the coweight lattice by the coroot lattice.)

\item (Remark \ref{LamRietschthmremark_Fact3} below) $\widetilde{\Phi}^0_{YZ}$ sends each $\Delta^{\gamma}|_{\BFt}$ to the fundamental class of a closed irreducible subvariety of $\ag_{ad}$.

\item (\cite[Theorem 10.21]{LamShimozono}) $\Phi^0_{PLS}$ is graded and sends every affine Schubert class $\xi_{\alpha}$ to either zero or $\left(\prod_{i\in I\setminus I_P} q_i^{d_i}\right)\sigma_v$ for some $(d_i)\in\mathbb{Z}^{I\setminus I_P}$ and $v\in W^P$.

\item (\cite[Theorem 9.2]{LamShimozono}) $\Phi^0_{PLS}$ is injective for $P=B$.

\item (\cite[Proposition 4.2]{Lusztig_positivity}) $U^-_{\geqslant 0}$ is closed in $U^-$ in the classical topology.

\item (\cite[Theorem 1.5]{BZSchubert}\footnote{More precisely, the result stated here follows from the cited one by putting $w=w_0$ and applying the transpose $g\mapsto g^T$.}) An element $x\in \widetilde{U}^-$ lies in $ \widetilde{U}^-_{>0}$ if and only if $\Delta^{\gamma}(x)>0$ for all $\gamma\in\Gamma$\footnote{In the statement of \cite[Theorem 1.5]{BZSchubert}, an additional assumption $x\in\widetilde{B}w^{-1}\widetilde{B}$ is imposed but in our case $w=w_0$, it follows automatically from the positivity condition.}.

\item (\cite[Proposition 5 \& Lemma 9]{KumarNori}) The fundamental class of any closed irreducible subvariety of $\ag_{ad}$ (resp. $\ag$) is equal to a non-zero linear combination of $\{\xi_{\alpha}\}_{\alpha\in A_{ad}}$ (resp. $\{\xi_{\alpha}\}_{\alpha\in A}$) with positive coefficients.

\item (Fact 8 applied to $\ag_{ad}\times\ag_{ad}$) The homology coproduct
\[\Delta:H_{-\bullet}(\ag_{ad})\rightarrow H_{-\bullet}(\ag_{ad})\otimes H_{-\bullet}(\ag_{ad})\]
satisfies 
\[ \Delta(\xi_{\alpha})= \xi_{\alpha}\otimes 1 + 1\otimes \xi_{\alpha} + \sum_{\substack{(\alpha_1,\alpha_2)\\ \ne (0,\alpha), (\alpha,0) }} c_{\alpha_1,\alpha_2}^{\alpha} \xi_{\alpha_1}\otimes \xi_{\alpha_2} \]
for all $\alpha\in A_{ad}$, where each $c_{\alpha_1,\alpha_2}^{\alpha}$ is non-negative.
\end{enumerate} 

\bigskip
\begin{proof}[Proof of Theorem \ref{LamRietschthm}]
Define $y_q^P:=(\spec(\Phi^0_{PLS}))(z_q^P)$ and $x_q^P:=(\spec(\Phi^0_{YZ}))(y_q^P)$. We have to show $x_q^P\in U_{\geqslant 0}^-$. Pick a point $q_0\in \mathbb{R}_{>0}^I$. By Proposition \ref{existSchubertpositive}, we have a Schubert positive $\mathbb{R}$-point $z_{q_0}^B\in\spec QH^{\bullet}(\GBd)_{q_0}$. Consider the $\mathbb{G}_m$-action on $\spec QH^{\bullet}(\GBd)$ induced by the grading introduced in Section \ref{Subsection-Amodelflagvariety}. Since every Schubert class is homogeneous, we obtain, by applying the action $s\cdot -$, a Schubert positive $\mathbb{R}$-point $z_{s\cdot q_0}^B\in \spec QH^{\bullet}(\GBd)_{s\cdot q_0}$ for each $s\in\mathbb{R}_{>0}$. (Notice that $s\cdot q_0$ is obtained from $q_0$ by multiplying each component by $s^{-4}$.) Define $y_{s\cdot q_0}^B$ and $x_{s\cdot q_0}^B$ similarly. Since $\Phi^0_{YZ}$ (resp. $\Phi^0_{PLS}$) is graded by Fact 1 (resp. Fact 4), we have $x_{s\cdot q_0}^B=(2\rho^{\vee})(s^{-1})x_{q_0}^B(2\rho^{\vee})(s)$, and hence $\lim_{s\to 0^+}x_{s\cdot q_0}^B=e$. Since $U^-_{\geqslant 0}$ is closed in $U^-$ (Fact 6), the proof will be complete if we can show $x(s):=x_{s\cdot q_0}^B\cdot x_q^P\in U^-_{>0}$ for all $s\in\mathbb{R}_{>0}$.

For a point $x\in \UF$ (resp. $y\in\spec H_{-\bullet}(\ag)$), define $\widetilde{x}:=(x,e)\in\BFt$ (resp. $\widetilde{y}:=(y,e)\in\spec H_{-\bullet}(\ag_{ad})$) using the first (resp. second) isomorphism from \eqref{LamRietschthmeq1}. We have $(\spec(\widetilde{\Phi}^0_{YZ}))(\widetilde{y})=\widetilde{x}$ whenever $\left(\spec\left(\Phi^0_{YZ}\right)\right)(y)=x$, by Fact 2. Observe that the projection $\widetilde{G}\rightarrow G$ restricts to an isomorphism $\widetilde{U}^-_{>0}\xrightarrow{\sim} U^-_{>0}$. Hence, it suffices to show $\widetilde{x(s)}\in\widetilde{U}^-_{>0}$. Clearly, we have $\widetilde{x(s)}=\widetilde{x_{s\cdot q_0}^B}\cdot\widetilde{x_q^P}$. By Fact 7, it suffices to show $\Delta^{\gamma}\left(\widetilde{x_{s\cdot q_0}^B}\cdot\widetilde{x_q^P}\right)>0$ for all $\gamma\in \Gamma$. Since $\widetilde{\Phi}^0_{YZ}$ preserves the coalgebra structures (Fact 1), we have 
\[ \Delta^{\gamma}\left(\widetilde{x_{s\cdot q_0}^B}\cdot\widetilde{x_q^P}\right) = \left(\Delta\left(\widetilde{\Phi}^0_{YZ}(\Delta^{\gamma}|_{\BFt})\right)\right)\left(\widetilde{y_{s\cdot q_0}^B},\widetilde{y_q^P}\right), \] 
where $\Delta:H_{-\bullet}(\ag_{ad}) \rightarrow H_{-\bullet}(\ag_{ad})\otimes H_{-\bullet}(\ag_{ad})$ is the homology coproduct. By Fact 3, $\widetilde{\Phi}^0_{YZ}(\Delta^{\gamma}|_{\BFt})=[C]$ for some closed irreducible subvariety $C$ of $\ag_{ad}$. By Fact 8, $[C]$ is equal to a non-zero linear combination of affine Schubert classes $\xi_{\alpha}$ with positive coefficients. It follows that, by Fact 9, we are done if we can show that for every $\alpha\in A_{ad}$,
\[ \xi_{\alpha}\left(\widetilde{y_{s\cdot q_0}^B}\right)>0\quad\text{ and }\quad \xi_{\alpha}\left(\widetilde{y_q^P}\right)\geqslant 0.\]

Observe that for every $\alpha\in A_{ad}$ and $y\in\spec H_{-\bullet}(\ag)$, there exists $p\in\ag_{ad}$ such that $[p]^{-1}\bulletsmall \xi_{\alpha}\in H_{-\bullet}(\ag)$ and $\xi_{\alpha}\left(\widetilde{y}\right)=\left([p]^{-1}\bulletsmall \xi_{\alpha}\right)(y)$. Notice that $[p]^{-1}\bulletsmall \xi_{\alpha}$ is the fundamental class of a closed irreducible subvariety of $\ag$. It follows that, by Fact 8, it suffices to show that for every $\alpha\in A$,
\[ \xi_{\alpha}\left(y_{s\cdot q_0}^B\right)>0\quad\text{ and }\quad \xi_{\alpha}\left(y_q^P\right)\geqslant 0.\]
By definition, these numbers are equal to $\left(\Phi^0_{PLS}(\xi_{\alpha})\right)(z_{s\cdot q_0}^B)$ and $\left(\Phi^0_{PLS}(\xi_{\alpha})\right)(z_q^P)$ respectively. The inequalities $\geqslant 0$ follow from Fact 4 and the Schubert positivity assumptions on $z_{s\cdot q_0}^B$ and $z_q^P$, and the strict inequality $>0$ for $y_{s\cdot q_0}^B$ follows in addition from Fact 5. The proof is complete.
\end{proof}

\begin{remark}\label{LamRietschthmremark_finish_def}
Let us finish the definition of $\Delta^{\gamma}$ ($\gamma=w\omega_i$) by specifying the weight vector $v_{\gamma}\in V(\omega_i)_{\gamma}$. Take a reduced decomposition $\mathbf{j}=(j_1,\ldots,j_m)$ of $w^{-1}$. For $k=1,\ldots, m$, define $b_k:=\langle \omega_i, s_{j_1}\cdots s_{j_{k-1}}(\alpha^{\vee}_{j_k})\rangle$. Then $v_{\gamma}$ is defined to be the unique vector satisfying
\[ \frac{e_{j_1}^{b_1}\cdots e_{j_m}^{b_m}}{(j_1)!\cdots (j_m)!}\cdot v_{\gamma} = v_i\]
(recall $e_j$ is fixed in Section \ref{Subsection-Bmodelnotation}). 
\end{remark}
\begin{remark}\label{LamRietschthmremark_Fact3}
We prove Fact 3 as follows. Write $\gamma=w\omega_i$. By the geometric Satake equivalence \cite{Ginzburg, Lusztig, MV}, $V(\omega_i)$ has a basis, called \textit{MV basis}, consisting of weight vectors that are indexed by a collection of closed irreducible subvarieties of $\ag_{ad}$, called \textit{MV cycles}. Rescale $v_i$ and $v_i^*$ simultaneously such that $\langle v_i^*,v_i\rangle =1$ continues to hold (so that $\Delta^{\gamma}$ is unchanged) and $v_i$ becomes an element of the MV basis. Let $v'_{\gamma}\in V(\omega_i)_{\gamma}$ be the unique element belonging to the MV basis. By \cite[Lemma 10.5]{Acta}, $\widetilde{\Phi}^0_{YZ}$ sends the regular function $b\mapsto \langle v_i^*, b^T\cdot v'_{\gamma}\rangle$ (defined on $\BFt$) to the fundamental class of an MV cycle. Hence it suffices to show $v'_{\gamma}=v_{\gamma}$ (see Remark \ref{LamRietschthmremark_finish_def} for the definition of $v_{\gamma}$), or equivalently
\[ \frac{e_{j_1}^{b_1}\cdots e_{j_m}^{b_m}}{(j_1)!\cdots (j_m)!}\cdot v'_{\gamma} = v_i.\]
The last equality follows from a main result of \cite{Acta} that the MV basis is \textit{perfect}. See Theorem 5.2 therein or specifically Theorem 5.4 and Proposition 5.5.
\end{remark}

\end{document}